\documentclass[12pt,reqno]{amsart}
\usepackage{amsthm,amsfonts,amssymb,euscript}
\usepackage[all,cmtip]{xy}

\let\oldtocsection=\tocsection

\let\oldtocsubsection=\tocsubsection

\let\oldtocsubsubsection=\tocsubsubsection

\renewcommand{\tocsection}[2]{\hspace{0em}\oldtocsection{#1}{#2}}
\renewcommand{\tocsubsection}[2]{\hspace{1em}\oldtocsubsection{#1}{#2}}
\renewcommand{\tocsubsubsection}[2]{\hspace{2em}\oldtocsubsubsection{#1}{#2}}

\def\ddb#1{\sqrt{-1}\partial\bar{\partial}#1}
\def\dt#1{\frac{\partial}{\partial t}#1}
\usepackage{xcolor}

\newcommand{\dr}{\omega}
\newcommand{\al}{\alpha}
\newcommand{\de}{\delta}
\newcommand{\e}{\epsilon}
\newcommand{\ka}{K\"{a}hler}
\newcommand{\MA}{Monge-Amp\`{e}re }

\theoremstyle{plain}
\newtheorem{theorem}{Theorem}[section]

\newtheorem{lemma}[theorem]{Lemma}
\newtheorem{corollary}[theorem]{Corollary}
\newtheorem{assumption}[theorem]{Assumption}

\newtheorem{remark}[theorem]{Remark}

\theoremstyle{definition}
\newtheorem{definition}[theorem]{Definition}

\numberwithin{equation}{section}

\begin{document}
	\title[K\"{a}hler-Ricci flow]{The K\"{a}hler-Ricci flow with log canonical singularities}
	\author{Albert Chau$^1$}
	\author{Huabin Ge$^2$}
	\author{Ka-Fai Li}
	\author{Liangming Shen$^3$}
	\address{Department of Mathematics, The University of British Columbia, Vancouver, B.C., Canada V6T 1Z2}
	\email{chau@math.ubc.ca}
	\address{School of Mathematics, Renmin University of China, Beijing, 100872, P.R. China }
	\email{hbge@bjtu.edu.cn}
	\address{Morningside Center of Mathematics, Academy of Mathematics and Systems Science, Chinese Academy of Sciences, Beijing, 100190, P.R. China}
	\email{kfli@amss.ac.cn}
	\address{School of Mathematical Sciences, Beihang University, and Key Laboratory of Mathematics Informatics Behavioral Semantics, Ministry of Education, Beijing 100191, P.R. China}
	\email{lmshen@buaa.edu.cn}
	\thanks{$^1$ Partially supported by NSERC Grant \#327637-06.}
	\thanks{$^2$ Partially supported by NSFC Grant \#11871094.}
	\thanks{$^3$ Partially supported by NSFC Grant \#11901553.}

	\begin{abstract}
		We establish the existence of the \ka-Ricci flow on projective varieties with log canonical singularities.  This generalizes some of the existence results of Song-Tian \cite{ST3} in case of projective varieties with klt singularities. We also prove that the normalized \ka-Ricci flow will converge to the \ka-Einstein metric with negative Ricci curvature on semi-log canonical models in the sense of currents. Finally we also construct \ka-Ricci flow solutions performing divisorial contractions and flips with log canonical singularities.
	\end{abstract}
	
	\maketitle
	
	\tableofcontents
	

	\section{Introduction}\label{section 1}
	
	In the decades of 1980 and 1990, Mori first proposed the minimal model program for high dimensional algebraic varieties, which has become an active field in algebraic geometry (cf.\cite{KMM,KM}). The main target of this program is to give a complete birational classification of algebraic varieties according to the birational classification of their minimal models. A variety $X$ is called minimal if its canonical line bundle $K_{X}$ is nef, i.e., it holds that $K_{X}\cdot C\geq 0$ for any algebraic curve $C.$  An important procedure in this program is to perform successive  birational surgeries, such as blow-downs and flips, to an algebraic variety until the so called minimal model is reached. Later, an essential breakthrough by Birkar-Cascini-Hacon-McKernan in \cite{BCHM} asserts the minimal model is indeed attained by this procedure for varieties with klt singularities.
	
	On the other hand, finding a canonical metric on a \ka\, manifold or a variety has long been a central problem in \ka\ geometry.
	Since Yau's solution to Calabi's conjecture \cite{Yau1} there have been a lot of developments in this direction. When the first Chern class $c_{1}(X)$ has definite sign the canonical metrics are \ka-Einstein metrics which have been studied systematically.  In particular, the \ka-Ricci flow can also be used to study the canonical metrics. In \cite{Cao} Cao gave a parabolic proof of existence of negative and zero \ka, metrics and showed that starting from some $\dr_0$ (belonging to $-c_1(X)$ in the negative case), the \ka-Ricci flow
	
	\begin{equation}\label{eq:krf}
		\left\{
		\begin{array}{ll}
			\displaystyle\dt\dr &=-Ric(\dr)\\
			\dr(0) &= \dr_0,
		\end{array}
		\right.
	\end{equation}
	has a longtime solution converging (after normalizing in the negative case) to the \ka-Einstein metric on $X$.
	
	However, in general the first Chern class will not be zero or have definite sign, and in these general situations the minimal model program provides ideas for finding so-called generalized \ka-Einstein metrics as canonical metrics on these varieties. In \cite{Ts} Tsuji used the \ka-Ricci flow to study the existence of the generalized \ka-Einstein metric on minimal projective manifolds of general type. In \cite{TZ} Tian-Zhang established the general existence result of the \ka-Ricci flow. In particular they gave a general existence time criterion of the \ka-Ricci flow on \ka\ manifolds and studied the long time behavior on projective manifolds. Furthermore, in a sequence of works \cite{ST1,ST2,ST3}, Song-Tian initiated the study of analytic minimal model program which proposes to find the canonical metrics on general projective varieties via the \ka-Ricci flow. They gave more precise descriptions of the long time behavior of the \ka-Ricci flow on projective manifolds and began the study of the \ka-Ricci flow on singular projective varieties.  There have been several related works on the \ka-Ricci flow in such singular settings, from a pluripotential theoretic point of view, and we refer to the recent works \cite{GLZ, D} and the references therein. 
	
	Let us briefly recall the minimal model program of Mori and discuss some corresponding works on the analytic aspect of this program.
	By the classical theory \cite{KMM,KM}, algebraic varieties can be classified by their Kodaira dimensions, where the Kodaira dimension of a n-dimenstional variety $X$ is defined by
	$$\kappa(X)=Kod(X):=\sup\{\kappa|\lim\inf_{l\to +\infty}\frac{\dim H^{0}(X,lK_{X})}{l^{\kappa}}>0\}.$$

	When $X$ is not minimal ($K_{X}$ is not nef.), by the cone theorem and base point free theorem there exists a contraction map $\varphi: X\to X'$ determined by the extremal ray which has negative intersection number with $K_{X}.$  If $\dim X'<n$, then $\varphi$ is of fiber type and $X$ is a uniruled Mori fiber space which implies $\kappa(X)=-\infty$ (\cite{KMM,KM}) and we can continue to consider the structure of the lower dimensional variety $X'.$ When $\dim X'=n,$ we can consider the exceptional set $Exc(\varphi)$ of $\varphi,$ which is the complement of the set in $X$ where $\varphi$ is an isomorphism. If $Exc(\varphi)$ has codimension 1, $\varphi$ is a divisorial contraction such that the Picard number $\rho(X')=\rho(X)-1.$ If $Exc(\varphi)$ has codimension 1 greater than 1, it was conjectured that there exists a flip morphism $X\to X^{+}$ with associated  $\varphi^{+}:X^{+}\to X'$ such that $-K_{X}$ is $\varphi$-ample while $K_{X^{+}}$ is $\varphi^{+}$-ample. The existence of the flip on normal varieties with klt singularities was confirmed in \cite{BCHM}. In these contexts, in \cite{ST3}, Song-Tian first defined the weak \ka-Ricci flow on $\mathbb{Q}$-factorial projective varieties with klt singularities and generalized Tian-Zhang's maximal time existence of the flow solution (\cite{TZ}) to singular settings. In particular, it was determined that when $X$ is not minimal ($K_{X}$ is not nef) then in finite time,  the flow will encounter an analytic singularity corresponding to algebraic surgeries characterized by $\varphi$, and that the \ka-Ricci flow could be extended through these singularities in the sense of currents.  Furthermore, if the surgeries at the singular time are divisorial contractions, Song-Weinkove  in \cite{SW1,SW2} proved the geometric convergence of the flow solution to the variety generated by the divisorial contractions. Song and Yuan also found several examples of metric flips corresponding to algebraic flips in \cite{So13,SY1}.

	For divisorial contractions we have $\rho(X')=\rho(X)-1$, so there can be at most finitely many divisorial contractions encountered by successive contractions of a non-minimal variety $X$ as above before a minimal variety is reached. In general however, it is still unknown whether the number of flips encountered is necessarily finite or not.  On certain varieties with dimension 3 or 4 it is known that only finitely many flips are encountered \cite{KMM}.  We will assume that only finitely many successive contractions of $X$ are required before reaching a minimal variety which we still denote by $X.$ We will also assume that the abundance conjecture is true, which asserts that if $K_{X}$ is nef then it is semi-ample, i.e., the canonical ring $\text{R}:=\bigoplus_{l\geq 0}H^{0}(X,lK_{X})$ is finitely generated. Thus there exists a natural morphism $\Phi_{|lK_{X}|}:X\to X_{can}=\text{Proj R}$ which is called the canonical model of $X.$  We then consider three separate cases in terms of the Kodaira dimension $\kappa(X)$ of $X$: $\kappa=0$; $0<\kappa<n$; $\kappa=n$.
	
	If $\kappa(X)=0$ then the abundance conjecture holds while $K_{X}$ is numerically trivial. In this case, if $X$ is smooth Cao  \cite{Cao}  proved the \ka-Ricci flow smoothly converges to the Calabi-Yau metric on $X$, while if $X$ has klt singularities Song-Yuan  \cite{SY2} proved that the \ka-Ricci flow converges to the singular Calabi-Yau metric in the current sense.
	
	If $\kappa(X)=n$ i.e., $K_{X}$ is big or $X$ is of general type, then the abundance conjecture also holds, and the morphism $\Phi_{|lK_{X}|}$ is a birational morphism from $X$ to $X_{can}.$ By \cite{TZ} if $X$ is nonsingular, the normalized \ka-Ricci flow with any initial \ka\ metric $\dr_{0}$
	\begin{equation}\label{eq:nor-krf}
		\left\{
		\begin{array}{ll}
			\displaystyle\dt\tilde{\dr} &=-Ric(\tilde{\dr})-\tilde{\dr} \\
			\tilde{\dr}(0) &= \tilde{\dr}_{0},
		\end{array}
		\right.
	\end{equation}
	has a unique solution on $[0,+\infty)$ and converges to the generalized \ka-Einstein metric on $X_{can}$ in current sense, where the generalized \ka-Einstein metric is a smooth \ka-Einstein metric on the regular part of $X_{can}.$ In case that $X$ is of general type, with klt singularities, but not minimal, in \cite{ST3} Song-Tian proved that given a suitable initial metric, one can continue the flow through its singularities, and that if only finitely many such are encountered, the weak flow \eqref{eq:nor-krf} will go through the surgeries and finally weakly converges to the generalized \ka-Einstein metric on $X_{can}.$

	If $0<\kappa(X)<n,$ assuming the abundance conjecture holds, the morphism $\Phi_{|lK_{X}|}$ will induce a fibration of $X$
	over a minimal $\kappa(X)$-dimensional variety $X_{can}:=\text{Proj R}$ where the canonical class $K_{X_{\eta}}$ of the generic fiber $X_{\eta}$ is numerically trivial. As a special case, in \cite{ST1} Song-Tian considered the case when $n=2,\kappa(X)=1,$ i.e., when $X$ is a nonsingular surface as the torus fibration over an elliptic curve which could be thought as $X_{can}.$ They proved that given any initial metric $\dr_{0}$, the normalized \ka-Ricci flow \eqref{eq:nor-krf} has a unique solution on $[0,+\infty)$ and converges to the generalized \ka-Einstein metric on the regular part of $X_{can}$ in $C^{1,1}$-sense, where the generalized \ka-Einstein metric $\dr_{GKE}$ on the regular part of $X_{can}$ satisfies
	$$Ric(\dr_{GKE})=-\dr_{GKE}+\dr_{WP}$$
	with a Weil-Petersson metric $\dr_{WP}$ which is generated by the deformation of the torus fibers. Furthermore, in \cite{ST2} they generalized their work to any n-dimensional nonsingular projective varieties with any $\kappa(X)\in (0,n)$ and one difference is that the convergence will only be in the current sense. If $\kappa(X)=1,$ recently in \cite{TZL} Tian-Z.L.Zhang proved that the convergence is in fact in the Gromov-Hausdorff sense.
	
	In summary, we observe that even starting from a nonsingular variety, singularities for the \ka\, Ricci flow may still develop in finite time at which point algebraic surgeries are encountered. To study the analytic minimal model program, we need to define the \ka-Ricci flow on singular varieties. As in \cite{BBEGZ,EGZ,ST3}, a reasonable background metric must first be defined on such singular varieties.  On a normal variety $X$, reference to a local metric can always be made, as in as in \cite{BBEGZ}, as any neighbourhood can always be considered as an analytic set in an ambient complex Euclidean space and a restriction of an ambient \ka\ metric can be considered.  On a $\mathbb{Q}$-factorial projective variety $X$ with a big and semi-ample divisor $H$, a global metric always exists as in \cite{ST3}, since we have a birational morphism $\Phi_{|mH|}: X\to\mathbb{CP}^{N_{m}}$ for sufficiently large integer $m$, which in turn induces a current $\dr_{0}:=\frac{1}{m}\Phi_{|mH|}^{*}\dr_{FS}\in [H]$ on $X$ where $\dr_{FS}$ is the canonical Fubini-Study metric on $\mathbb{CP}^{N_{m}}.$ As $H$ is big and semi-ample, by the properties of birational morphism, $\dr_{0}$ is a positive current, smooth and non-degenerate in a Zariski open subset of $X.$ So we can think of $\dr_{0}$ as a metric on the projective variety $X$, and we can define the class of $\dr_{0}$-PSH functions on $X:$
	$$PSH(X,\dr_{0}):=\{\varphi\in [-\infty,+\infty)|\dr_{0}+\ddb\varphi\geq 0\}.$$  To define the \ka-Ricci flow of $\dr_{0}+\ddb\varphi$ on the possibly singular variety $X$ for $\varphi\in PSH(X,\dr_{0})$, we must pull back and work on a smooth resolution of $X$ as follows.  By Hironaka's resolution (cf. \cite{KMM,KM}) we have a nonsingular projective variety $X'$ with a birational morphism $$\pi:X'\to X$$ where the canonical classes of $X$ and $X'$ are related by the adjunction formula:
	\begin{equation}\label{eq:adjunction 1}
		K_{X'}=\pi^{*}K_{X}+\sum_{j}a_{j}E_{j},
	\end{equation}
	where $E_{j}$ is an exceptional divisor and $a_{j}$ is the corresponding discrepancy.  We may then consider the \ka-Ricci flow  \eqref{eq:krf} on $X'$ of the possibly degenerate pullback metric $\pi^{*}(\dr_{0}+\ddb\varphi)$ on $X'.$  A solution $\omega(t)$ to this could be then considered as a solution to   \eqref{eq:krf} on $X$ provided $\omega(t)$ restricts to be zero along the fibers of $\pi$, and thus ``descends" down to $X$.
	
	The above study leads in general to degenerate elliptic or parabolic complex \MA\ equations on the non-singular variety $X'.$ In \cite{EGZ} Eyssidieux-Guedj-Zeriahi generalized Kolodziej's $L^{\infty}$-estimate \cite{Ko} for complex \MA\ equation to the degenerate case and established the existence of singular \ka-Einstein metrics with zero or negative Ricci curvature on $\mathbb{Q}$-factorial projective varieties with klt singularities. In \cite{ST3} Song-Tian also made use of this crucial estimate to establish the existence of the solutions to the \ka-Ricci flow on $\mathbb{Q}$-factorial projective varieties with klt singularities. A critical point of those works is that the klt singularities, where for any exceptional divisor $E_{j}$ it holds that $a_{j}>-1,$ only result in a $L^{p}$-integrable volume form in the degenerate \MA\ equation on $X'$ for some $p>1$ where the crucial $L^{\infty}$-estimate of the potential holds in \cite{Ko}. However we will see that this integrability fails in case of log canonical singularities, where $a_{j}\geq-1.$
	
	As indicated in \cite{KMM,KM}, in the minimal model program we are concerned mainly with singular varieties with at worst log canonical singularities, as classification according to discrepancies is invariant under different resolutions, i.e., the properties that all $a_{j}>-1$ and $a_{j}\geq-1$ are independent of the resolutions.  Following \cite{BCHM} on varieties with klt singularities, Birkar \cite{Bir}, Hacon-Xu \cite{HX} and Fujino \cite{Fu1,Fu2} established the minimal model theory for log canonical pairs. In particular they established the existence of birational surgeries including blow-downs and flips for log canonical pairs. In the analytic aspect, by \cite{BBEGZ}, $\mathbb{Q}$-Fano varieties admit at worst klt singularities so log canonical singularities cannot appear on $\mathbb{Q}$-Fano varieties. In \cite{BG}, Berman-Guenancia proved the existence of a Kahler Einstein metric with negative Ricci curvature on projective varieties with semi-log canonical singularities and ample canonical bundle. Here the semi-log canonical means that the twisted canonical class $K_{X}+D$ is $\mathbb{Q}$-Cartier ample, $X$ has only ordinary nodes with codimension 1, and any resolution of this pair satisfies the log canonical condition. As in the log canonical case, the $L^{\infty}$-estimate in \cite{EGZ} does not hold, and they used instead the variational method developed by \cite{BBGZ} to establish the existence of a weak solution. In \cite{So17}, Song also derived the existence of the \ka-Einstein metric on the semi-log canonical pairs by purely PDE methods where he also proved that the semi-log canonical model can be the limit in the moduli space of negative \ka-Einstein metrics.
	
	In this work, we will generalize the results on existence of solutions to the \ka-Ricci flow in \cite{ST3} to the case of $\mathbb{Q}$-factorial projective varieties with log canonical singularities. Our first result is in the following:
	\begin{theorem}\label{thm-main1}
		Let $X$ be a $\mathbb{Q}$-factorial projective variety with log canonical singularities and $H$ be a big and semi-ample $\mathbb{Q}$-Cartier divisor on $X.$
		Suppose $\Phi_{|mH|}$ defines a birational morphism $X\to\mathbb{CP}^{N_{m}}$ for some large integer $m$ and $\dr_{0}=:\frac{1}{m}\Phi_{|mH|}^{*}\dr_{FS}\in[H]$ is semi-positive current on $X.$ Then given the initial potential function $\varphi_{0}\in PSH(X,\dr_{0})\bigcap L^{\infty}_{loc}(X\setminus X_{lc})$ such that $\dr_{0}+\ddb\varphi_{0}$ is a current with zero Lelong number, there exists a unique maximal weak solution $\dr(t)$ to the \ka-Ricci flow  \eqref{eq:krf}:
		\begin{equation}\nonumber
			\left\{
			\begin{array}{ll}
				\displaystyle\dt\dr &=-Ric(\dr)\\
				\dr(0) &= \dr_0+\ddb\varphi_{0},
			\end{array}
			\right.
		\end{equation}
		on the time interval $X\times [0,T_{0}),$ where $X_{lc}:=\pi(\bigcup_{a_{i=-1}}E_{i})$ and $T_{0}:=\sup\{t>0|H+tK_{X}\;is\; nef\}.$ In particular, $\dr(t)$ is a current on $X$ with zero Lelong number for all $t\in [0,T_{0})$, solves \eqref{eq:krf} smoothly on $X_{reg}\times(0,T_{0})$ and converges to the initial current $\dr(0)$ in the current sense.
	\end{theorem}
	\begin{remark}
		We refer to Definition \ref{logcanonical} for the definition of the $log$ $canonical$ $locus$ $X_{lc}$ and to Definition \ref{weaksolkrf} for the precise definition of weak and maximal solutions to \eqref{eq:krf} in the context of Theorem \ref{thm-main1} above.
	\end{remark}
	
	Let us briefly describe the strategy of the proof. As in \cite{BBEGZ,ST3} and described above we pull back and study the \ka-Ricci flow on the resolution $X'$ of $X$ and derive a degenerate complex \MA\ flow equation for a family of singular potentials $\varphi(t)$ on $X'.$ As indicated in \cite{EGZ}, we cannot restrict the initial potential $\varphi(0)$  to be in $L^{\infty}(X')$, and more generally must consider currents of zero Lelong number.  The main difficulty here arises from the exceptional divisor on $X'$ with discrepancy $-1,$ precluding a treated as in \cite{EGZ,ST3}.  We overcome this difficulty as in  \cite{CLS, De,GZ} by first regularizing $\varphi(0)$ by a family of smooth bounded potentials on $X'$, and also by regularizing the degenerate \MA flow equation on $X'$ while introducing a new family of background forms which combine Carlson-Griffiths forms \cite{CG} and the regularized conical forms by Guenancia-P\u{a}un \cite{GP} on $X'$.  In particular, these background forms will provide a family smooth complete bounded curvature \ka\, metrics in the open compliment of some divisor on $X'$.  We will then derive a uniform upper bound for solutions of the regularized equation as well as local lower bounds in the regular region after which we establish successive local high order estimates.  For the unique maximality of solutions, we will adapt the arguments in \cite{GZ} to prove the continuity of the solution at time zero in $L^{1}$-sense which will imply the maximality of the weak solution to the weak \ka-Ricci flow \eqref{eq:krf}.
	
	By Theorem \ref{thm-main1}, when $X$ is a semi-log canonical model $X$ so that $K_{X}$ is ample $\mathbb{Q}$-Cartier, the solution $\dr(t)$ to the \ka-Ricci flow  \eqref{eq:krf} exists for $t\in [0,+\infty).$  In fact, the normalized \ka-Ricci flow \eqref{eq:nor-krf} will also have a longtime solution in this case as well and natural problem is to study the limit behaviour of the normalized flow. From \cite{Chau,LZ}, in the complete smooth case the normalized \ka-Ricci flow converges to the complete \ka-Einstein metric with negative Ricci curvature. In the next theorem we show that similar limit behaviour is also true for such semi-log canonical models:
	\begin{theorem}\label{thm-main2}
		Suppose that in Theorem  \ref{thm-main1} we have that $X$ is a semi-log canonical model.  Then, the normalized \ka-Ricci flow
		\eqref{eq:nor-krf} has a unique maximal weak solution on $[0,+\infty)$ with initial condition $\dr(0)=\dr_{0}+\ddb\varphi_{0}$.   Moreover as $t$ tends to infinity $\dr(t)$ converges to the \ka-Einstein current $\dr_{KE}$ in
		both current and $C^{\infty}_{loc}(X_{reg})$-senses.  Moreover $\dr_{KE}$ is a smooth \ka-Einstein metric in $X_{reg},$ and a current with zero Lelong number and bounded local potential away from $X_{lc}.$
	\end{theorem}
	
	\begin{remark}
		We refer to \S 5  for the definition of semi-log canonical models.
	\end{remark}
	Theorem \ref{thm-main2} can be viewed as a parabolic version of the proof of existence of the \ka-Einstein current on semi-log canonical varieties in  \cite{BG,So17}.
	The next important problem is the behaviour of the \ka-Ricci flow when $K_{X}$ is not nef. In this case the flow will arrive at a singularity at the finite time $T_0$ in Theorem \ref{thm-main1}.   As indicated by \cite{Bir,Fu1,Fu2,HX}, birational surgeries will occur at time $T_0$.  In the context of the analytic minimal model program, as Song-Tian did for klt case \cite{ST3}, we have the following result asserting that the \ka-Ricci flow extends through the bi-rational surgeries in the log canonical case as in the following
	\begin{theorem}\label{thm-main3}
		Suppose that in Theorem  \ref{thm-main1}, we have that $T_{0}<\infty$
		and $H_{T_{0}}=H+T_{0}K_{X}$ is $\mathbb{Q}$-semi-ample, such that for some large integer $m$ the linear system $|mH_{T_{0}}|$ induces a morphism $\pi:X\to Y.$ Then:
		\begin{enumerate}
			\item If $\pi:X\to Y$ is a divisorial contraction, then there exists a closed semi-positive (1, 1)-current $\dr_{Y}$ with zero Lelong number on $Y$ such that
			\begin{enumerate}
				\item The weak \ka-Ricci flow can be uniquely continued on $Y$ starting with $\dr_{Y}$ at $t=T_{0}.$
				\item $\dr(t,\cdot)$ converges to $\pi^{*}\dr_{Y}$ in $C^{\infty}(X_{reg}\setminus Exc(\pi))$ as $t\to T_{0}^{-}.$
				\item Still denote the \ka-Ricci flow starting on $Y$ at $t=T_{0}$ with the initial metric $\dr_{Y}$ by $\dr(t,\cdot),$ then $\dr(t,\cdot)$ converges to $\dr_{Y}$ in $C^{\infty}(Y_{reg}\setminus\pi(Exc(\pi)))$ as $t\to T_{0}^{+}.$
			\end{enumerate}
			\item If $\pi:X\to Y$ is a small contraction, i.e., $Exc(\pi))$ has codimension greater than 1, and there exists a flip
			\begin{equation}\label{eq:flip}
				\xymatrix{
					X \ar@{.>}[rr]^{\bar{\pi}^{-1}} \ar[dr]_{\pi}
					&  &    X^{+} \ar[dl]^{\pi^{+}}    \\
					& Y                 }
			\end{equation}
			with the property that $X^{+}_{lc}\bigcap Exc(\pi^{+})=\varnothing,$
			then there exists a closed semi-positive (1, 1)-current $\dr_{Y}$ with zero Lelong number on $Y$ such that
			\begin{enumerate}
				\item $\dr(t,\cdot)$ converges to $\pi^{*}\dr_{Y}$ in $C^{\infty}(X_{reg}\setminus Exc(\pi))$ as $t\to T_{0}^{-}.$
				\item The weak \ka-Ricci flow can be uniquely continued on $X^{+}$ starting with $\pi^{+*}\dr_{Y}$ at $t=T_{0}.$ Denote the solution still by $\dr(t,\cdot),$ then $\dr(t,\cdot)$ converges to $\pi^{+*}\dr_{Y}$ in $C^{\infty}(X^{+}_{reg}\setminus Exc(\pi^{+}))$ as $t\to T_{0}^{+}.$
			\end{enumerate}
		\end{enumerate}
		In summary, the weak \ka-Ricci flow can be uniquely extended through the divisorial contractions and flips on $\mathbb{Q}$-factorial projective varieties with log canonical singularities.
	\end{theorem}

	Compared to \cite{ST3}, our main difficulty is the case when the exceptional locus has nonempty intersection with the log canonical locus $X_{lc},$ where the local potential is $-\infty.$ Actually we will show that in that case the local potential on the exceptional locus will also attain $-\infty$ with zero Lelong number. After the birational surgeries the \ka-Ricci flow will be continued with the new log canonical locus where the initial potential is $-\infty$ with zero Lelong number.
	
	The paper is organized as follows. In section 2 we establish some necessary preliminaries and definitions needed throughout the paper, and in particular, to consider the \ka\, Ricci flow of a metric $\dr_{0}+\ddb\varphi_{0}$ on a log canonical variety $X$ as in the above Theorems.  In section 3 we formulate the corresponding main degenerate  \MA flow equation to be solved on the smooth resolution $X'$ of $X$.  We then formulate a regularization of this degenerate  flow equation and derive corresponding a priori estimates.   Then in secitons 4, 5, 6  we prove Theorems \ref{thm-main1}, \ref{thm-main2} and \ref{thm-main3} respectively.
	
	\noindent{\bf Acknowledgment.} The authors want to thank Professor Gang Tian for his interest in this work and lots of encouragement. They also want to thank Professor Jian Song for his careful reading the draft and beneficial advice. The last author wants to thank Professor Chenyang Xu, Chi Li and Yuchen Liu for discussions in algebraic geometry. He also wants to thank Professor Yuan Yuan for careful discussions on the proof details. The second and last authors want to thank BICMR for its hospitality where partial work was done during the summer of 2018.  Finally, the authors would like to thank the referee for their time and effort and for numerous valuable observations and suggestions on our original manuscript.
	
	\section{Preliminaries}\label{section 2}
	
	\subsection{Divisors and singularities in the minimal model program}
	
	Let us collect some necessary background materials we need in this paper, which mainly come from \cite{KMM,KM,ST3}.
	First we recall the basic
	
	\begin{definition}\label{def-divisor}
		Given a $\mathbb{Q}$-Cartier divisor $D$ on a projective variety $X$
		\begin{enumerate}
			\item we say $D$ is ample if there exists a positive integer $m$ such that the linear system $|mD|$ induces an embedding of $X$ into $\mathbb{CP}^{N_{m}};$
			\item we say $D$ is semi-ample if there exists a positive integer $m$ such that the linear system $|mD|$ induces a morphism of $X$ into $\mathbb{CP}^{N_{m}};$
			\item we say $D$ is effective if $D=\sum_{i=1}^{k}n_{i}D_{i}$ where the integers $n_{i}\geq 0$ and $D_{i}$ are prime divisors;
			\item we say $D$ is nef if $D\cdot C\geq 0$ for any curve $C$ on $X.$
			\item we say $D$ is big if $\dim\;H^{0}(X,mD)\sim m^{n}$ for positive integer $m\to+\infty.$
		\end{enumerate}
	\end{definition}
	
	We will restrict to considering $\mathbb{Q}$-$factorial$ projective varieties which are defined as
	
	\begin{definition}\label{Qfactorial}
		An $n$-dimensional projective variety $X$ is called $\mathbb{Q}$-$factorial$ if
		$\mathbb{Q}$-divisors are in fact $\mathbb{Q}$-Cartier divisors and $X$ is normal (ie, $\dim(X_{\rm sing})\leq n-2$).
	\end{definition}
	
	Let $X$ be a $\mathbb{Q}$-factorial projective variety.  Then Hironaka's resolution theorem (cf.\cite{KMM,KM}) provides a smooth projective variety $X'$ and a birational morphism $$\pi:X'\to X$$ with the adjunction formula \eqref{eq:adjunction 1}:
	\begin{equation}\label{eq:adjunction 2}
		K_{X'}=\pi^{*}K_{X}+\sum_{j}a_{j}E_{j},
	\end{equation}
	where $E_{j}'s$ are exceptional divisors belonging to the exceptional locus $Exc(\pi)$ with simple normal crossings and $a_{j}'s$ are the unique collection of rational corresponding discrepancies.  In particular,  $\pi(\bigcup_j E_j )=X_{\rm sing}$ and $\pi: X'\setminus (\bigcup_j E_j) \to X_{\rm reg}$ is a complex holomorphic map.
	
	\begin{definition}\label{def-divisor-class}
		Let $X$ be a $\mathbb{Q}$-factorial projective variety with a resolution  $\pi:X'\to X$ as above
		
		\begin{enumerate}
			
			\item [(a)] We say that $X$ has
			\begin{enumerate}
				\item [(i)] \begin{it} log canonical singularities\end{it} if $a_{j}\geq -1$ for all $j$
				
				\item [(ii)]  \begin{it} log terminal singularities\end{it} if $a_{j}> -1$ for all $j$
				
				\item  [(iii)] \begin{it} canonical singularities\end{it} if $a_{j}\geq 0$ for all $j$
				
				\item  [(iv)] \begin{it} terminal singularities\end{it} if $a_{j}>0$ for all $j$
			\end{enumerate}
			
			\item[(b)]
			For any divisor $E_j$ in \eqref{eq:adjunction 2} we say
			\begin{enumerate}
				\item  [(i)] $E_j$ is a \begin{it}log canonical (lc) divisor\end{it} if $a_{j}=-1$,
				\item  [(ii)] $E_j$ is a \begin{it}log terminal (lt) divisor\end{it} if $-1<a_{j}<0$,
				\item  [(iii)]  $E_j$ is a \begin{it}canonical divisor\end{it}  if $a_{j}\geq 0$.
			\end{enumerate}
		\end{enumerate}
	\end{definition}

	According to \cite{KMM}, given two different resolutions of $X$ above with \begin{it}log canonical singularities\end{it}
	the classification above is independent of the choices of resolutions in the sense that the log canonical divisors are in strictly one-to-one correspondence for two different resolutions, and likewise for log terminal and canonical divisors.   In particular, we may make the following
	
	\begin{definition}\label{logcanonical}
		Let $X$ be a $\mathbb{Q}$-factorial projective variety.  We say $X$ has \begin{it}log canonical singularities\end{it} if $a_{j}\geq -1$ for all $j$ in \eqref{eq:adjunction 2} and we define the \begin{it} log canonical locus\end{it} as $X_{lc}:=\pi(\bigcup_{a_{i=-1}}E_{i})$.
	\end{definition}


	

	As  in \cite{KMM,ST3}, we have the following special case of Kodaira's Lemma, which plays a crucial role as in \cite{ST3,TZ,Ts}:
	\begin{lemma}\label{lem-kod}
		Given a $\mathbb{Q}$-factorial projective variety $X$ with a semi-ample and big $\mathbb{Q}$-divisor $H,$ for any resolution $\pi:X'\to X,$ there exists an effective divisor $E$ whose support is contained in $Exc(\pi)$ on $X'$, and $d>0$ such that $\pi^{*}H-\de E$ is ample for any rational number $0<\de<d.$
	\end{lemma}
	
	\subsection{PSH functions}
	
	Let $X$ be a $\mathbb{Q}$-factorial projective variety with a resolution $\pi:X'\to X$ as above.   Now we define a global semi-positive 1-1 form $\omega_0$ on $X$ so that $\pi^* \omega_0$ is smooth on $X'$.
	As in \cite{ST3} there exists a big and semi-ample $\mathbb{Q}$-Cartier divisor $H$ on $X$.  Thus we obtain a birational morphism $\Phi_{|mH|}: X\to\mathbb{CP}^{N_{m}}$ for some large integer $m$ and some $N_m$.  We define
	$$\dr_{0}:=\frac{1}{m}\Phi_{|mH|}^{*}\dr_{FS}\in[H]$$ which is semi-positive current on $X$, where $\dr_{FS}$ is the Fubini-Study metric on $\mathbb{CP}^{N}$.  In particular, $\Phi_{|mH|}$ is a holomorphic map on $X_{\rm reg}$ while $\pi^* \dr_{0}$ is a smooth semi-positive closed $1-1$ form on the smooth variety $X'$.   We may conveniently define plurisubharmonic functions on $X$ relative to any such form on $X$ as follows

	\begin{definition}
		Let $\omega$ be a closed 1-1 form on $X_{\rm reg}$ such that $\pi^*\omega$ extends smoothly to $X'$.   We say function $\varphi: X\to[-\infty,+\infty)$ is $\dr$-PSH on $X$ if, $u+\varphi \circ \pi$ is a classical plurisubharmonic function in local holomprhic coordinates for any local potential $u$ for $\omega$ (ie, $\ddb u =\omega$).\end{definition}
	
	As we mentioned in the introduction, unlike \cite{EGZ,Ko,ST3}, we need to deal with the currents with unbounded local potentials.  On the other hand, we will restrict to considering PSH functions on $X$ with so called zero Lelong number as in the following

	\begin{definition}\label{lem-zero Lelong}
		Suppose $\varphi$ is an $\dr$-PSH potential function on $X$ and let $E$ be a divisor on $X'$.  We say $\varphi$, or equivalently $\dr+\ddb \varphi$, has zero Lelong numbers along $\pi(E)$ if for any $\e>0$ there exists a constant $C_{\e}$ such that the pull-back $\pi^{*}\varphi$ on $X'$
		satisfies
		\begin{equation}\label{eq:zero-Lelong}
			\pi^{*}\varphi\geq\e \log|S|^{2}+C_{\e}
		\end{equation}
		where $S$ is a holomorphic section, and $|\cdot|$ is a Hermitian metric,  associated to the holomorphic line bundle associated to $E$.
	\end{definition}

	We may now define weak solutions to \eqref{eq:krf} as follows
	
	\begin{definition}\label{weaksolkrf}

		Let $\varphi_{0}$ be an $\omega_0$-PSH function on $X$ for some closed 1-1 form $\omega_0$ on $X$ which is also smooth on $X_{\rm reg}$.   We say a family of closed 1-1 forms $\omega(t)$ is a weak solution to the \ka \, Ricci flow
		
		\begin{equation}\label{eq:krfII}
			\left\{
			\begin{array}{ll}
				\displaystyle\dt\dr &=-Ric(\dr)\\
				\dr(0) &= \dr_0+\ddb\varphi_{0},
			\end{array}
			\right.
		\end{equation}
		on $X\times [0, T)$ if
		
		\begin{enumerate}
			\item $ \omega(t)$ restricts to a smooth solution to \eqref{eq:krfII} on $X_{\rm reg} \times (0, T)$.
			\item $\omega(t)=\omega_0 -t \eta +\ddb \varphi(t)$ on $X \times(0,T)$ for some $\eta\in [K_X]\bigcap C^{\infty}(X_{\rm reg})$ where $[K_X]$ denotes the canonical class of $X$ and $\varphi(t)$ is a $\omega_0 -t \eta$ PSH function on $X$.
			\item $\varphi(t)\to  \varphi_0$ in $L^1(X)$ as $t\to 0$.
		\end{enumerate}
		
		A weak solution $\omega(t)=\omega_0 -t \eta +\ddb \varphi(t)$ to the \ka \, Ricci flow on $X\times [0, T)$ above is called $maximal$ if given any other weak solution  $\omega'(t)=\omega_0 -t \eta +\ddb \varphi'(t)$ with $\varphi(0)=\varphi'(0)$, we have $\varphi(t)\geq \varphi'(t)$ on $X\times [0, T)$.

		We could likewise define a weak solution on $X$ in terms of the resolution $\pi:X'\to X$ as follows.  In conditions (1)-(3) above we replace $\omega_0, \omega(t), X, X_{\rm reg}$ and $\eta$ respectively with
		$\pi^*\omega_0, \omega'(t), X', X'\setminus Exc(\pi)$  and $\eta$ by any smooth representative $\eta'$ of $\pi^*K_{X}$ on $X'$, provided we then require the solution $\omega'(t)$ ``descends" to $X$ in the sense that:  for each $p\in X$, if $\pi^*\omega_0 -t \eta' =0$ on $\pi^{-1}(p)$ then $\varphi$ is constant on $\pi^{-1}(p)$.

	\end{definition}
	
	With the above definitions and results, we may now summarize once and for all, the main assumptions and notations we will adopt throughout the paper.

	\begin{assumption}\label{mainass}
		Let $X$ be a $\mathbb{Q}$-factorial projective variety with log canonical singularities and $H$ be a big and semi-ample $\mathbb{Q}$-Cartier divisor on $X$.  Consider the maps
		
		$$X' \xrightarrow{\pi} X \xrightarrow{\Phi} \mathbb{CP}^{N}$$
		where $\pi$ is a resolution of $X$ and $\Phi$ is a birational morphism, generated by $H,$ for some $N$.  In particular, the map $\Phi\circ \pi$ is holomorphic and non-degenerate away from the exceptional locus $Exc(\pi)$.  Let $\theta$ be a smooth \ka\,form and $\Omega'$ be a smooth volume form on $X'$.  Then we make the following assumptions and definitions
		
		\begin{enumerate}
			\item $Exc(\pi)$ is the union of simple normal crossing log canonical, log terminal and canonical divisors on $X'$ which we respectively denote by $D_i, E_j, F_k$.  In particular, we have
			\begin{align}\label{eq:coho def 1}
				\pi^* K_X=& K_{X'}+\sum_{i}D_{i}+\sum_{j}b_{j}E_{j}-\sum_{k}a_{k}F_{k},
			\end{align}
			where  $0\leq a_{k}$ and $0<b_{j}<1$  for all $k, j$.

			\item $S_i, S_j, S_k, \tilde{S}$ will respectively denote holomorphic sections of the line bundles associated with $D_i, E_j, F_k, \tilde{E}$.  $|S_i|, |S_j| , |S_k|, |\widetilde S|$ will respectively denote lengths relative to hermitian metrics $h_i, h_j, h_k, \tilde{h}$.   $\Theta_i, \Theta_j, \Theta_k, \tilde{\Theta}$  will respectively denote the curvature forms of $h_i, h_j, h_k, \tilde{h}$.
			
			\item  $T_{0}:=\sup\{t>0| \pi^* H+t(K_{X'} +\sum_{i}D_{i}+\sum_{j}b_{j}E_{j}-\sum_{k}a_{k}F_{k}  )   \;is\; nef\}=\sup\{t>0| H+t(K_{X}) \;is\; nef\}$

			\item $\dr_{0}:=\Phi^{*}\dr_{FS}\in[H]$ on $X$ where $\dr_{FS}$ is the Fubini-Study metric on $\mathbb{CP}^{N}$ and the smooth semi positive form $\pi^* \dr_{0}\in [\pi^{*}H] $ satisfies
			
			\begin{equation}\label{lowerdegeneratebound}
				\pi^* \dr_{0}\geq |\widetilde S|^{c} \theta
			\end{equation}
			on $X'$ for some $c>0$.
			
			\item
			$\widetilde E$ and $d>0$ are as in Lemma \ref{lem-kod}.  In particular, the support of $\widetilde E$ is contained in $Exc(\pi)$ and for all $0<\de <d$ we have $\pi^* \omega_0 +\de \ddb \log |\widetilde S|^2 \geq c_{\de} \theta$ for some $c_{\de} >0$.

			\item \color{black}{$\varphi_{0}$ is as in Definition \ref{weaksolkrf} and satisfies $\pi^*\varphi_{0}\in PSH(X',\pi^*\dr_{0})\bigcap L^{\infty}_{loc}(X'\setminus \widetilde E)$ and $\pi^*\dr_{0}+\pi^*\ddb\varphi_{0}$ is a current with zero Lelong number on $X'$.}
		\end{enumerate}
	\end{assumption}

	\section{A degenerate parabolic \MA\ equation}
	
	To transform the \ka-Ricci flow \eqref{eq:krf} on $X$ to a complex \MA\ flow equation, we consider a corresponding degenerate complex \MA\ flow equation on the resolution $X'$ as in \cite{BBEGZ,ST3}.  Define the following smooth family of closed 1-1 forms on $X'$
	
	\begin{equation}\label{eq:bg 1}
		\dr_{t}:=\pi^{*}\dr_{0}+t\chi=\pi^{*}\dr_{0}+t(-Ric(\Omega')+\sum_{i}\Theta_{i}+\sum_{j}b_{j}\Theta_{j}-\sum_{k}a_{k}\Theta_{k}).
	\end{equation}
	Now for a given family of $\dr_{t}$ plurisubharmonic functions $\varphi(t)$ on $X' \times [0, T)$, define the family of forms
	\begin{equation}\label{omega(t)}
		\omega(t):=\dr_{t}+\ddb\varphi
	\end{equation}
	on $X' \times [0, T)$.   By a straight forward computation using  Poincar\'{e}-Lelong Formula, it follows that $\omega(t)$ is a weak solution to \ka\, Ricci flow as in definition \ref{weaksolkrf}  on $(X'\setminus Exc(\pi)) \times (0, T)$ provided that $\varphi$ is a smooth solution to the equation
	\begin{equation}\label{eq:ma flow 1}
		\left\{
		\begin{array}{ll}
			\displaystyle\dt\varphi &=\displaystyle\log\frac{(\dr_{t}+\ddb\varphi)^{n}\prod_{i}|S_{i}|_{i}^{2}\prod_{j}|S_{j}|_{j}^{2b_{j}}}{\Omega'\prod_{k}|S_{k}|_{k}^{2a_{k}}}\\
			\varphi(0) &= \pi^{*}\varphi_{0},
		\end{array}
		\right.
	\end{equation}
	on $(X'\setminus Exc(\pi)) \times (0, T)$ with $\varphi(t)\to \pi^* \varphi_0$ in $L^1(X')$ as $t\to 0$.

	In particular, if $\varphi$ solves \eqref{eq:ma flow 1} in the above sense and $\omega(t)$ descends to $X$, then we obtain a weak solution to \eqref{eq:krf} on $X \times [0, T)$ as in Definition \ref{weaksolkrf}.  Note that as $T_{0}:=\sup\{t>0| \pi^* H+tK_{X'}\;is\; nef\}$ as in Assumption \ref{mainass}, the adjunction formula in Assumption \ref{mainass} (1) implies that the smooth background form $\dr_{t}$ is nef on $X'$ for all $t\in [0, T_0)$, and it follows that for any $T'<T_0$, we may have $\dr_{t}+ \ddb \psi_{T'} \geq 0$ for some $\psi_{T'}\in C^{\infty}(X')$.  As mentioned in the introduction however, due to the existence of lc divisors in the resolution \eqref{eq:adjunction 1}, we cannot make direct use of the estimates in \cite{EGZ, ST3} to construct solutions to  \eqref{eq:ma flow 1}.  We will establish the a priori estimates of \eqref{eq:ma flow 1} instead through an approximation process in the next section involving the use of both the approximate conical \ka\, metrics and Carlson-Grifiths metrics on $X'$.   Solving \eqref{eq:ma flow 1} for a zero Lelong number solution $\varphi(t)$ on $X'\setminus{\tilde{E}} \times[0, T_0)$ can be regarded as the chief analytic goal of this paper.

\subsection{An approximate equation (existence)}

To study solutions to \eqref{eq:ma flow 1} on $X',$ we need to overcome the singularities in the equation corresponding to the lc divisors $D_{i}$ and lt divisors $E_{j}$.  We will do this by perturbing these singular terms in \eqref{eq:ma flow 1} to arrive at an approximate equation which is known to have a solution on $X'\times[0, T_0)$.   We begin with the following Lemmas which will be used in this perturbative process.

The following approximation lemma follows from Theorem 2 in \cite{Blo} and will be used to deal with the singularities in the initial potential $\pi^* \phi_0$.
\begin{lemma}[\cite{Blo}] \label{lem-demailly}
Given a smooth semi-positive $(1,1)$-form $\dr_{0}$ on $X'$ and a $\dr_{0}$-PSH function $\varphi,$ there exist non-increasing sequence of functions $\{\varphi_{l}\}$ on $X'$ such that 
\begin{enumerate}
\item each $\varphi_l$ is smooth on $X'$ and a  $\dr_{0}+l^{-1}\theta$-PSH function,
\item  $\varphi_{l}\searrow \varphi$ pointwise on $X'$,
\item $(\dr_{0}+\ddb\varphi_{l})\to(\dr_{0}+\ddb\varphi)$ in the current sense on $X'$.
\end{enumerate}
\end{lemma}

We will use the following approximate conical \ka \, forms to handle the singularities in \eqref{eq:ma flow 1} around the log terminal divisors $E_j$.
\begin{lemma}[\cite{GP}]\label{conicalapprox}
Define the function

\begin{equation}\label{conic-regg}
\mathcal{F}(t,\beta,\e):=\frac{1}{\beta}\int_{0}^{t}\frac{(r+\e)^{\beta}-\e^{\beta}}{r}dr,
\end{equation}
Given a \ka \, form $\theta$ on $X'$ and Hermitian metric $h_j$ on $S_j$, there exists $\eta>0$ such that for all $\e$ sufficiently small,  $\theta + \eta\ddb\mathcal{F}(|S_{j}|_{j}^{2},\beta,\e^{2})$ is a \ka \, form on $X'$ and is uniformly (over  $\e$) equivalent to the local model
\begin{align}\label{eq:cone-app}
 \sqrt{-1} \sum_{j=2}^{n} dz^j \wedge dz^{\bar{j}} +\sqrt{-1}\frac{dz_{1}\wedge d\bar{z}_{1}}{(|z_{1}|^{2}+\e^{2})^{1-\beta}}
\end{align}
 in local holomorphic coordinates where $E_j=\{z_1=0\}$.

\end{lemma}

The singularities in \eqref{eq:ma flow 1} around the log canonical divisors $E_j$ will be dealt with through the use of Carlson-Grifiths forms on which we have the following lemma (see \cite{CG} or  \cite{Guenancia})

\begin{lemma}[\cite{CG}]\label{CGlemma}
Given a \ka \,  form $\theta$ on $X'$ and Hermitian metric $h_i$ on $S_i$, we may scale $h_j$ so that the Carlson-Griffiths type form \begin{equation}\label{CGexpansion}\begin{split}\widehat{\omega}_{\theta, h}:=&\theta - \ddb (\log\log^2 \|S\|^{2}_{h})\\
&=\theta -2\frac{\ddb \log \|S\|^{2}_{h}}{\log \|S\|^{2}_{h}}+2 \frac{\partial \log \|S\|^{2}_{h}}{\log \|S\|^{2}_{h}}\wedge \frac{\bar{\partial} \log \|S\|^{2}_{h}}{\log \|S\|^{2}_{h}}\\
\end{split}
\end{equation}
satisfies

\begin{enumerate}
\item $\hat{ \omega}_{\theta,  h}$ is a complete \ka  \, metric on $X' \setminus \bigcup_i D_i $, and for any $D_k$ it is equivalent to the local model \begin{equation}\label{localCGmodel}  \sqrt{-1}\sum_{j=2}^{n} dz^j \wedge dz^{\bar{j}}+\sqrt{-1} \displaystyle \frac{ dz^1 \wedge dz^{\bar{1}}}{|z^1|^2 \log^2 |z^1|^2} \end{equation} in local holomorphic coordinates around any point $p\in D_k$, were $D_k=\{z_1=0\}$.
\item $\hat{ \omega}_{\theta, h}$ has bounded geometry of infinite order.
\item $-\log\log^2 \|S\|^{2}_{h}$ is bounded above and in $L^1(X')$
\item $\log\frac{\widehat{\omega}_{\theta, h}^{n}\|S\|_{h}^{2}\log^{2}\|S\|_{h}^{2}}{\theta^n}$ is bounded on $X'\setminus  \bigcup_i D_i $.  \end{enumerate}
 In particular (3)  implies that $\widehat{\omega}_{\theta, h}$ is a well defined current on $X'$. (see for example \cite{LZ} (\S8, example 8.15)).
\end{lemma}

Now we may write our approximation of \eqref{eq:ma flow 1}.  For any positive integer $l$, and real numbers $\eta, v, \e >0$ and consider the equation

\begin{equation}\label{eq:ma flow 4}
\left\{
   \begin{array}{ll}
     \displaystyle\dt\varphi'_{l,v,\e} &=\displaystyle\log\frac{(\dr'_{t,l,v, \e}+\ddb\varphi'_{l,v,\e})^{n}\prod_{i}|S_{i}|_{i}^{2}\log^{2}|S_{i}|_{i}^{2}
     \prod_{j}(|S_{j}|_{j}^{2}+\e^{2})^{b_{j}}}{\Omega'\prod_{k}(|S_{k}|_{k}^{2}+\e^{2})^{a_{k}}}\\
     \;\\
     \varphi'_{l,v,\e}(0) &= \varphi_{l,0}-\eta\sum_{j}\mathcal{F}(|S_{j}|_{j}^{2},1-b_{j},\e^{2}).
   \end{array}
 \right.
\end{equation}
where
\begin{align} \nonumber
\dr'_{t,l,v,\e}:&=\pi^{*}\dr_{0}+l^{-1}\theta+t\chi-(t+v)\sum_{i}\ddb\log\log^{2}|S_{i}|_{i}^{2}\nonumber\\\nonumber
&+\eta\sum_{j}\ddb\mathcal{F}(|S_{j}|_{j}^{2},1-b_{j},\e^{2}),
\end{align}

Here $\varphi_{l,0}$ is the non-increasing sequence of $\pi^* \omega_0+l^{-1}\theta$ PSH functions given by Lemma \ref{lem-demailly} where $\pi^* \varphi_0$ is the initial condition for  \eqref{eq:ma flow 1} and $\eta>0$ is chosen as in  Lemma \ref{conicalapprox} relative to $\theta$.
We suppress the dependence of $\varphi$ on $\eta$ as this parameter will be fixed at some point in our arguments, and we will not need to let $\eta$ pass to any limits. This is partially related to the fact that the initial form $\dr'_{0,l, v,\e}+\ddb \varphi'_{l,v,\e}(0)$ is independent of $\eta$ and the following remark which explains the sense in which equation \eqref{eq:ma flow 4}  is a perturbation of \eqref{eq:ma flow 1}.

\begin{remark}\label{r00}
If a family of solutions $\varphi'_{l,v,\e}$ to \eqref{eq:ma flow 4} converges locally smoothly to a limit $\varphi'\in C^{\infty}(X' \setminus \widetilde{E} \color{black}{ \times(0, T_0)}\color{black}{})$, as we let $l\to \infty$ and $v,\e, \to 0$, then $\varphi'$ will solve the equation

\begin{equation}\label{eq:ma flow 2}
\left\{
   \begin{array}{ll}
     \displaystyle\dt\varphi' &=\displaystyle\log\frac{(\dr'_{t}+\ddb\varphi')^{n}\prod_{i}|S_{i}|_{i}^{2}\log^{2}|S_{i}|_{i}^{2}
     \prod_{j}(|S_{j}|_{j}^{2})^{b_{j}}}{\Omega'\prod_{k}(|S_{k}|_{k}^{2})^{a_{k}}}\\
     \;\\
     \varphi'(0) &= \pi^* \varphi_{0}-\eta\sum_{j}\mathcal{F}(|S_{j}|_{j}^{2},1-b_{j},0).
   \end{array}
 \right.
\end{equation}
on $(X' \setminus \widetilde{E} \color{black}{ \times(0, T_0)}\color{black}{})$ where
\begin{align} \nonumber
\dr'_{t}:&=\pi^{*}\dr_{0}+t\chi-t\sum_{i}\ddb\log\log^{2}|S_{i}|_{i}^{2}+\eta\sum_{j}\ddb\mathcal{F}(|S_{j}|_{j}^{2},1-b_{j},0),
\end{align}

  In particular, $\varphi =\varphi' +\eta\sum_{j}|S_{j}|_{j}^{2(1-b_{j})}-t \log \log^{2}|S_{i}|_{i}^{2}$ will solve \eqref{eq:ma flow 1} on $X' \setminus \widetilde{E}\color{black}{ \times(0, T_0)}\color{black}{}) $.
\end{remark}

 Roughly speaking, we have perturbed so that the corresponding background form
$\dr'_{t,l,v,\e}$ will be equivalent to a Carlson-Grifiths form around the log canonical divisors $D_i$, and will be approximately conical near the log terminal divisors $E_j$.  This will essentially allow us to combine the techniques from \cite{BG,CLS,EGZ, LZ} with those in  \cite{ST3} in our study of \eqref{eq:ma flow 4}.   In particular, we have the following existence theorem essentially due to \cite{LZ}. 

\begin{theorem}\label{thmLZ}
Suppose $\dr'_{0,l,v,\e} + \ddb \varphi'_{l,v,\e}(0)$ is \ka \, on $X'\color{black}{ \setminus  \bigcup_{i}D_{i}}\color{black}{}$.  Then \eqref{eq:ma flow 4} has a solution $\varphi'_{l,v,\e}\in C^{\infty}((X' \setminus \bigcup_{i}D_{i})\times[0, T_0)) \bigcap L^{\infty}((X' \setminus \bigcup_{i}D_{i})\times[0, T_0))$.  Moreover, the family of K\"{a}hler metrics  $\dr'_{t,l,v,\e}+\ddb\varphi'_{l,v,\e}$ is equivalent to a Carlson-Grifiths form on $X' \setminus \bigcup_{i}D_{i}$, and has bounded curvature, for all $t\in [0, T_0)$.  \end{theorem}

\begin{proof}
Consider the Carlson Grifiths metric on $X'\color{black}{ \setminus  \bigcup_{i}D_{i}}\color{black}{}$ given by

\begin{equation}\label{initialCG}
\begin{split}
\widetilde{\omega} &=\dr'_{0,l,v,\e}+\ddb \varphi'_{l,v,\e}(0)\\
&= \pi^{*}\dr_{0}+l^{-1}\theta-v\sum_{i}\ddb\log\log^{2}|S_{i}|_{i}^{2}+\eta\sum_{j}\ddb\mathcal{F}(|S_{j}|_{j}^{2},1-b_{j},\e^{2})+\ddb \varphi_{l, 0}\\
\end{split}
\end{equation}

\color{black}{ Then $$\varphi'_{l,v,\e}(t):=\varphi'_{l,v,\e}(0)+\psi+\displaystyle \int_0^t \displaystyle\log\frac{\widetilde{\omega}^n\Omega'\prod_{k}(|S_{k}|_{k}^{2}+\e^{2})^{a_{k}}}{\prod_{i}|S_{i}|_{i}^{2}\log^{2}|S_{i}|_{i}^{2}
     \prod_{j}(|S_{j}|_{j}^{2}+\e^{2})^{b_{j}}}$$ will solve \eqref{eq:ma flow 4} exactly when $\psi$ solves }\color{black}{}

\begin{equation}\label{eq:ma flow 6}
\left\{
   \begin{array}{ll}
     \displaystyle\dt\psi&=\displaystyle\log\frac{(\sigma_t+\ddb\psi)^{n}}{\widetilde{\omega}^n} \\
     \;\\
     \psi(0) &= 0
   \end{array}
 \right.
\end{equation}
where
\begin{align} \nonumber
\sigma_t:= \widetilde{\omega}-t Ric( \widetilde{\omega}) +t\left(\sum_{j}b_{j}\Theta_{j}-\sum_{k}a_{k}\Theta_{k}-\ddb\log\frac{\prod_j (|S_{j}|_{j}^{2}+\e^{2})^{b_j}}{\prod_k (|S_{k}|_{k}^{2}+\e^2)^{a_k} }\right)
\end{align}

Now as $\widetilde{\omega}$ is complete on $X'$ with bounded covariant derivatives of curvature, the proof of the main theorem in \cite{LZ} implies that \eqref{eq:ma flow 6} has a solution on $X'\setminus {\bigcup_i D_i} \times[0, T_0)$ such that $\sigma_t+\ddb\psi$ is a Carlson-Grifiths metric on $X'$ for each $t$, provided that for every $T<T_0$ we have $\sigma_T \geq c\widetilde{\omega} + \ddb F$ for some $c$ and smooth $F$ which is bounded on $X'$ along with all covariant derivatives relative to $\widetilde{\omega}$.  Now for any $T$ we may use Lemma 8.6 in \cite{LZ} to calculate the Ricci form of a Carlson-Griffiths metric as in the following
\begin{equation}
\begin{split}
\sigma_T=&  \widetilde{\omega}-T Ric( \widetilde{\omega})\\ &+T\left(\sum_{j}b_{j}\Theta_{j}-\sum_{k}a_{k}\Theta_{k}-\ddb\log\frac{\prod_j (|S_{j}|_{j}^{2}+\e^{2})^{b_j}}{\prod_k (|S_{k}|_{k}^{2}+\e^2)^{a_k} }\right)\\
= & \widetilde{\omega}-T (Ric(\theta)+ \ddb\log\log^{2}|S_{i}|_{i}^{2}+\sum_{i}\Theta_{i} + \ddb H) \\
&+T\left(\sum_{j}b_{j}\Theta_{j}-\sum_{k}a_{k}\Theta_{k}-\ddb\log\frac{\prod_j (|S_{j}|_{j}^{2}+\e^{2})^{b_j}}{\prod_k (|S_{k}|_{k}^{2}+\e^2)^{a_k} }\right)\\
= & \widetilde{\omega}-T \left(Ric(\theta)+\sum_{i}\Theta_{i}+ \sum_{j}b_{j}\Theta_{j}-\sum_{k}a_{k}\Theta_{k}\right)-T \ddb\log\log^{2}|S_{i}|_{i}^{2} \\
\end{split}
\end{equation}

\begin{equation}\nonumber
\begin{split}
&-T\ddb\left(\log\frac{\prod_j (|S_{j}|_{j}^{2}+\e^{2})^{b_j}}{\prod_k (|S_{k}|_{k}^{2}+\e^2)^{a_k} } + H\right)\\
=:& I + II + III+ IV
\end{split}
\end{equation}
where $H$ is bounded on $X'$ along with all covariant derivatives relative to $\widetilde{\omega}$.  It follows from the definition of $T_0$ in Assumption \ref{mainass} and \eqref{initialCG} that if $T<T_0$, then $I+II  \geq c\widetilde{\omega} + \ddb F$ for some $c$ and smooth $F$ which is bounded on $X'$ along with all covariant derivatives relative to $\widetilde{\omega}$.  Thus in turn,  by the above we conclude that $\sigma_T\geq c'\widetilde{\omega} + \ddb F'$ for some $c'$ and smooth $F'$ which is bounded on $X'$ along with all covariant derivatives relative to $\widetilde{\omega}$.

 We conclude that \eqref{eq:ma flow 6}, and thus \eqref{eq:ma flow 4} has a solution on $X'\setminus {\bigcup_i D_i} \times[0, T_0)$ with the properties stated in the Theorem.

\end{proof}

\subsection{An approximate equation (a priori estimates)}


 In this subsection we fix Hermitian metrics $h_i, h_j, h_k$ on $D_i, E_j, F_k$ and a solution $\varphi'_{l,v,\e}\in C^{\infty}((X' \setminus \bigcup_{i}D_{i})\times[0, T_0))$
 to \eqref{eq:ma flow 4} as in Theorem \ref{thmLZ} for some values of the parameters $l,v,\e>0$ and $\eta>0$.   We will derive a priori estimates for $\varphi'_{l,v,\e}$ which are local in $C^{\infty}((X' \setminus \widetilde{E})\times(0, T_0))$, and independent of the parameters $l,v,\e$.

For the remainder of the subsection we will fix some $T'\in (0, T_0)$ and some $\de \in (0, \frac{d}{T'+1}(1-\frac{T'}{T''}))$ where $d$ is from Assumption \ref{mainass} (4) and we will let $T''=T'+(T_0 -T')/2$.  Our goal, more precisely, will be to estimate $\varphi'_{l,v,\e}$ on $(X' \setminus \widetilde{E})\times(0, T']$ depending possibly on $\delta, T'$ and possibly $\eta$, but independently of $l,v,\e$ provided $v$ is sufficiently small depending on $l$.  This assumption on $v$ is without loss of generality as we will first let $v\to 0$ in the limit process we adopt in \S 4.1. We will also suppress mention of our estimates on $\eta$ since, as we mentioned below \eqref{eq:ma flow 4}, it will be fixed throughout and at no point will we need to let $\eta$ pass to a limit (unlike $\delta, T'$).

 We will also make the following assumption throughout the subsection

\begin{assumption}\label{changinghermitianmetricsass}
	Given the fixed constants above, the forms
	\begin{enumerate}
		\item[(i)] \color{black}{$\theta+\eta\sum_{j}\ddb\mathcal{F}(|S_{j}|_{j}^{2},1-b_{j},\e^{2})-\sum_{i}\ddb\log\log^{2}|S_{i}|_{i}^{2},$}\color{black}{}
		\item[(ii)] $\pi^{*}\dr_{0}+l^{-1}\theta+t\chi + t \de \ddb \log^{2}|S_{\widetilde{E}}|^{2} \color{black}{-(t+v)\sum_{i}\ddb\log\log^{2}|S_{i}|_{i}^{2},}\color{black}{}$
		\item[(iii)] $\dr'_{t,l,v,\e} + \de \ddb\log^{2}|S_{\widetilde{E}}|^{2}$
	\end{enumerate}
	are \ka\, on $X'\setminus \widetilde{E}$ for all $t\in [0, T']$ and in particular, for any fixed $i, j$, equivalent to the local model
	
	\begin{equation}\label{eq:bg-expansion}
		\begin{split}
			\sum_{k\geq 3}dz^k\wedge d\overline{z}^k +  (t+v)\sqrt{-1}(\frac{dz^1\wedge d\overline{z}^1}{|z^1|_{i}^{2}\log^{2}|z^1|_{i}^{2}}+\frac{\Theta_{i}}{\log |z^1|_{i}^{2}})
			+\eta\sqrt{-1}\frac{dz^2\wedge d\bar{z}_2}{(|z^2|_{j}^{2}+\e^{2})^{b_{j}}}\\
		\end{split}
	\end{equation}
	in some local holomorphic coordinate around any point in which $D_i=\{z^1=0\}$ and $E_j=\{z^2=0\}$.
\end{assumption}
We explain how this assumption follows provided $v$ is small (depending on $l$) and by making appropriate choices of $h_i, h_j, h_k, \eta$ (depending on $T'$ and $\delta$) which we may do without loss of generality in view of remarks  \ref{changingeta} and \ref{eq:ma flow 4} (see also remark \ref{r00}).

 For part (i), we may choose $\eta>0$ sufficiently small and  $h_i$ (by  Lemma \ref{conicalapprox}) so that the forms
 $$\theta/2+\eta\sum_{j}\ddb\mathcal{F}(|S_{j}|_{j}^{2},1-b_{j},\e^{2})$$ and $$\theta/2-\sum_{i}\ddb\log\log^{2}|S_{i}|_{i}^{2}$$
 are both \ka\, in which case the form in (i) will also be \ka\, on $X'$.
 
 For part (ii), we begin by assuming $v$ is sufficiently small depending only on $l$ so that 
 $$\Theta_1:=l^{-1}\theta-v\sum_{i}\ddb\log\log^{2}|S_{i}|_{i}^{2}$$ 
 
 is \ka\, on $X'$.  Next we consider
 \begin{equation}
 	\begin{split}
 		\Theta_2:=&\pi^{*}\dr_{0}+T'\chi + T'\de \ddb \log^{2}|S_{\widetilde{E}}|^{2}-T'\sum_{i}\ddb\log\log^{2}|S_{i}|_{i}^{2}+\Theta_1\\
 		=&(T'/T'') (\pi^{*}\dr_{0}+T''\chi)+(1-T'/T'' )\pi^* \omega_0+ T'\de \ddb \log^{2}|S_{\widetilde{E}}|^{2}\\
 		&-T'\sum_{i}\ddb\log\log^{2}|S_{i}|_{i}^{2}+\Theta_1\\
 		=:&I+II+III+IV+\Theta_1.
 	\end{split}
 \end{equation}
 Now we may choose $h_j, h_k$ so that $I$ is non-negative on $X'$ (by Assumption \ref{mainass} (3)), then note that $II+III$ is \ka\ on $X'$ by Assumption \ref{mainass} (5) and our choice of $c$ and $\delta$, then scale $h_i$ (which does not affect $I$) so that $II+III+IV$ is \ka\, on $X'$ (by Lemma \ref{CGlemma}).  Finally, we observe that the form in part (ii) is just a linear interpolation between the \ka\, forms $\Theta_1+\pi^{*}\dr_{0}$ and $\Theta_2$ and is thus also \ka\ on $X'$ for any $t\in [0, T']$. 
 
 Finally for part (iii), by Lemma \ref{conicalapprox} and Assumption \ref{mainass} (5) we may further shrink  $\eta$ and $\delta$ if necessary so that 
 $$\Theta_3:=\de \ddb \log^{2}|S_{\widetilde{E}}|^{2}+\eta\sum_{j}\ddb\mathcal{F}(|S_{j}|_{j}^{2},1-b_{j},\e^{2})+\Theta_1$$
 is \ka\, on $X'$.  Now we consider
 $$\Theta_4:=\pi^{*}\dr_{0}+T'\chi-T'\sum_{i}\ddb\log\log^{2}|S_{i}|_{i}^{2}+\Theta_3$$
 and we may proceed to show, as in the case for $\Theta_2$, that we may further scale $h_i$ if necessary so that $\Theta_4$ is \ka\, on $X'$.  Again, we observe that the form in part (iii) is just a linear interpolation between the \ka\, forms $\Theta_3+\pi^{*}\dr_{0}$ and $\Theta_4$ and is thus also \ka\ on $X'$ for any $t\in [0, T']$.

\begin{remark}\label{changingeta}
Changing from $\eta=a$ to $\eta=b$ in \eqref{eq:ma flow 4}, while fixing $l, v, \e$, corresponds simply to adding $(a-b)\mathcal{F}(|S_{j}|_{j}^{2},1-b_{j},\e^{2})$ to the solution.
\end{remark}

 \begin{remark}\label{changinghermitianmetrics}
 Given Hermitian metrics $h_i, h_j, h_k$ on $D_i, E_j, F_k$ respectively, a smooth volume form $\Omega'$ on $X'$, denote by $\varphi'^{h_i, h_j, h_k, \Omega', \eta}$ the corresponding solution to \eqref{eq:ma flow 4} (for some fixed set of  parameters $l,v,\eta$ which we suppress in the notation).  Then using this notation, for another set of Hermitian metrics  $\widetilde h_i,\widetilde h_j, \widetilde h_k$ we have

 $$\varphi'^{h_i, h_j, h_k, \Omega'}=\varphi'^{h_i, \widetilde h_j, \widetilde h_k, \widetilde \Omega'}=\varphi'^{\widetilde h_i, \widetilde h_j, \widetilde h_k, \widetilde \Omega'}+\psi$$ for some smooth volume form $\widetilde \Omega'$ on $X'$ and smooth bounded function $\psi$ on $X'\setminus \widetilde{E} \times [0, T_0)$.
\end{remark}

\subsubsection{Upper bound on $\varphi'_{l,v,\e}$}
We begin with
\begin{lemma}\label{lem-upper 1}
There exists a uniform constant $C$ depending only on $\delta$ and  $T'$ such that on $X' \setminus (\bigcup_i D_i) \times[0, T']$ we have
\begin{equation}\label{eq:upper-1}
\sup\varphi'_{l,v,\e}\leq C-t\sum_{k}a_{k}\log(|S_{k}|_{k}^{2}+\e^{2}).
\end{equation}
\end{lemma}

\begin{proof}

Consider the function $$\phi_{l,v,\e}(t):=\varphi'_{l,v,\e}(t)+t\sum_{k}a_{k}\log(|S_{k}|_{k}^{2}+\e^{2})$$ which satisfies the equation
\begin{equation}\label{eq:ma flow 5}
\left\{
   \begin{array}{ll}
     \displaystyle\dt\phi_{l,v,\e} &=\displaystyle\log\frac{(\dr''_{t,l,v,\e}+\ddb\phi_{l,v,\e})^{n}\prod_{i}|S_{i}|_{i}^{2}\log^{2}|S_{i}|_{i}^{2}
     \prod_{j}(|S_{j}|_{j}^{2}+\e^{2})^{b_{j}}}{\Omega'}\\
     \;\\
     \phi_{l,v,\e}(0) &= \varphi_{l,0}-\eta\sum_{j}\mathcal{F}(|S_{j}|_{j}^{2},1-b_{j},\e^{2}),
   \end{array}
 \right.
\end{equation}
where
\begin{align*}
\dr''_{t,l,v, \e}=&\dr'_{t,l,v,\e}-t\sum_{k}a_{k}\ddb\log(|S_{k}|_{k}^{2}+\e^{2})\\
=&\dr'_{t,l,v,\e}+t\sum_{k}a_{k}\sqrt{-1}\left(\frac{|S_{k}|_{k}^{2}\Theta_{k}}{|S_{k}|_{k}^{2}+\e^{2}}
-\frac{DS_{k}\wedge\overline{DS}_{k}}{(|S_{k}|_{k}^{2}+\e^{2})^{2}}\right)\\
\leq&C_1\dr'_{t,l,v,\e},
\end{align*}
for some uniform constant $C_1>0$ depending only on $T'$ where in the last inequality we have used Assumption \ref{changinghermitianmetricsass} (iii)

We may now prove the Lemmma using a maximum principle argument as follows.  Using Assumption \ref{changinghermitianmetricsass} (iii), Lemmas \ref{conicalapprox} and \ref{CGlemma} we may define
\begin{equation}
	\begin{split}
		 C_2:=\sup_{(X'\setminus (\bigcup_i D_i)) \times[0, T']}\log&\frac{((C_1+1)\dr'_{t,l,v,\e})^{n}\prod_{i}|S_{i}|_{i}^{2}\log^{2}|S_{i}|_{i}^{2}
			\prod_{j}(|S_{j}|_{j}^{2}+\e^{2})^{b_{j}}}{\Omega'}\\
	\end{split}
\end{equation}
 where $0<C_2<\infty$ depending on $\de, T'$.  Now let $C_3=\sup_{X'\setminus (\bigcup_i D_i)} \phi_{l,v,\e}(\cdot, 0)$ and consider the function 
$$\Psi:=\phi_{l,v,\e}-C_3-(C_2+1)t.$$
Then $\Psi$ is  bounded on $(X'\setminus (\bigcup_i D_i)) \times[0, T']$ by Theorem \ref{thmLZ}, and it follows from  Omori-Yau's maximum principle that there is a sequence $(p_i, t_i)\in (X'\setminus (\bigcup_i D_i)) \times[0, T']$ and a sequence $c_i\to 0$ so that for all $i$ we have
 \begin{enumerate}
 	\item $\Psi(p_i, t_i) >  \sup_{X'\setminus (\bigcup_i D_i)\times[0, T']} \Psi -c_i$
 	\item  $\ddb \Psi(p_i, t_i)=\ddb\phi_{l,v,\e} \leq c_i \dr'_{t,l,v,\e} $
 \end{enumerate}
 where we have used the fact that $\dr'_{t,l,v,\e}$ is complete with bounded curvature by Assumption \ref{changinghermitianmetricsass} (iii) and Lemma \ref{CGlemma}.  Combining this with \eqref{eq:ma flow 5} to give
 
 \begin{equation}\label{finalcd7}
 \frac{\partial}{\partial t} \Psi(p_i, t_i)\leq -1
 \end{equation} 
for large $i$ while $|\partial^2 \Psi/\partial t^2 |$ is uniformly bounded on $(X'\setminus (\bigcup_i D_i))\times[0, T']$ by Theorem \ref{thmLZ}.  Combining this with (1) above gives $$\Psi(p_i, t_i/2) >  \sup_{X'\setminus (\bigcup_i D_i)\times[0, T']} \Psi -c_i + C_4 (t_i /2)$$
for large $i$ and some $C_4>0$ independent of $i$.  We must then have $t_i\to 0$ which in turn implies that $\Psi\leq \sup_{X'\setminus (\bigcup_i D_i)} \Psi (\cdot, 0)=0$ on $(X'\setminus (\bigcup_i D_i))\times[0, T']$ which in turn implies the Lemma by the definition of $\Psi$.



\end{proof}

Using Lemma \ref{lem-upper 1} we now derive a uniform upper bound as in the following

\begin{theorem}\label{thm:upper-2}
There exists a constant $C$ depending only on $\delta$ and $T'$  such that on $X' \setminus(\bigcup_i D_i) \times[0, T']$ we have
\begin{equation}\label{eq:upper-2}
\sup\varphi'_{l,v,\e}\leq C.
\end{equation}
\end{theorem}
\begin{proof}
First we observe from \eqref{eq:upper-1} that it suffices to conclude an upper bound in a neighbourhood of the canonical divisors $\bigcup_k F_{k}.$  Let us cover $\bigcup_k F_{k}.$ by finitely many coordinate charts $V_{\al}$ such that the exceptional divisors correspond to coordinate hyperplanes in each $V_{\al}$. We will derive a uniform upper bound for $\varphi'_{l,v,\e}$ in each compliment $V_{\al}\setminus \bigcup_i D_{i}.$

We may suitably shrink the charts such that in each $V_{\al}$ we have  $\dr'_{l,t}:=\pi^{*}\dr_{0}+l^{-1}\theta+t\chi=\ddb\Phi_{\al}$ smooth function $\Phi_{\al}$ which is uniformly bounded by a constant $C_{0}$ depending only on $T'$.  Moreover we may require that each chart has a shape of the product of the polydisk
$\mathbb{D}^{n'}:=\{|z_{k_{1}}|\leq r_{1},|z_{k_{2}}|\leq r_{2},\cdots,|z_{k_{n'}}|\leq r_{n'}\}$ and a bounded region $U^{n-n'}$ in $\mathbb{C}^{n-n'},$
where \color{black}{those canonical divisors $F_{k_{i}}$ intersecting the chart correspond to $z_{k_{i}}$ hyperplanes, while those canonical divisors $D_i$ intersecting the chart correspond to $z_i$ hyperplanes (recall that the canonical divisors have simple normal crossings in $X'$ by assumption)}.  \color{black}{}Fix one such chart $V_{\al}$.  It suffices to prove the following claim:
\begin{equation}\label{eq:nbhd bd}
\sup_{V_{\al}}\varphi'_{l,v,\e}\leq\sup_{\mathbb{T}^{n'}\times U^{n-n'}}\varphi'_{l,v,\e}+2n'C_{0},
\end{equation}
where $\mathbb{T}^{n'}:=\{|z_{k_{1}}|=r_{1},|z_{k_{2}}|=r_{2},\cdots,|z_{k_{n'}}|=r_{n'}\}$ is the boundary torus around $\mathbb{D}^{n'}.$ Suppose the claim is true.  Now as $F_{k_{1}},\cdots,F_{k_{n'}}$ are the only canonical divisors intersecting $V_{\al}$ by assumption, it follows the RHS of \eqref{eq:nbhd bd} is uniformly bounded from above by Lemma \ref{lem-upper 1}, in which case we see that the theorem follows by the finiteness of the covering $\{V_{\al}\}$ of $\bigcup_k F_{k}.$

Let us prove the claim \eqref{eq:nbhd bd} by induction. For any point $p\in V_{\al}\setminus\bigcup_{i}D_{i}$ with local coordinates $(z_{k_{1}},\cdots,z_{k_{n'}},z'_{n-n'}),$ where $|z_{k_{i}}|\leq r_{i}$ and $z_{i}\neq 0,$ fix all coordinates except for $z_{k_{1}}$ and apply the maximum principle to the local potential of the metric $\dr'_{t,l,v,\e}+\ddb\varphi'_{l,v,\e}$ on the one-dimensional disk $\mathbb{D}^{1}:=\{|z_{k_{1}}|\leq r_{1}\}\times\{z'_{n-1}\},$ it follows that
\begin{equation}\label{upperbound1}
\begin{split}
&(\Phi_{\al}-(t+v)\sum_{i}\log\log^{2}|S_{i}|_{i}^{2}
+\eta\sum_{j}\mathcal{F}(|S_{j}|_{j}^{2},1-b_{j},\e^{2})+\varphi'_{l,v,\e})(p)\\
\leq&\max_{\mathbb{T}^{1}}(\Phi_{\al}-(t+v)\sum_{i}\log\log^{2}|S_{i}|_{i}^{2}
+\eta\sum_{j}\mathcal{F}(|S_{j}|_{j}^{2},1-b_{j},\e^{2})+\varphi'_{l,v,\e}),
\end{split}
\end{equation}
where $\mathbb{T}^{1}$ is the boundary of $\mathbb{D}^{1}.$ \color{black}{Now the functions $\Phi_{\alpha}$ and $\mathcal{F}(|S_{j}|_{j}^{2},1-b_{j},\e^{2})$ are uniformly bounded in $V_{\al}\setminus\bigcup_{i}D_{i}$ indpendent of $\alpha$.   Assume the maximum on the RHS of \eqref{upperbound1} is attained at some point $q$.   Then for each $i$ we can write the difference $\log\log^{2}|S_{i}|_{i}^{2}(q)-\log\log^{2}|S_{i}|_{i}^{2}(p)$ as $\log\log^{2}a(q)|z_i|^{2}-\log\log^{2}a(p)|z_i|^{2}=2\log\frac{|\log a(q)+\log |z_i|^{2}|}{|\log a(p)+\log |z_i|^{2}|}$ where the function $a_i$ is uniformly positive and bounded on $V_{\al}\setminus\bigcup_{i}D_{i}$.  Noting that the last expression approaches $1$ as $z_i(p)=z_i(q)$ approaches $0$, we conclude that the difference is bounded in absolute value independent of
$p, q \in V_{\al}\setminus\bigcup_{i}D_{i}$ (recall that by assumption we have $|S_i|_i <1$ on $X'$). We may thus isolate $\varphi'_{l,v,\e}(p)$ from the LHS of \eqref{upperbound1} to obtain the estimate}
\color{black}{}
$$\varphi'_{l,v,\e}(p)\leq\sup_{\mathbb{T}^{1}}\varphi'_{l,v,\e}+2C_{0}.$$
Similarly for the $i$-torus $\mathbb{T}^{i}:=\{|z_{k_{1}}|=r_{1},\cdots,|z_{k_{i}}|=r_{i}\}\times\{z'_{n-i}\},$ with $i<n',$
it follows that $$\sup_{\mathbb{T}^{i}}\varphi'_{l,v,\e}\leq\sup_{\mathbb{T}^{i+1}}\varphi'_{l,v,\e}+2C_{0}.$$
By induction the claim follows.
\end{proof}

\subsubsection{Lower bound on $\varphi'_{l,v,\e}$}

The uniform upper bound established above already makes the family $\{ \varphi'_{l,v,\e}\}$ a pre-compact set in
the class of quasi-PSH functions. However we still need a lower bound which not only guarantees the solution will not tend to $-\infty$ on the whole space, but also controls the behaviour of the solution near the exceptional divisors. We will follow the ideas in \cite{CLS} combined with Tsuji's trick of applying Kodaira's Lemma (Lemma \ref{lem-kod}) as in \cite{ST3,TZ,Ts}.

By Definition \ref{lem-zero Lelong}, we can see that the initial potential in \eqref{eq:ma flow 4} satisfies
\begin{equation}\label{eq:initial lb}
\varphi'_{l,v,\e}(0)\geq\de(\sum_{i}\log|S_{i}|_{i}^{2}+\sum_{j}\log|S_{j}|_{j}^{2}+\sum_{k}\log|S_{k}|_{k}^{2})+C.
\end{equation}
for some constant  $C$ depending on $\delta$.  We now show

\begin{theorem}\label{thm-lb1}
There exists a constant $C$ depending on $\de$ and  $T'$  such that on $(X'\setminus\tilde{E})\times[0,T'],$ we have
\begin{equation}\label{eq:lb1}
\varphi'_{l,v,\e}\geq\de\log|\tilde{S}_{\tilde{E}}|^{2}+C.
\end{equation}
\end{theorem}
\begin{proof}
We adapt the idea from \cite{CLS}.  In the following, the $C_i's$ will denote some constants depending only on $\delta$ and $T'$.  By the definition of $\tilde{E}$ in Assumption \ref{mainass}  \eqref{eq:initial lb} implies
\begin{equation}\label{eq:initial lb1}
\varphi'_{l,v,\e}(0)\geq\frac{\de}{2}\log|\tilde{S}_{\tilde{E}}|^{2}+C_{1}
\end{equation}
for some $C_{1}$.  Now we write: \begin{align}\label{eq:metric}
&\dr'_{t,l,v,\e}+\ddb\varphi'_{l,v,\e}\nonumber\\
=&\dr'_{t,l,v,\e}\ddb\de\log|\tilde{S}_{\tilde{E}}|^{2} +\ddb(\varphi'_{l, v,\e}-\de\log|\tilde{S}_{\tilde{E}}|^{2})\nonumber\\
:=\;&\dr'_{t,l,v,\e,\de}+\ddb(\varphi'_{l,v,\e}-\de\log|\tilde{S}_{\tilde{E}}|^{2}).
\end{align}

As $\varphi'_{l,v,\e}$ is a bounded solution to \eqref{eq:ma flow 4}, we see that for any $l\geq 1$ and $v,\e>0$ and $t\in[0,T']$ the function $\varphi'_{l,v,\e}-\de\log|\tilde{S}_{\tilde{E}}|^{2} (x, t)\to \infty$ as $x\to \tilde{E}$, and thus for any fixed time $t\in [0, T']$ it attains a minimum value at some point $p_t$ away from $\tilde{E}$ where by the maximum principle we have
\begin{equation}\label{EEEE1}
\ddb(\varphi'_{l,v,\e}-\de\log|\tilde{S}_{\tilde{E}}|^{2})(p_t, t)\geq 0.
\end{equation}
\color{black}{Substituting \eqref{eq:metric} in \eqref{eq:ma flow 4}, it follows that at $(p_t, t)$ we have
\begin{align}\label{eq:lb2}
\dt(\varphi'_{l,v,\e}-\de\log|\tilde{S}_{\tilde{E}}|^{2})&\geq\log\frac{\dr'^{n}_{t,l,v,\e,\de}
\prod_{i}|S_{i}|_{i}^{2}\log^{2}|S_{i}|_{i}^{2}\prod_{j}(|S_{j}|_{j}^{2}+\e^{2})^{b_{j}}}{\Omega'\prod_{k}(|S_{k}|_{k}^{2}+\e^{2})^{a_{k}}}\nonumber\\
&\geq\log\frac{\dr'^{n}_{t,l,v,\e,\de}
\prod_{i}|S_{i}|_{i}^{2}\log^{2}|S_{i}|_{i}^{2}\prod_{j}(|S_{j}|_{j}^{2}+\e^{2})^{b_{j}}}{\Omega'}-C_2,
\end{align}
 as $\prod_{k}(|S_{k}|_{k}^{2}+\e^{2})^{a_{k}}$ is uniformly bounded from above.
On the other hand, by Assumption \ref{changinghermitianmetricsass} (iii) we have

\begin{equation}\label{eq:vol}
\dr'^{n}_{t,l,v,\e,\de}\geq\frac{C_3(v+t)^{n}\Omega'}
{\prod_{i}|S_{i}|_{i}^{2}\log^{2}|S_{i}|_{i}^{2}\prod_{j}(|S_{j}|_{j}^{2}+\e^{2})^{b_{j}}}.
\end{equation}
 Combining \eqref{eq:lb2} and \eqref{eq:vol} gives the following at $(p_t, t)$:
$$\dt(\varphi'_{l,v,\e}-\de\log|\tilde{S}_{\tilde{E}}|^{2})\geq-C_4+n\log(v+t).$$
From this and the initial lower bound \eqref{eq:initial lb1}, we may conclude the theorem by a maximum principle argument as in the proof of Lemma \ref{lem-upper 1}.
}\end{proof}

This theorem informs us that $\varphi'_{l,v,\e}$ is uniformly locally bounded away from $\tilde{E}.$ In \S 4.4  we will further show that the corresponding solution to \eqref{eq:ma flow 1} is uniformly locally bounded away from the log canonical locus.

To establish the existence of the solution to \eqref{eq:ma flow 1}, or equivalently \eqref{eq:ma flow 2} from the compactness of solutions to \eqref{eq:ma flow 4}, we still need to derive  uniform high order estimates for solutions to \eqref{eq:ma flow 4} which we do in the following sub subsections. The main idea will be similar to \cite{CLS,ST3}.

\subsubsection{Upper and lower bounds on $\dot\varphi'_{l,v,\e}$ }

We begin with the following estimate for the time derivative of $\varphi'_{l,v,\e}:$
\begin{lemma}\label{lem-t-dev}
There exist $c_1 ,c_2>0$ depending only on $\de, T'$ such that on $(X'\setminus\tilde{E})\times(0,T'],$ we have
\begin{equation}\label{eq:t-dev bd}
n\log t+\de\log|\tilde{S}_{\tilde{E}}|^{2}-c_1\leq\dot{\varphi}'_{l,v,\e}\leq n+\frac{c_2-\de\log|\tilde{S}_{\tilde{E}}|^{2}}{t}.
\end{equation}
\end{lemma}
\begin{proof}
Denote the evolving metric in \eqref{eq:ma flow 4} as $\dr(t)$ for simplicity and take the time derivative of \eqref{eq:ma flow 4} to give
\begin{equation}\label{eq:t-dev}
\dt\dot{\varphi}'_{l,v,\e}=\Delta_{\dr}\dot{\varphi}'_{l,v,\e}+tr_{\dr}(\chi-\sum_{i}\ddb\log\log^{2}|S_{i}|_{i}^{2}).
\end{equation}
Thus we have
$$(\dt-\Delta_{\dr})(t\dot{\varphi}'_{l,v,\e}-(\varphi'_{l,v,\e}+\de\log|\tilde{S}_{\tilde{E}}|^{2})-nt)
=-tr_{\dr}(\dr'_{t,l,v,\e}(0)-\de\Theta_{\tilde{E}})\leq 0.$$
where the last inequality holds by Assumption \ref{changinghermitianmetricsass} (iii).  Note that $H^{+}(p,t):=t\dot{\varphi}'_{l,v,\e}-(\varphi'_{l,v,\e}-\de\log|\tilde{S}_{\tilde{E}}|^{2})-nt\to-\infty$ when $p\in X'\setminus \tilde{E}$ approaches $\tilde{E},$ and thus by the maximum principle we have
$$\dot{\varphi}'_{l,v,\e}\leq n+\frac{\varphi'_{l,v,\e}-\de\log|\tilde{S}_{\tilde{E}}|^{2}+C_1}{t}\leq n+\frac{C_{2}-\de\log|\tilde{S}_{\tilde{E}}|^{2}}{t},$$
where above and in what follows, $C_i$ will denote some constant depending only on $\de, T'$.  we have used Theorem \ref{thm:upper-2}.

For the lower bound, by \eqref{eq:t-dev}, \eqref{CGexpansion} and \eqref{eq:vol}, for $A\gg 1$ dependong on $\de, T'$ we have the following inequality
\begin{align}\label{eq:lb-heat}
&(\dt-\Delta_{\dr})\left(\dot{\varphi}'_{l,v,\e}+A(\varphi'_{l,v,\e}-\de\log|\tilde{S}_{\tilde{E}}|^{2})-n\log t\right)\nonumber\\
=\;&tr_{\dr}(\chi-\sum_{i}\ddb\log\log^{2}|S_{i}|_{i}^{2})+A\log\frac{\dr^{n}\prod_{i}|S_{i}|_{i}^{2}\log^{2}|S_{i}|_{i}^{2}
\prod_{j}(|S_{j}|_{j}^{2}+\e^{2})^{b_{j}}}{\Omega'\prod_{k}(|S_{k}|_{k}^{2}+\e^{2})^{a_{k}}}\nonumber\\
&+\;Atr_{\dr}(\dr'_{t,l,v,\e}-\de\Theta_{\tilde{E}})-An-\frac{n}{t}\nonumber\\
\geq\;&tr_{\dr}(\chi-\sum_{i}\ddb\log\log^{2}|S_{i}|_{i}^{2})-A\sum_{k}a_{k}\log(|S_{k}|_{k}^{2}+\e^{2})-An-\frac{n}{t}\nonumber\\
&+\;A\log\frac{C_3(v+t)^{n}\dr^{n}}{\dr'^{n}_{t,l,v,\e,\de}}+Atr_{\dr}\dr'_{t,l,v,\e,\de}\nonumber\\
\geq\;&tr_{\dr}\left((A-1)\dr'_{t,l,v,\e,\de}+(\chi-\sum_{i}\ddb\log\log^{2}|S_{i}|_{i}^{2})\right)-\frac{C_4}{t}\nonumber\\
\geq\;&\frac{A}{2}tr_{\dr}\dr'_{t,l,v, \e,\de}-\frac{C_4}{t}
\geq\frac{An}{2}\left(\frac{\dr'^{n}_{t,l,v, \e,\de}}{\dr^{n}}\right)^{1/n}-\frac{C_4}{t},
\end{align}
where the second inequality follows from the property of logarithmic functions and we have also used  Assumption \ref{changinghermitianmetricsass} (iii).

 Recall that $\varphi'_{l,v,\e}$ itself is a smooth and bounded
solution to \eqref{eq:ma flow 4}, thus when $t\to 0$ or the point $p\in X'\setminus\tilde{E}$ approaches $\tilde{E},$ the function
$H^{-}(p,t):=\dot{\varphi}'_{}+A(\varphi'_{l,v,\e}-\de\log|\tilde{S}_{\tilde{E}}|^{2})-n\log t\to+\infty.$ Thus its infimum over
$(X'\setminus\tilde{E})\times(0,T']$ will be attained at some $(p_{0},t_{0})$ where $t_0\neq 0$ and  $p_0\notin \tilde{E},$ and by the maximum principle it follows from \eqref{eq:lb-heat} that at $(p_{0},t_{0})$ we have
\begin{equation}\label{eq:vol-lb}
\dr^{n}\geq C_5 t^{n}\dr'^{n}_{t,l,v,\e,\de}.
\end{equation}
Thus we have
\begin{align}\label{eq:H-lb}
&\dot{\varphi}'_{l,v,\e}+A(\varphi'_{l,v,\e}-\de\log|\tilde{S}_{\tilde{E}}|^{2})-n\log t\nonumber\\
\geq\;&\inf H^{-}(p,t)=H^{-}(p_{0},t_{0})\nonumber\\
\geq\;&\log\frac{\dr^{n}}{\dr'^{n}_{t,l,v,\e,\de}}(p_{0},t_{0})+n\log(v+t_{0})-n\log t_{0}+C_6
\geq\; C_7
\end{align}
\color{black}{where in the second inequality we have used \eqref{eq:ma flow 4}, \eqref{eq:vol-lb} and Assumption \ref{changinghermitianmetricsass} (iii)}\color{black}{}.  By Theorem \ref{thm:upper-2} and \ref{thm-lb1} we conclude the lower bound estimate of $\dot{\varphi}'_{l,v,\e}.$
\end{proof}

\subsubsection{Upper bound on Laplacian of $ \varphi'_{l,v,\e}$}

Next we will prove the following Laplacian estimate for the solutions:
\begin{theorem}\label{thm-C2}
There exist constants $c_{1},c_{2}>0$ depending only on $\de$ and $T'$ such that on $(X'\setminus\tilde{E})\times(0,T']$ we have

\begin{equation}\label{eq:C2}
c_{3} t^n  |\tilde{S}_{\tilde{E}}|^{(c_2/t)+\de}\prod_{k}(|S_{k}|_{k}^{2}+\e^{2})^{a_{k}}e^{-\frac{c_{1}}{t}}\leq tr_{\hat{\dr}}\dr\leq\frac{1}{|\tilde{S}_{\tilde{E}}|^{c_2/t}}e^{\frac{c_{1}}{t}}.
\end{equation}

where $$\dr:=\;\dr'_{t,l,v,\e}+\ddb\varphi'_{l,v,\e}$$

$$\hat{\dr}:=\theta-\sum_{i}\ddb\log\log^{2}|S_{i}|_{i}^{2}+\eta\sum_{j}\ddb\mathcal{F}(|S_{j}|_{j}^{2},1-b_{j},\e^{2}).$$
\end{theorem}
We begin with the following Lemma on the background form $\hat{\dr}$ \color{black}{which is positive by Assumption \ref{changinghermitianmetricsass} (i)} \color{black}{}.

\begin{lemma}\label{lem-bisec lb}
There exist a sufficiently small constant $\rho\in(0,1)$ and constants $C_{1},C_{2}>0$ such that
\begin{equation}\label{eq:bisec lb}
Bisec(\hat{\dr})\geq-(C_{1}\ddb\Psi_{\rho}+C_{2}\hat{\dr})\otimes\hat{\dr},
\end{equation}
where $\Psi_{\rho}:=\sum_{j}\mathcal{F}(|S_{j}|^{2},\rho,\e^{2})$ is PSH with $\mathcal{F}$ defined in \eqref{conic-regg}.
\end{lemma}
The proof of this lemma is almost identical to the proof of (4.3) in \cite{GP} on p31-32. The only difference is that in \cite{GP} the term $-\sum_{i}\ddb\log\log^{2}|S_{i}|_{i}^{2}$ is omitted from $\hat{\omega}$ thus making it smooth across the log canonical divisor $\bigcup_i D_i$, while in our case $\hat{\omega}$ above has cusp like singularities at $\bigcup_i D_i$.  However, as $\hat{\dr}$ is a standard Carlson-Griffiths form as in Lemma \ref{CGlemma}, there exist quasi-coordinates near those divisors $D_i$ such that the metric has bounded geometry in these coordinates  \cite{Koba,TY}, and making use of the quasi-coordinates we can adapt the construction in \cite{GP} to our case and conclude the lemma.

\begin{proof}[proof of Theorem \ref{thm-C2}]
This estimate highly depends on the properties of $\hat{\omega}$.  This is similar to the role played by the approximate conic metrics from \cite{GP} in the conical \ka-Ricci flow \cite{LiuZ,Shen}.
Recall that the approximate \MA flow equation \eqref{eq:ma flow 4} corresponds to the twisted \ka-Ricci flow:
\begin{align}\label{eq:krf-twisted}
\dt\dr=&-Ric(\dr)+2\pi\sum_{i}[D_{i}]+\sum_{j}b_{j}\ddb\log\frac{|S_{j}|_{j}^{2}+\e^{2}}{|S_j|_j^2}\nonumber\\
&-\sum_{k}a_{k}\ddb\log\frac{|S_{k}|_{k}^{2}+\e^{2}}{|S_k|_k^2}.
\end{align}
Thus at any point $p$, after choosing holomorphic coordinates in which $\hat{g}_{i\bar{j}}|_{p}=\de_{ij}$ and $g_{i\bar{j}}|_{p}=\lambda_{i}\de_{ij}$ (correspond to $\hat{\dr}$ and $\dr$) the corresponding parabolic Aubin-Yau Inequality (the elliptic version appeared in \cite{Au,Yau2}) takes the form
\begin{align}\label{eq:aubinyau1}
&(\dt-\Delta)\log tr_{\hat{\dr}}\dr\nonumber\\ \leq&\frac{1}{tr_{\hat{\dr}}\dr}\left(-\sum_{p,q}\frac{\lambda_{p}}{\lambda_{q}}R_{p\bar{p}q\bar{q}}(\hat{\dr})
+tr_{\hat{\dr}}(\sum_{j}b_{j}\ddb\log\frac{|S_{j}|_{j}^{2}+\e^{2}}{|S_j|_j^2}-\sum_{k}a_{k}\ddb\log\frac{|S_{k}|_{k}^{2}+\e^{2}}{|S_k|_k^2})\right)\nonumber\\
=&-\frac{1}{tr_{\hat{\dr}}\dr}\sum_{\color{black}{p\leq q}}(\frac{\lambda_{p}}{\lambda_{q}}+\frac{\lambda_{q}}{\lambda_{p}}-2)R_{p\bar{p}q\bar{q}}(\hat{\dr})\nonumber\\
&+\frac{1}{tr_{\hat{\dr}}\dr}tr_{\hat{\dr}}(\ddb\log\frac{\prod_{j}(|S_{j}|_{j}^{2}+\e^{2})^{b_{j}}\hat{\dr}^{n}}{\Omega'}-Ric(\Omega')
+\sum_{j}b_{j}\Theta_{j})\nonumber\\&-\frac{1}{tr_{\hat{\dr}}\dr}\sum_{k}a_{k}tr_{\hat{\dr}}(\frac{\e^{2}}{|S_{k}|_{k}^{2}+\e^{2}}\Theta_{k}
+\frac{\e^{2}DS_{k}\wedge\overline{DS}_{k}}{(|S_{k}|_{k}^{2}+\e^{2})^{2}}),
\end{align}
where $\Delta$ represents the Laplacian with respect to the evolving metric $\dr.$ Note that from \cite{GP} although the bisectional curvature of $\hat{\dr}$ is not uniformly bounded from below, we can still construct a bounded auxiliary function to derive the trace estimate.

By Lemma \ref{lem-bisec lb}, the first term in \eqref{eq:aubinyau1} can be controlled as following:
\begin{align}\label{eq:control 1}
&-\frac{1}{tr_{\hat{\dr}}\dr}\sum_{\color{black}{p\leq q}}(\frac{\lambda_{p}}{\lambda_{q}}+\frac{\lambda_{q}}{\lambda_{p}}-2)R_{p\bar{p}q\bar{q}}(\hat{\dr})\nonumber\\
\leq&\frac{1}{\sum_{p}\lambda_{p}}\sum_{\color{black}{p< q}}\left(\frac{\lambda_{p}}{\lambda_{q}}(C_{1}\Psi_{\rho,q\bar{q}}+C_{2})+
\frac{\lambda_{q}}{\lambda_{p}}(C_{1}\Psi_{\rho,p\bar{p}}+C_{2})\right)\leq C_{3}\Delta\Psi_{\rho}+C_{4}tr_{\dr}\hat{\dr}.
\end{align}
\color{black}{ where above, and in what follows, $C_i$ will denote a constant depending only on $\de, T'$ and where in the last inequality we have used the fact that $\partial \overline{\partial}\Psi_{\rho} \geq -C \widehat{\omega}$ for a constant $C$ independent of $\rho$ by Lemma \ref{conicalapprox}.}\color{black} {} For the second term, we first note that $-Ric(\Omega')+\sum_{j}b_{j}\Theta_{j}$ is uniformly bounded with respect to $\hat{\dr}.$ Next, note that
$\hat{\dr}$ has similar expansion as \eqref{eq:bg-expansion}, which implies that
\begin{align}\label{eq:control 2}
\ddb\log\frac{\prod_{j}(|S_{j}|_{j}^{2}+\e^{2})^{b_{j}}\hat{\dr}^{n}}{\Omega'}&=-\sum_{i}\ddb(\log|S_{i}|_{i}^{2}\log^{2}|S_{i}|_{i}^{2})+\ddb\log H\nonumber\\&=\sum_{i}(\Theta_{i}-\ddb\log\log^{2}|S_{i}|_{i}^{2})+\ddb\log H,
\end{align}
where by direct computations $$H=f_{0}+f_{1}(\sum_{i}|S_{i}|_{i}^{2}\log|S_{i}|_{i}^{2}+\sum_{j}(|S_{j}|_{j}^{2}+\e^{2})^{b_{j}})+\cdots$$
where $f_{0},f_{1},\cdots$ are smooth and bounded functions with $f_{0}>0$. Thus the second term of \eqref{eq:aubinyau1} is also uniformly bounded with respect to $\hat{\dr}.$ Finally, as $\Theta_{k}$ is uniformly bounded with respect to $\hat{\dr},$ the last term obviously is bounded from above with respect to $\hat{\dr}.$ Combine those arguments with \eqref{eq:aubinyau1} we have
\begin{equation}\label{eq:aubinyau2}
(\dt-\Delta)\log tr_{\hat{\dr}}\dr\leq C_{3}\Delta\Psi_{\rho}+C_{4}tr_{\dr}\hat{\dr}+\frac{C_{5}}{tr_{\hat{\dr}}\dr}\leq C_{3}\Delta\Psi_{\rho}+C_{6}tr_{\dr}\hat{\dr}.
\end{equation}
Now we consider the function $G(p,t):=t\log tr_{\hat{\dr}}\dr-A(\varphi'_{l,v,\e}-\de\log|\tilde{S}_{\tilde{E}}|^{2})+B\Psi_{\rho},$ it follows that
\begin{align}\label{eq:aubinyau3}
(\dt-\Delta)G\leq&\log tr_{\hat{\dr}}\dr+C_{6}t\;tr_{\dr}\hat{\dr}+A(-\dot{\varphi}'_{l,v,\e}
+n-tr_{\dr}(\dr'_{t,l,v,\e}-\de\Theta_{\tilde{E}}))\nonumber\\
&+(C_{3}t-B)\Delta\Psi_{\rho}\nonumber\\
\leq&\log tr_{\hat{\dr}}\dr-A\log\frac{\dr^{n}}{\hat{\dr}^{n}}+A\sum_{k}a_{k}\log(|S_{k}|_{k}^{2}+\e^{2})+(C_{3}t-B)\Delta\Psi_{\rho}\nonumber\\
&+tr_{\dr}(C_{6}t\hat{\dr}-A\dr'_{t,l,v,\e,\de})+C_7\nonumber\\
\leq&tr_{\dr}((1+C_{6}t)\hat{\dr}-A\dr'_{t,l,v,\e,\de})+(C_{3}t-B)\Delta\Psi_{\rho}+C_8
\end{align}
where the last inequality follows from the property of logarithmic function again and $C_8$ is sufficiently large depending on $A$ and $\delta, T'$. Considering that $\Psi_{\rho}$ is \color{black}{$\theta$}\color{black}{}-PSH, for $t\leq T'$ we can choose $B\geq C_{3}T'$ so that $(C_{3}t-B)\Delta\Psi_{\rho}\leq 0.$ On the other hand, by Assumption \ref{changinghermitianmetricsass} \color{black}{(iii)}\color{black}{}  it follows that $\dr'_{t,l,v,\e,\de}\geq c_{\de}\hat{\dr}$ for some $c_{\de }>0.$ Thus by choosing large enough $A$ we have
$$(\dt-\Delta)G\leq-tr_{\dr}\hat{\dr}+C_{9}.$$  Note that $G$ is bounded from above at $t=0$ by assumptions and tends to $-\infty$ near $\tilde{E}$ so we may assume the supremum of $G$ is attained at $(p_{0},t_{0})$ for $p_{0}\in X'\setminus\tilde{E}$ and $t_{0}>0,$ thus by the maximum principle it follows from above that $tr_{\dr}\hat{\dr}(p_{0},t_{0})\leq C_{9}.$ Then at $(p_{0},t_{0})$ 
it follows that
\begin{align}\label{eq:tr-max}
G(p_{0},t_{0})&\leq t_{0}\log(tr_{\dr}\hat{\dr})^{n-1}(\frac{\dr^{n}}{\hat{\dr}^{n}})-A(\varphi'_{l,v,\e}-\de\log|\tilde{S}_{\tilde{E}}|^{2})+B\Psi_{\rho}\nonumber\\
&\leq C_{9}+t_{0}(\dot{\varphi}'_{l,v,\e}+\sum_{k}a_{k}\log(|S_{k}|_{k}^{2}+\e^{2})+C_{10})\leq C_{11},
\end{align}
where the last inequality follows from Lemma \ref{lem-t-dev}. Thus for any $(p,t)\in(X'\setminus\tilde{E})\times(0,T']$ it follows that
$$G(p,t):=t\log tr_{\hat{\dr}}\dr-A(\varphi'_{l,v,\e}-\de\log|\tilde{S}_{\tilde{E}}|^{2})+B\Psi_{\rho}\leq C_{11},$$ which together with Lemma \ref{lem-t-dev} implies the upper bound in \eqref{eq:C2} of the Theorem, namely
\begin{equation}\label{eq:C2-upper}
tr_{\hat{\dr}}\dr\leq  \frac{e^{c_1 /t}}{|\tilde{S}_{\tilde{E}}|^{c_2/t}}\end{equation}
for constants $c_1, c_2$ depending only on $\de, T'$.  On the other hand, as
\begin{equation}\label{eq:det}
\frac{\dr^{n}}{\hat{\dr}^{n}}\geq\prod_{k}(|S_{k}|_{k}^{2}+\e^{2})^{a_{k}}e^{\dot{\varphi}'_{l,v,\e}-C_{12}},
\end{equation}
it follows from \eqref{eq:t-dev bd} and \eqref{eq:C2-upper} that
\begin{equation}\label{eq:C2-lower}
c_{3} t^n  |\tilde{S}_{\tilde{E}}|^{(c_2/t)+\de}\prod_{k}(|S_{k}|_{k}^{2}+\e^{2})^{a_{k}}e^{-\frac{c_{1}}{t}}\leq tr_{\hat{\dr}}\dr
\end{equation}
for a constant $c_3$ depending only on $\de, T'$ which completes the Laplacian estimate \eqref{eq:C2} and the proof of the Theorem.
\end{proof}
\begin{remark}\label{rem-C2}
It is possible to eliminate the factor $\prod_{k}(|S_{k}|_{k}^{2}+\e^{2})^{a_{k}}$ in \eqref{eq:C2} if we can replace the conic approximation method
in \cite{GP} by other approximation using the upper bound of the bisectional curvature. Then we may use Chern-Lu Inequality to get rid of that factor. We may consider this approximation in the future.
\end{remark}

\subsubsection{higher order derivative estimates on $ \varphi'_{l,v,\e}$}

By the bounds in Lemma \ref{lem-t-dev} and Theorem \ref{thm-C2} we may have

\begin{theorem}\label{thm-localequivalence}
For any compact set $K\subset (X'\setminus\tilde{E})$ there exist constants $C(m,K,T')$ such that on $K\times(0,T']$ we have
\begin{equation}\label{eq:-localequivalence}
 C(m,K,T')^{-1}\theta \leq \dr'_{t,l,v,\e}+\ddb\varphi'_{l,v,\e} \leq C(m,K,T') \theta
\end{equation}

\end{theorem}

Finally, by standard Evans-Krylov estimates (cf. \cite{ST3,Yau1}), we can establish the local high order estimates in any compact set $K\subset X'\setminus\tilde{E}:$
\begin{theorem}\label{thm-high order}
For any compact set $K\subset (X'\setminus\tilde{E})\times(0,T']$ there exist constants $C(m,K,T')$ such that
\begin{equation}\label{eq:high order}
|\varphi'_{l,v,\e}|_{C^{m}(K)}\leq C(m,K,T').
\end{equation}
\end{theorem}

\subsubsection{Estimates when $\varphi_{0}\in C^{\infty}(X')$}
In the following subsubsection we assume in addition that $\varphi_{0}\in C^{\infty}(X')$.  Note that by absorbing the smooth form $\ddb \pi^* \varphi_{0}$ into $\pi^* \omega_0$, we will assume, without loss of generality, that in fact $\varphi_{0}=0$.
Thus we simply take $\varphi_{0, l}=0$ for all $l$ in the previous subsections.  In this case we will establish estimates for $\varphi'_{l,v,\e}$ which are uniform on $K\times[0, T']$ for any $K\subset \subset X'\setminus\tilde{E}$.

\begin{lemma}\label{lem-upper 1new}[modified Lemma \ref{lem-upper 1}]
 There exists a constant $C$ depending only on $\delta, T'$ such that
\begin{equation}\label{eq:upper-1new}
\sup\varphi'_{l,v,\e}\leq t(C-\sum_{k}a_{k}\log(|S_{k}|_{k}^{2}+\e^{2})).
\end{equation}
on $(X'\setminus \tilde{E}) \times[0,T']$.
\end{lemma}

\begin{proof}
This basically follows from a slight variant on the proof of Lemma \ref{lem-upper 1}.  Notice that the RHS in \eqref{eq:ma flow 5} is bounded above uniformly on $(X'\setminus \tilde{E}) \times[0, T']$.   On the other hand, since $\varphi'_{l,v,\e}(0)=0$ we may also have $\sup_M \phi_{l,v,\e}(0) \leq 0$ in that proof and it follows from the maximum principle that $\phi_{l,v,\e}\leq Ct$ there.  The Lemma then follows from the fact that $\varphi_{l,v,\e}:=\phi_{l,v,\e}-t \sum_{k}a_{k}\log(|S_{k}|_{k}^{2}+\e^{2})$ as in the proof of Lemma \ref{lem-upper 1}.

\end{proof}

Next we want to modify the upper and lower bound in Lemma \ref{lem-t-dev}.   We begin with the upper bound

\begin{lemma}\label{lem-t-dev new}[modification of upper bound in Lemma \ref{lem-t-dev}]
There exist $C>0$ depending only on $\de, T'$ such that we have
\begin{equation}\label{eq:t-dev bdnew}
\dot{\varphi}'_{l,v,\e}\leq (C-\sum_{k}a_{k}\log(|S_{k}|_{k}^{2}+\e^{2}))+n-\de\log|\tilde{S}_{\tilde{E}}|^{2}
\end{equation}
on $(X'\setminus \tilde{E})  \times[0,T']$.
\end{lemma}

\begin{proof}

Define \begin{align}\psi_{l,v,\e}&:= \varphi'_{l,v,\e}+\eta\sum_{j}\mathcal{F}(|S_{j}|_{j}^{2},1-b_{j},\e^{2})\\ \sigma_{t,l,v,\e}&:=\dr'_{t, l,v,\e}- \eta\ddb \sum_{j}\mathcal{F}(|S_{j}|_{j}^{2},1-b_{j},\e^{2})\nonumber\end{align}

Then $\psi_{l,v,\e}$ solves \eqref{eq:ma flow 4} but with initial condition $\psi_{l,v,\e}(0)=0$ and with $\dr'_{t, l,v,\e}$ replaced with $\sigma_{t, l,v,\e}$ which we will denote below simply as $\sigma$.  We have
\begin{equation}\label{eq:t-devnew}
\dt\dot{\psi}'_{l,v,\e}=\Delta_{\sigma}\dot{\psi}'_{l,v,\e}+tr_{\sigma}(\chi-\sum_{i}\ddb\log\log^{2}|S_{i}|_{i}^{2}).
\end{equation}
Thus we have the following on $(X'\setminus \tilde{E}) \times[0, T']$  assuming $v$ is sufficiently small 
\begin{equation}\begin{split}
(\dt-\Delta_{\sigma})(t\dot{\psi}'_{l,v,\e}&-(\psi'_{l,v,\e}-t\de\log|\tilde{S}_{\tilde{E}}|^{2})-nt)
\\
&=-tr_{\sigma}(\sigma_{t,l,v,\e} + t \de \ddb \log|\tilde{S}_{\tilde{E}}|^{2})+\de\log|\tilde{S}_{\tilde{E}}|^{2} \leq 0.\end{split}\end{equation}
where the last inequality follows by Assumption \ref{changinghermitianmetricsass} \color{black}{(ii)}\color{black}{} and the assumption $\log|\tilde{S}_{\tilde{E}}|^{2} <0$.

Now note that the quantity $Q:=t\dot{\psi}'_{l,v,\e}-(\psi'_{l,v,\e}-t\de\log|\tilde{S}_{\tilde{E}}|^{2})-nt$  is  bounded above by $0$ at $t=0$.   Also, the quantity approaches $-\infty$ towards the divisors at all positive times.  Thus by a maximum principle argument we may conclude that $Q(x, t) \leq 0$ on $(X'\setminus \tilde{E}) \times[0, T']$ and thus
$$\dot{\psi}'_{l,v,\e}\leq \frac{(\psi'_{l,v,\e}-t\de\log|\tilde{S}_{\tilde{E}}|^{2})+nt)}{t} \leq
(C_1-\sum_{k}a_{k}\log(|S_{k}|_{k}^{2}+\e^{2}))+n-\de\log|\tilde{S}_{\tilde{E}}|^{2}$$
for some constant $C_1$ depending only on $\delta, T'$ where in the last inequality we have used Lemma \ref{lem-upper 1new} together with the definition of $\psi_{l,v,\e}$.   The Lemma then follows as $\dot\psi_{l,v,\e}=\dot{\varphi}'_{l,v,\e}$.
\end{proof}

Next we modify the lower bound in Lemma \ref{lem-t-dev}.

\begin{lemma}\label{lem-t-dev-new}[modification of lower bound in Lemma \ref{lem-t-dev}]
There exist $C>0$ depending only on $\de, T'$  such that 
\begin{equation}\label{eq:t-dev bdnewnew}
\de\log|\tilde{S}_{\tilde{E}}|^{2}-C \leq\dot{\varphi}'_{l,v,\e}
\end{equation}
on $(X'\setminus \tilde{E}) \times[0, T']$
\end{lemma}
\begin{proof}


The proof is basically the same as the proof of Lemma \ref{lem-t-dev} except that we evolve instead the quantity $Q:=\left(\dot{\varphi}'_{l,v,\e}+A(\varphi'_{l,v,\e}-\de\log|\tilde{S}_{\tilde{E}}|^{2})\right)$.
Note that by (3.8) we have $Q(x, 0) \geq C_\delta$ on $\tilde X$ for some constant $C_{\delta}$.   We also see that for all $t\in [0, T']$ we have $Q(x, t) \to \infty$ as $x$ approaches the divisor $\tilde{E}$.   Thus $Q$ attains a minimum on $(X'\setminus \tilde{E}) \times[0, T']$ at  some point $(x_0, t_0)$.  If $t_0=0$ then the Lemma follows from the previous observation.  If $t_0>0$ then similarly to the proof of Lemma \ref{lem-t-dev}, we have the following at $(x_0, t_0)$ where $A$ is a sufficiently large constant depending only on $T'$ and $\delta$ and $C_i$'s are some constants depending only on $T', \delta$ and where $\hat\omega$ is as in the statement of Theorem 3.7:

\begin{align}\label{eq:lb-heatnew}
0\geq &(\dt-\Delta_{\dr})\left(\dot{\varphi}'_{l,v,\e}+A(\varphi'_{l,v,\e}-\de\log|\tilde{S}_{\tilde{E}}|^{2})\right)\nonumber\\
=\;&tr_{\dr}(\chi-\sum_{i}\ddb\log\log^{2}|S_{i}|_{i}^{2})+A\log\frac{\dr^{n}\prod_{i}|S_{i}|_{i}^{2}\log^{2}|S_{i}|_{i}^{2}
\prod_{j}(|S_{j}|_{j}^{2}+\e^{2})^{b_{j}}}{\Omega'\prod_{k}(|S_{k}|_{k}^{2}+\e^{2})^{a_{k}}}\nonumber\\
&+\;Atr_{\dr}(\dr'_{t,l,v,\e}-\de\Theta_{\tilde{E}})-An\nonumber\\
\geq\;&tr_{\dr}(\chi-\sum_{i}\ddb\log\log^{2}|S_{i}|_{i}^{2}+A(\dr'_{t,l,v,\e}-\de\Theta_{\tilde{E}}))+A\log\frac{C_0\dr^{n}}{\hat{\omega}^n}\nonumber\\
&+\;-An-A\sum_{k}a_{k}\log(|S_{k}|_{k}^{2}+\e^{2})\nonumber\\
\geq\;&C_1tr_{\dr}(\hat{\omega})+A\log\frac{C(c_{\delta},\eta)\dr^{n}}{\hat{\omega}^n}-C_2\nonumber\\
\geq\;&C_3 \left( \frac{\hat{\omega}^n}{\dr^n} \right)^{1/n} -C_2\nonumber\\
\end{align}
   Thus we get $\hat{\omega}^n \leq C_4 \dr^n$ at $(x_0, t_0)$ and it follows from (3.8) and the definition of $\hat{\omega}$ in Theorem 3.7, that $\dot{\varphi}'_{l,v,\e}(x_0, t_0)$ and thus $Q(x_0, t_0)$ is bounded below by some constant $C_5$ by Theorem \ref{thm-lb1}.  The lemma then follows from Theorem \ref{thm-lb1}.

\end{proof}

\begin{lemma}\label{lemmtracenew}[modification of upper bound in Theorem \ref{thm-C2}] There exist a function $F(x)$ depending only on $T', \de$ such that  $$tr_{\hat{\dr}}\dr \leq F(x)$$ on $(X'\setminus \tilde{E}) \times[0,T']$ where $\hat{\dr}$ is as in Theorem \ref{thm-C2}.

\end{lemma}

\begin{proof}

The proof is basically as the upper bound proof in Theorem \ref{thm-C2} except that  we consider instead the quantity $$G(p,t):=\log tr_{\hat{\dr}}\dr-A(\varphi'_{l,v,\e}-\de \log|\tilde{S}_{\tilde{E}}|^{2}-\sum_{i}\log\log^{2}|S_{i}|_{i}^{2})+B\Psi_{\rho}$$ for constants $A, B$ to be determined and depending only on $\delta, T'$.    In particular, $G(p,t)$ is bounded above on $X'\setminus \tilde{E}$ for each $t$.  Moreover, by the hypothesis and Theorem 3.5 for sufficiently large $A$ we have $G(p,0)\leq C(A)$ on $X'\setminus \tilde{E}$.

By the same calculations in Lemma \ref{lem-bisec lb} we may choose $A$ and $B$ sufficiently large, so that the following holds where throughout the proof $C_i$ will denote a constant depending only on $\de, T'$. 

\begin{align}\label{eq:aubinyau3new}
(\dt-\Delta)G\leq&C_{1}\;tr_{\dr}\hat{\dr}+A(-\dot{\varphi}'_{l,v,\e}
+n-tr_{\dr}(\dr'_{t,l,v,\e}-\de \Theta_{\tilde{E}})-\ddb\sum_{i}\ddb\log\log^{2}|S_{i}|_{i}^{2})\nonumber\\
&+(C_{2}-B)\Delta\Psi_{\rho}\nonumber\\
\leq&-A\log\frac{\dr^{n}}{\hat{\dr}^{n}}+A\sum_{k}a_{k}\log(|S_{k}|_{k}^{2}+\e^{2})+(C'_{}-B)\Delta\Psi_{\rho}\nonumber\\
&+tr_{\dr}(C_{1}\hat{\dr}-A(\dr'_{t,l,v,\e}-\de \Theta_{\tilde{E}}-\ddb\sum_{i}\ddb\log\log^{2}|S_{i}|_{i}^{2}))\nonumber\\
&+(C_{2}-B)\Delta\Psi_{\rho}+C_3\nonumber\\
\leq&-tr_{\dr}\hat{\dr}+C\nonumber\\
\end{align}
where in the last inequality we have used Assumption \ref{changinghermitianmetricsass} and the definition of $\hat\omega$ in Theorem \eqref{thm-C2}.  Note that for all $t\in [0, T']$ we have $G(x, t) \to -\infty$ as $x$ approaches $\tilde{E}$.   Thus $\max_{(X'\setminus \tilde{E}) \times[0, T']} G(x, t) = G(p_0, t_0)>0$ for some $(p_0, t_0)\in (X'\setminus \tilde{E}) \times [0, T']$.  Now if $t_0=0$ the Lemma holds.  If $t_0 >0$ we have $tr_{\dr}\hat{\dr}\leq C$ at $(p_0, t_0)$ by \eqref{eq:aubinyau3new}, and thus as in the proof of  Lemma \ref{lem-bisec lb} we conclude the following at $(p_0, t_0)$
\begin{align}\label{eq:tr-maxnew2}
G(p_{0},t_{0})&\leq \log(tr_{\dr}\hat{\dr})^{n-1}(\frac{\dr^{n}}{\hat{\dr}^{n}})-A(\varphi'_{l,v,\e}-\de \log|\tilde{S}_{\tilde{E}}|^{2}-\sum_{i}\log\log^{2}|S_{i}|_{i}^{2})+B\Psi_{\rho}\nonumber\\
&\leq C_4(\dot{\varphi}'_{l,v,\e}+\sum_{k}a_{k}\log(|S_{k}|_{k}^{2}+\e^{2}))\nonumber\\
&-[A(\varphi'_{l,v,\e}-\de \log|\tilde{S}_{\tilde{E}}|^{2}-\sum_{i}\log\log^{2}|S_{i}|_{i}^{2})+B\Psi_{\rho}]\\
&=: H_1(p_0, t_0) + H_2(x_0, t_0)\nonumber
\end{align}
where by Lemma \ref{lem-t-dev new} we have $H_1(p, t) \leq -(\de/3)  \log|\tilde{S}_{\tilde{E}}|^{2} +C_5$ on $(X'\setminus \tilde{E}) \times [0, T']$ and by Theorem \ref{thm-lb1} we have $H_2(p, t) \leq (2\de/3)  \log|\tilde{S}_{\tilde{E}}|^{2} +C_6$ on $(X'\setminus \tilde{E}) \times [0, T']$  and thus we may continue the estimate in \eqref{eq:tr-maxnew2} as

\begin{equation}\label{eq:tr-maxnew3}
 H_1(p_0, t_0) + H_2(p_0, t_0)\leq (\de/3)  \log|\tilde{S}_{\tilde{E}}|^{2} +C_5+C_6 \leq C_7
\end{equation}
  Combining \eqref{eq:tr-maxnew2} and \eqref{eq:tr-maxnew3} we conclude that $G\leq C_8$ on $(X'\setminus \tilde{E}) \times [0, T')$ and the Lemma  follows from the definition of $G$ and Lemma \ref{lem-upper 1new}, or alternately Lemma \ref{lem-upper 1}.

\end{proof}

 As in the previous section, we may deduce from Lemmas \ref{lem-t-dev-new} and \ref{lemmtracenew} that for any compact set $K\subset X'\setminus\tilde{E},\;t\in[0,T']$ we have the equivalence $$C^{-1}\theta\leq \omega_{t, l,v,\e}+\ddb \varphi'_{l,v,\e}\leq C \theta$$ where $C>0$ depends only on $T', K$.    Then as in Theorem \ref{thm-high order} we can establish the local high order estimates as in the following
\begin{theorem}\label{thm-high ordernew}
Suppose $\varphi_{0}\in C^{\infty}(X')$.   For any compact set $K\subset X'\setminus\tilde{E}$ there exist constants $C(m,K,T')$ such that
\begin{equation}\label{eq:high order new}
|\varphi'_{l,v,\e}|_{C^{m}(K\times[0,T'])}\leq C(m,K,T').
\end{equation}
\end{theorem}

\section{Proof of Theorem \ref{thm-main1}}
In the following sections we will prove Theorem \ref{thm-main1}.  We begin by fix some choice of Hermitian metrics $h_i, h_j, h_k$ and a constants $\eta>0$.   Now given any $l,\e$, the hypothesis of Theorem \ref{thmLZ} will be satisfied provided $v$ is sufficiently small depending only on $l$, and thus we have a solution $\varphi'_{l,v,\e}(t)$ to \eqref{eq:ma flow 4}  on $(X'\setminus\tilde{E})\times[0,T_0)$.  Moreover,  given any $T'\in (0, T_0)$ and $\de$ as in the beginning of \S 3.2, we may assume without loss of generality that Assumption \ref{changinghermitianmetricsass} there is satisfied, in view of the remarks following the assumption, and thus
on $(X'\setminus\tilde{E})\times[0,T']$ our solution $\varphi'_{l,v,\e}(t)$ satisfies a uniform global upper bound by Theorem \ref{thm:upper-2}, a uniform local lower bound estimate as in Theorem \ref{thm-lb1}, and also uniform local higher order estimates \color{black}{on $X' \setminus \widetilde{E} \times(0, T_0)$}\color{black}{} as in Theorem \ref{thm-high order} where the uniformity is over the parameters $l,v,\e$.

\subsection{Existence of a solution}
  Consider the family of solutions $\varphi'_{l,v,\e}(t)$ to \eqref{eq:ma flow 4}.  It follows from the above estimates and the Arzela-Ascoli theorem that we may let the parameters approach their limits in the following order (and along appropriate subsequences) to obtain a smooth local limit $\varphi'$ solving \eqref{eq:ma flow 2}:
  \begin{equation}\label{conergencetosolution}
  	\begin{split}
  		\varphi'&= \lim_{l\to \infty}\lim_{\e\to 0} \lim_{v\to 0}\varphi'_{l,v,\e}\\
  	\end{split}
  \end{equation}
  where the convergence is locally uniform in $C^{\infty}((X'\setminus\tilde{E})\times   \color{black}{ (0, T_0)})$.    In addition, given any $t\in [0, T_0)$, the family 
  $\varphi'_{l,v,\e}(t)-\ddb\log\log^{2}|S_{i}|_{i}^{2}$ extends to a family of $C\theta$ plurisubharmonic functions on $X'$ for some $C$ independent of the parameters, while also satisfying a uniform global upper bound and a uniform local lower bound away from $\tilde{E}$.  It follows from the classical theory of plurisubharmonic functions (see \cite{De}) that $\varphi'(t)$ extends to a $C\theta$ plurisubharmonic function on $X'$ and that for fixed $t$ the above convergences hold in $L^1(X')$ and thus in the sense of currents on $X'$.  In summary, $\varphi'(t)$ solves \eqref{eq:ma flow 2} smoothly on $X' \setminus \widetilde{E} \times \color{black}{ (0, T_0)}\color{black}{}$ and in the sense of currents on $X' \times \color{black}{ (0, T_0)}$, with zero Lelong number for each $t$, while $\varphi =\varphi' +\eta\sum_{j}|S_{j}|_{j}^{2(1-b_{j})}-t \log \log^{2}|S_{i}|_{i}^{2}$ will solve \eqref{eq:ma flow 1} in a similar sense.

 It will be useful in the following to write $\varphi'$ as 
 
 $$\varphi'= \lim_{l\to \infty}\varphi'_l \,\,\,\,\,\,\,\,\,\,\,\, \text{where} \,\,\,\, \varphi'_{l}:=\lim_{\e\to 0} \lim_{v\to 0}\varphi'_{l,v,\e}.$$

\subsection{Weak continuity at $t=0$}


Consider the solution $\varphi'(t)$ to \eqref{eq:ma flow 2} constructed above.  We have the following weak convergence as $t\to 0$.

\begin{theorem}\label{thm-continuity3}
We have the convergence $\lim\limits_{t\searrow 0}\varphi(t)=\varphi(0)$ in $L^1(X')$, thus in the sense of currents on $X'$.
\end{theorem}

\begin{proof} It will suffice to prove $\lim\limits_{t\searrow 0}\varphi'(t)=\varphi'(0)$ in $L^1(X')$.

 Fix some $0<T'<T_0$.  We may then assume, \color{black}{as explained in the remarks under Assumption \ref{changinghermitianmetricsass}}\color{black}{}, that Assumption \ref{changinghermitianmetricsass} and the estimates in \S 3.2 apply unifromly to each $\varphi'_{l,v,\e}$ and thus to $\varphi'$ on $(X'\setminus\tilde{E})\times(0,T')$.

Fix some $c >0$.  Then for all $t<T'$ the family of functions $\varphi'(t)$ are $C\theta$ plurisubharmonic on $X'$ for some $C$ independent of $t$ while the family is also uniformly bounded above on $X'$ by Theorem \ref{thm:upper-2}, and locally bounded below on $X'\setminus \tilde{E}$ by Theorem \ref{thm-lb1}.  Again, from the classical theory of plurisubharmonic functions, given any sequence $t_m\to 0$, there exists a subsequence which we will continue to denote as $t_m$, so that  $\varphi'(t_{m})\to\psi$ in $L^1(X')$ for some $C\theta$ plurisubharmonic function $\psi$ on $X'$.

 To prove the Theorem we only need to prove that $\psi=\varphi'(0)$ almost everywhere on $X'$.

\begin{bf} Claim \end{bf} $\lim\limits_{t\searrow 0}\varphi_{l}(t)=\varphi_{l}(0)$ in $L^1(X')$, thus in the sense of currents on $X'$.

Note that $\varphi'_{l} \in C^{\infty}((X'\setminus\tilde{E})\times[0,T_0)$ for each $l$ by Theorem \ref{thm-high ordernew}, and the fact that $\varphi_{0, l}\in C^{\infty}(X')$ for every $l$.  Thus the limit in the claim holds at every $x\in X'\setminus \tilde{E}$.  On the other hand, the family of functions $\varphi'(t)$ are $C\theta$ plurisubharmonic on $X'$ for some $C$ independent of $t<T'$, satisfy a global upper bound on $X'$ and a local lower bound on $X'\setminus \tilde{E}$ uniformly over $t<T'$.  The claim then follows by the classical theory of plurisubharmonic functions.  Using the claim and the fact that $\varphi'_{l} \geq \varphi'_{k}$ on $(X' \setminus \tilde{E}) \times [0, T')$ for $l\leq k$ as noted in Remark \ref{finalcd8}, we get

$$\psi(x)=\lim_{m}\varphi'(t_{m},x)\leq\lim_{l} \lim_{m}\varphi'_{l}(t_{m},x)=\lim_{l} \varphi'_{l}(0,x)=\varphi'(0,x)$$
 for almost all $x\in X'\setminus\tilde{E}$. On the   hand, by the time derivative estimate in Lemma \ref{lem-t-dev}, for any $\de>0$ it follows that
\begin{equation}\label{E1}\varphi'_{l}(t_{m},x)-\varphi'_{l}(0,x)\geq\int_{0}^{t_{m}}(n\log t+\de\log|\tilde{S}_{\tilde{E}}|^{2}-C_{1\de})dt.\end{equation}
Thus the limit $\lim_{m}\varphi'_{l}(t_{m},x)\geq\varphi'_{l}(0,x)$ holds for almost all $x\in X'\setminus\tilde{E}$ and by taking $l\to +\infty$ we get  $\psi(x)=\lim_{m}\varphi'(t_{m},x)\geq\varphi'(0,x)$.

 We thus conclude that $\psi(x)=\varphi'(0,x)$ for almost all $x\in X'\setminus\tilde{E}$ which completes the proof of the Proposition.

\end{proof}

\subsection{Uniqueness and maximality}

We now follow the argument in \cite{GZ} to prove uniqueness and maximality of the solution to \eqref{eq:ma flow 1} constructed in \S 4.1.   First we show that when the approximation parameters $v,\e,\e=0$ the corresponding solutions $\varphi'_{l}$ are unique in the category of bounded functions  with the corresponding continuity at $t=0$ as in the following
\begin{lemma}\label{lem-unique1}
For each $l>0$< there exists a unique solution $\varphi'_{l}\in L^{\infty}((X'\setminus\tilde{E})\times[0,T'])\bigcap C^{\infty}((X'\setminus\tilde{E})\times[0,T'])$ which solves the following equation
\begin{equation}\label{eq:ma flow 8}
\left\{
   \begin{array}{ll}
     \displaystyle\dt\varphi'_{l} &=\displaystyle\log\frac{(\dr'_{t,l}+\ddb\varphi'_{l})^{n}\prod_{i}|S_{i}|_{i}^{2}\log^{2}|S_{i}|_{i}^{2}
     \prod_{j}|S_{j}|_{j}^{2b_{j}}}{\Omega'\prod_{k}|S_{k}|_{k}^{2a_{k}}}\\
     \;\\
     \varphi'_{l}(0) &= \varphi_{l,0}-\eta\sum_{j}|S_{j}|_{j}^{2(1-b_{j})},
   \end{array}
 \right.
\end{equation}
and satisfies that $\lim\limits_{t\searrow 0}\varphi'_{l}(t)=\varphi'_{l}(0)$ almost everywhere and in the current sense. Here $\dr'_{t,l}=\dr'_{t,l, v=0, \e=0}$
\end{lemma}
\begin{proof}
We can prove this Lemma using the trick from \cite{ST3}. 
Note that $\varphi'_{l}:=\lim_{\e\to 0}\lim_{v\to 0}\varphi'_{l, v, \e}$ is a particular solution to the above equation satisfying conditions in the Lemma. Suppose there exists another solution $\psi'_{l}$ satisfying all the same conditions in this Lemma. Without loss of generality assume that $|\tilde{S}_{\tilde{E}}|\leq 1$ everywhere, then consider the function $D_{+}:=\psi'_{l}-\varphi'_{l}+\de\log|\tilde{S}_{\tilde{E}}|^{2}$ which satisfies the following equation:
\begin{equation}\label{eq:ma flow 9}
\left\{
   \begin{array}{ll}
     \displaystyle\dt\ D_{+} &=\displaystyle\log\frac{(\dr'_{t,l}+\ddb\varphi'_{l}+\de\Theta_{\tilde{E}}+\ddb D_{+})^{n}}{(\dr'_{t,l}+\ddb\varphi'_{l})^{n}}\\
     \;\\
     D_{+}(0) &= \de\log|\tilde{S}_{\tilde{E}}|^{2}.
   \end{array}
 \right.
\end{equation}
For any time slice $t\in [0,T']$ as $D_{+}\to-\infty$ near $\tilde{E},$ the supremum of $D_{+}$ will always be obtained away from $\tilde{E}$ where
$\psi'_{l},\varphi'_{l}$ are smooth. Also note that we have the equivalence of metrics $$c^{-1} \theta \leq \omega'_{t, l}+\sqrt{-1} \partial \bar{\partial} \tilde{\varphi}'_{l}\leq c \theta$$  for all $t\in [0, T']$ and some constant $c>0$ possibly depending on $l, T'$ (which are fixed in the Lemma) but independent of $\delta$, and where $\theta$ is a fixed \ka\, form on $X'$.  As $\Theta_{\tilde{E}}$ is a smooth form on $X'$, it follows from the above equivalence that for $\delta>0$ sufficiently small we may have 
$$c_{\delta}^{-1}  (\omega'_{t, l}+\sqrt{-1} \partial \bar{\partial} \tilde{\varphi}'_{l}) \leq \omega'_{t, l}+\sqrt{-1} \partial \bar{\partial} \tilde{\varphi}'_{l}+\delta  \Theta_{\tilde{E}}  \leq c_{\delta}  (\omega'_{t, l}+\sqrt{-1} \partial \bar{\partial} \tilde{\varphi}'_{l})$$
for all $t\in [0, T']$ and some constant $c_{\delta}>0$ depending on $l, T'$ and $\delta$ where $c_{\delta}\to 1$ as $\delta \to 0$.   Now let $t_0\in [0, T']$ be given.  Then for any $0\leq s\leq t_0$, if $\max_{X'\setminus \tilde{E}} D_+(\cdot, s) = D_+(x_s, s)$  then from \eqref{eq:ma flow 9} we have the following at $(x_s, s)$
\begin{equation}
	\begin{split}
		\frac{\partial}{\partial t} D_+ \leq \log \frac{( \omega'_{t, l}+\sqrt{-1} \partial \bar{\partial} \tilde{\varphi}'_{l}+\delta \Theta_{\tilde{E}})^n}{ (\omega'_{t, l}+\sqrt{-1} \partial \bar{\partial} \tilde{\varphi}'_{l})^n}\leq n\log c_{\delta}
	\end{split}
\end{equation}
and by a maximum principle argument as in the proof of Lemma \ref{lem-upper 1} we conclude $$\sup D_+ (t_0) \leq  \sup D_+ (0) +n|\log c_{\delta}|T'=n|\log c_{\delta}|T'$$

and by letting $\de\to 0,$ we see that $\psi'_{l}\leq\varphi'_{l}.$ Similarly by taking $D_{-}:=\psi'_{l}-\varphi'_{l}-\de\log|\tilde{S}_{\tilde{E}}|^{2}$ we can show that $\psi'_{l}\geq\varphi'_{l},$ which completes the proof of this Lemma.
\end{proof}
\begin{remark}\label{finalcd8}
	As the functions $\varphi_{l, 0}$ are non-increasing in $l$ by construction, the proof above can also be used to show that the functions $\varphi_l'$ in the Lemma are non-increasing in $l$.
	\end{remark}
We now follow the argument in \cite{GZ} to prove that the solution $\varphi$  to \eqref{eq:ma flow 1}  constructed in \S 4.1  is the unique maximal solution to \eqref{eq:ma flow 1} in the sense of the following:

\begin{theorem}\label{thm-unique2}
Let $\psi \in L_{loc}^{\infty}((X'\setminus\tilde{E})\times[0,T_0))\bigcap C^{\infty}((X'\setminus\tilde{E})\times(0,T_0))$ solve \eqref{eq:ma flow 1} on $(X'\setminus\tilde{E})\times(0,T_0)$  satisfying
$\lim\limits_{t\searrow 0} \psi(t)=\varphi_0$ in $L^1(X')$.  Then $\psi(t)\leq\varphi(t)$ where $\varphi(t)$ is the solution to  \eqref{eq:ma flow 1} constructed above.
\end{theorem}
\begin{proof}

Recall  that $\varphi =\varphi' +\eta\sum_{j}|S_{j}|_{j}^{2(1-b_{j})}-t \log \log^{2}|S_{i}|_{i}^{2}$ where in turn $\varphi'_{l}(t)\searrow \varphi'$ and each $\varphi'_{l}$ solves

\begin{equation}\label{D1}
\left\{
   \begin{array}{ll}
     \displaystyle\dt\varphi' &=\displaystyle\log\frac{(\dr'_{t}+\ddb\varphi')^{n}\prod_{i}|S_{i}|_{i}^{2}\log^{2}|S_{i}|_{i}^{2}
     \prod_{j}(|S_{j}|_{j}^{2})^{b_{j}}}{\Omega'\prod_{k}(|S_{k}|_{k}^{2})^{a_{k}}}\\
     \;\\
     \varphi'(0) &= \varphi_{l,0}-\eta\sum_{j}\mathcal{F}(|S_{j}|_{j}^{2},1-b_{j},\e).
   \end{array}
 \right.
\end{equation}
where
\begin{align} \nonumber
\dr'_{t}:&=\pi^{*}\dr_{0}+t\chi-t\sum_{i}\ddb\log\log^{2}|S_{i}|_{i}^{2}+\eta\sum_{j}\ddb\mathcal{F}(|S_{j}|_{j}^{2},1-b_{j},\e^{2}),
\end{align}

 By the construction of $\varphi(t)$ above it suffices to prove $\psi'(t)\leq\varphi'_{l}(t)$ for any $l>0$ where $\psi'$ is another solution to \eqref{D1}.

Fix some $(x, t)\in (X'\setminus\tilde{E})\times[0,T_0)$.  Then for any $0<t_{\e}<t$ the function $\psi'-\varphi'_{l}$ attains the maximum on $(X'\setminus\tilde{E})\times[t_{\e},T_0)$ at some $(t_{\e},x_{\e})$.  This can be shown by applying the
maximum principle to the function $\psi'-\varphi'_{l}+\de\log|\tilde{S}_{\tilde{E}}|^{2}$ as in the proof of Lemma \ref{lem-unique1} for some $\de > 0$ then letting $\de\to 0$.  Thus for each $l>0$ we have

$$\psi'(t,x)-\varphi'_{l}(t,x)\leq\sup_{X'}(\psi'(t_{\e})-\varphi'_{l}(t_{\e}))\leq\sup_{X'}(\psi'(t_{\e})-\varphi'_{l}(0)) + \sup_{X'}(\varphi'_{l}(0)-\varphi'_{l}(t_{\e}))$$

Now by the same proof of Lemma \ref{lem-t-dev}, we may have $n\log t - C_l\leq \dot \varphi'_{l}$ on $X'\times [0, t_{\e}]$ for some $C_l >0$ depending on $l$ and $t_{\e}$.   Integrating  in time gives $ \sup_{X'}(\varphi'_{l}(0)-\varphi'_{l}(t_{\e})) \to 0$ as $t_{\e}\to 0$.  Meanwhile, Hartogs Lemma and the continuity of $\varphi'_{l}(0)$ on $X'$ gives $\sup_{X'}(\psi'(t_{\e})-\varphi'_{l}(0))\to  \sup_{X'}(\psi'(0)-\varphi'_{l}(0))\leq 0$ as $t_{\e}\to 0$ where we have used the fact that $\varphi'_{l}(0)$ is non-increasing by construction.  By letting $t_{\e}\to 0$ in the inequality above we conclude that $\psi'(t,x)-\varphi'_{l}(t,x)\leq 0$ for all $l$ which in turn implies $\psi'(t,x)-\varphi'(t,x)\leq 0$.  This proves the Theorem.

\end{proof}

\subsection{Improved lower bounds}

The main goal of this section is to prove the following stronger lower bound for the solution $\varphi(t)$ to \eqref{eq:ma flow 1} constructed in \S 4.1.
\begin{theorem}\label{thm-lb2}
If $\varphi(0)\in L_{loc}^{\infty}(X'\setminus\pi^{-1}(X_{lc}))$ and gives rise to a current with zero Lelong number, then $\varphi(t)$ belongs to $L_{loc}^{\infty}((X'\setminus\pi^{-1}(X_{lc}))\times[0,T'])\bigcap C^{\infty}((X'\setminus\pi^{-1}(\tilde{E}))\times(0,T_0))$ and also gives rise to a current with zero Lelong number.
\end{theorem}
\begin{proof}
We will apply the comparison principle associated with the $L^{\infty}$-estimate in \cite{EGZ,GZ,ST3}. Assume that $\pi^{-1}(X_{lc})=\bigcup_{i}D_{i}\bigcup_{j'}D_{j'}\bigcup_{k'}D_{k'},$ where $D_{j'},D_{k'}$ represent those lt divisors and canonical divisors which have non-empty intersections with lc divisors $D_{i}.$ By the assumption of this theorem for any $\de>0$ there exists a constant $C_{\de}>0$ such that
\begin{equation}\label{eq:lower II-1}
\varphi(0)\geq\de\sum_{i,j',k'}(\log|S_{i}|_{i}^{2}+\log|S_{j'}|_{j'}^{2}+\log|S_{k'}|_{k'}^{2})-C_{\de}.
\end{equation}

Consider the following equation approximating \eqref{eq:ma flow 1}:
\begin{equation}\label{eq:ma flow 1-app}
\left\{
   \begin{array}{ll}
     \displaystyle\dt\varphi_{\e, l} &=\displaystyle\log\frac{(\dr_{t}+l^{-1}\theta + \ddb\varphi_{\e})^{n}\prod_{i}(|S_{i}|_{i}^{2}+\e^{2})\prod_{j}|S_{j}|_{j}^{2b_{j}}}
     {\Omega'\prod_{k}|S_{k}|_{k}^{2a_{k}}}\\
     \varphi_{\e, l}(0) &=\varphi_{l, 0},
   \end{array}
 \right.
\end{equation}
for approximation parameter $\e>0.$  Fix $T'<T_0$.  As $\varphi_{l, 0}$ is smooth on $X'$, by the results in section 7 of \cite{GZ}, for each  $\e>0$ there exists a unique solution
$\varphi_{\e, l}(t)\in L^{\infty}(X'\times(0,T'))\bigcap C^{\infty}_{loc}((X'\setminus\tilde{E})\times(0,T')).$ Moreover, by a maximum principle argument as in \S 4.3 (see Remark \ref{finalcd8}), we may conclude that the family of functions $\varphi_{\e, l}(t)$ are non-increasing as $\e\searrow 0$ and $l\to \infty$ separately and that in particular the zero Lelong number solution $\varphi(t)$ to \eqref{eq:ma flow 1} constructed in \S 4.1 provides a uniform lower barrier for the family of solutions $\varphi_{\e, l}(t)$ and it follows from this, and the the proof of Theorem 7.5 in \cite{GZ}, that $\varphi_{\e, l}(t)$ satisfies local higher order estimates on $X' \setminus \tilde{E} \times[0, T']$ as in Theorem \ref{thm-high order}, where the estimates are independent $\e, l$.  In particular, we conclude the monotone convergence of $\varphi_{\e, l}(t)$, as  $\e\searrow 0$ then $l\to \infty$, to a smooth limit solution $\psi(t)$ to \eqref{eq:ma flow 1} on $X'\setminus \tilde{E}\times [0, T')$ satisfying the conditions of Theorem \ref{thm-unique2}, thus giving $\psi(t)\leq \varphi(t)$.  On the other hand, we have $\psi(t)\geq \varphi(t)$ by construction, and we conclude that $\psi(t)=\varphi(t)$.

  For small constants $\de',\e >0$ set
\begin{align}\label{eq:lower II-2}
\phi_{\de',\e}(t):&=\varphi_{\e, l}(t)-\de't\sum_{i}\log(|S_{i}|_{i}^{2}+\e^{2})\\
&-\de\sum_{i,j',k'}(\log(|S_{i}|_{i}^{2}+\e^{2})+\log(|S_{j'}|_{j'}^{2}+\e^{2})+\log(|S_{k'}|_{k'}^{2}+\e^{2}))\nonumber
\end{align}
 and compare  \eqref{eq:ma flow 1-app}, it follows that $\phi_{\de',\e}(t)$ satisfies
\begin{equation}\label{eq:ma flow 10}
\left\{
   \begin{array}{ll}
     \displaystyle\dt\phi_{\de',\e} &=\displaystyle\log\frac{(\dr_{t,\de',\e, l}+\ddb\phi_{\de',\tilde{e}})^{n}\prod_{i}(|S_{i}|_{i}^{2}+\e^{2})^{1-\de'}\prod_{j}|S_{j}|_{j}^{2b_{j}}}
     {\Omega'\prod_{k}|S_{k}|_{k}^{2a_{k}}}\\ \;\\
     \phi_{\de',\e}(0)&= \varphi_{l,0}-\de\sum_{i,j',k'}(\log(|S_{i}|_{i}^{2}+\e^{2})+\log(|S_{j'}|_{j'}^{2}+\e^{2})+\log(|S_{k'}|_{k'}^{2}+\e^{2})),
   \end{array}
 \right.
\end{equation}
on $X'\setminus{\tilde{E}}$ where

\begin{align*}
\dr_{t,\de',\e, l}:&=\dr_{t}+l^{-1}\theta + \de't\sum_{i}\ddb\log(|S_{i}|_{i}^{2}+\e^{2})\\
&+\de\sum_{i,j',k'}\ddb(\log(|S_{i}|_{i}^{2}+\e^{2})+\log(|S_{j'}|_{j'}^{2}+\e^{2})+\log(|S_{k'}|_{k'}^{2}+\e^{2}))\\
&\geq\dr_{t}+l^{-1}\theta-\de't\sum_{i}\frac{|S_{i}|_{i}^{2}\Theta_{i}}{|S_{i}|_{i}^{2}+\e^{2}}-\de\sum_{i,j',k'}
(\frac{|S_{i}|_{i}^{2}\Theta_{i}}{|S_{i}|_{i}^{2}+\e^{2}}+\frac{|S_{j'}|_{j'}^{2}\Theta_{j'}}{|S_{j'}|_{j'}^{2}+\e^{2}}
+\frac{|S_{k'}|_{k'}^{2}\Theta_{k'}}{|S_{k'}|_{k'}^{2}+\e^{2}})
\end{align*}
  Note that as $D_{i},D_{j'},D_{k'}$ are all exceptional divisors which are generated by blow-up operations thus carry hermitian metrics with non-positive curvature, by choosing small enough $\de,\de'>0$ we may assume by Assumption \ref{mainass} (3) and Remark \ref{changinghermitianmetrics} that $\dr_{t,\de',\e, l}$ is \ka\, on $X'$ for all $t\in [0,T'].$  We then choose another fixed \ka\, form $\kappa$ on $X'$ satisfying  $\kappa\leq\frac{\dr_{t,\de',\e}}{2}$ for all $t\in [0,T']$ and     consider the complex \MA\ equation:
\begin{equation}\label{eq:ma-compare}
(\kappa+\ddb\psi_{\de',\e})^{n}=\frac{C_{\de',\e}\prod_{k}|S_{k}|_{k}^{2a_{k}}\Omega'}
{\prod_{i}(|S_{i}|_{i}^{2}+\e^{2})^{1-\de'}\prod_{j}|S_{j}|_{j}^{2b_{j}}}
\end{equation}
where $C_{\de',\e}$ is chosen such that
$$[\kappa]^{n}=C_{\de',\e}\int_{X'}\frac{\prod_{k}|S_{k}|_{k}^{2a_{k}}\Omega'}
{\prod_{i}(|S_{i}|_{i}^{2}+\e^{2})^{1-\de'}\prod_{j}|S_{j}|_{j}^{2b_{j}}}$$ and moreover $C_{\de',\e}$ are uniformly bounded positive constants whose bounds are independent of $\e.$ By the $L^{\infty}$-estimates in \cite{EGZ} there exists a unique solution $\psi_{\de',\e}$ which is uniformly bounded independent of $\e.$ Similar to \cite{ST3}, set $\xi_{\de',\e}(t):=\phi_{\de',\e}(t)-\psi_{\de',\e}$ it follows that $\xi(t)$ satisfies the following equation
\begin{equation}\label{eq:ma flow-compare}
\left\{
   \begin{array}{ll}
     \displaystyle\dt\xi_{\de',\e} &=\displaystyle\log\frac{((\kappa+\ddb\psi_{\de',\e})+(\dr_{t,\de',\e}-\kappa)+\ddb\xi_{\de',\e})^{n}}
     {(\kappa+\ddb\psi_{\de',\e})^{n}}+\log C_{\de',\e}\\ \;\\
     \xi_{\de',\e}(0)&= \phi_{\de',\e}(0)-\psi_{\de',\e},
   \end{array}
 \right.
\end{equation}
and by a maximum principle argument as in the proof of Lemma \ref{lem-upper 1} we may  conclude that
$$\xi_{\de',\e}(t)\geq\phi_{\de',\e}(0)-\psi_{\de',\e}+t\log C_{\de',\e}.$$ Combine this with \eqref{eq:lower II-1}, \eqref{eq:lower II-2} to deduce the lower bound
\begin{align}\label{eq:lower II-app}
\varphi_{\e,l}(t)&\geq -C(\de,\de',T')+\de't\sum_{i}\log(|S_{i}|_{i}^{2}+\e^{2})\nonumber\\
&+\de\sum_{i,j',k'}(\log(|S_{i}|_{i}^{2}+\e^{2})+\log(|S_{j'}|_{j'}^{2}+\e^{2})+\log(|S_{k'}|_{k'}^{2}+\e^{2}))\nonumber\\
&\geq-C(\de,\de',T')+\de't\sum_{i}\log|S_{i}|_{i}^{2}+\de\sum_{i,j',k'}(\log|S_{i}|_{i}^{2}+\log|S_{j'}|_{j'}^{2}+\log|S_{k'}|_{k'}^{2})
\end{align}
on $(X'\setminus\tilde{E})\times[0,T']$ for all $\e, l$, and we may extend this inequality to hold on $(X'\setminus\pi^{-1}(X_{lc}))\times[0,T'].$  As we have shown that $\varphi_{\e, l}(t)\searrow\varphi(t)$ as  $\e\searrow 0$ and $l\to \infty$ it follows that $\varphi(t)$ satisfies the same lower bound above and the proof of this theorem is complete.
\end{proof}

\subsection{Completion of proof of Theorem \ref{thm-main1}}

\begin{proof}[Proof of Theorem \ref{thm-main1}]
In \S 4.1 we have constructed a smooth solution $\varphi(t)$ to \eqref{eq:ma flow 1} on $(X'\setminus \widetilde{E}) \times\color{black}{(0, T_0)}$   \color{black}{converging as $t\to 0$ in the sense of currents on $X'$ and such that $\omega(t)=\pi^* \omega_0 + t \chi +\ddb\varphi(t)$ and extends to a current on $X'$ with zero Lelong number for each $t\in [0, T_0)$} \color{black}{}.
In view of the properties of $\varphi$ already established in \S 4.2 and 4.3, it remains only to show that the solution $\omega(t)$ descends from $X'$ to $X$ in the sense: if $\pi^* \omega_0 + t \chi=0$ along any fiber $\pi^{-1}(x)$ of $\pi$, then $\varphi(t)$ is constant along $\pi^{-1}(x)$.

 We argue as in \cite{ST3}.  Note that $\varphi(t)$ is a $\dr_{t}=\pi^* \omega_0 + t \chi $ plurisubharmonic function along any fiber $\pi^{-1}(x)$.  On the other hand, we may  always choose hermitian metrics in $\chi$ so that $\pi^* \omega_0 + t \chi =0$ on $\pi^{-1}(x)$ and it follows that $\varphi=c$ on $\pi^{-1}(x)$ for some $c\in [-\infty, \infty)$.  Thus $\omega(t)$ descends to a current on $X$ having zero Lelong number.

 Moreover, if $\pi^{-1}(x) \bigcap X_{lc}$ above is empty, then $c\neq -\infty$  by Theorem \ref{thm-lb2}. On the other hand, if  $\pi^{-1}(x)\bigcap X_{lc}$ is nonempty than as $$\varphi=\varphi'-t\sum_{i}\log\log^{2}|S_{i}|_{i}^{2}-\eta\sum_{j}|S_{j}|_{j}^{2(1-b_{j})}$$ and $\varphi'$ is uniformly bounded from above by Theorem \ref{thm:upper-2}, it follows that $c=-\infty$.

 This completes the proof of Theorem \ref{thm-main1}.

\end{proof}

\section{Normalized \ka-Ricci flow on semi-log canonical models (proof of Theorem \ref{thm-main2}) }
\subsection{Semi-log canonical models and basic settings}
In this section we will prove Theorem \ref{thm-main2}. First we briefly recall the definition of semi-log canonical models (cf. \cite{BG,KM,So17}):
\begin{definition}\label{def-semi-lc}
A reduced projective variety $X$ with $dim_{\mathbb{C}}X=n$ which is $\mathbb{Q}$-Gorenstein and satisfies Serre's $S_{2}$ condition is said to be a semi-log canonical model if
\begin{enumerate}
\item $K_{X}$ is an ample $\mathbb{Q}$-Cartier divisor,
\item X has only ordinary nodes in codimension 1,
\item X has log canonical singularities, i.e., for any resolution $\pi:X'\to X,$ in the adjunction formula \eqref{eq:adjunction 2} it holds that all $a_{i}\geq -1.$
\end{enumerate}
\end{definition}
Note that from (2) in general, semi-log canonical models may not be normal varieties. By \cite{So17} in this case by the standard normalization $\nu:X^{\nu}\to X$ it follows that $K_{X^{\nu}}=\nu^{*}K_{X}-cond(\nu),$ where $cond(\nu)$ is an effective reduced divisor which comes from the inverse image of the codimension 1 nodes. Combined with the resolution $\pi^{\nu}:X'\to X^{\nu},$ we can consider the resolution $$\pi:=\pi^{\nu}\circ\nu:X'\to X^{\nu}\to X$$ which
 satisfies the condition of log canonical singularities.

 We may thus consider to be in the situation of Theorem  \ref{thm-main1} though with the additional condition that, in terms of the above resolution, $\chi=(-Ric(\Omega')+\sum_{i}\Theta_{i}+\sum_{j}b_{j}\Theta_{j}-\sum_{k}a_{k}\Theta_{k})$ is itself semi ample and big.  In particular, this gives  $T_0=\infty$ in Assumption \ref{mainass} (4) while in Assumption \ref{mainass} (5) we have the stronger statement that \begin{equation}\label{B1}\lambda \pi^{*}\dr_{0}+(1-\lambda)\chi+\de\ddb \log|\tilde{S}_{\tilde{E}}|^{2}\geq c_{\de}\theta\end{equation} on $X'$ for some $c_{\de}>0$ and all $\de < d$ and $0\leq \lambda \leq 1$.  Under the above conditions we will show that the longtime solution to \eqref{eq:krf} given by Theorem \ref{thm-main1} can be transformed to a longtime solution to the normalized \ka-Ricci flow \eqref{eq:nor-krf} converging to a negative \ka-Einstein current on $X'$ in the current sense and the $C^{\infty}_{loc}$ sense on $X'\setminus \tilde{E}$.  We point out that the existence of \ka-Einstein currents on such semi-log canonical models had been established earlier by  \cite{BG,So17} using elliptic methods.

 By Theorem \ref{thm-main1} it follows  that \eqref{eq:krf} has a longtime weak solution $\dr(t)=\pi^{*}\dr_{0}+t \chi +\ddb \varphi'(t)$ on $(X'\setminus \widetilde{E}) \times[0, \infty)$ where $\varphi'(t)$ solves \eqref{eq:ma flow 2} on $(X'\setminus \widetilde{E}) \times(0, \infty)$.  We also recall that $\varphi'_{l,v,\e}$ solves the approximate euqation \eqref{eq:ma flow 4}  on $(X'\setminus \widetilde{E}) \times[0, \infty)$ and that by construction we have $\varphi'_{l,v,\e}+\eta\sum_{j}|S_{j}|_{j}^{2(1-b_{j})}-t \log \log^{2}|S_{i}|_{i}^{2} \to \varphi$ on $(X'\setminus \widetilde{E}) \times(0, \infty)$ as the parameters  $l,v,\e$ approach their limits along appropriate subsequences.

 From the above, it is straight forward to check that $\tilde{\dr}(t):=e^{-t}\dr(e^t -1)$ is a longtime weak solution to the normalized \ka-Ricci flow

\begin{equation}\label{E2}
\left\{
   \begin{array}{ll}
     \displaystyle\dt\tilde{\dr} &=-Ric(\tilde{\dr})-\tilde{\dr} \\
     \tilde{\dr}(0) &= \pi^{*}{\dr}_{0} +\ddb \varphi_0
   \end{array}
 \right.
\end{equation}
on $(X'\setminus \widetilde{E}) \times[0, \infty)$ where
\begin{equation}\label{eeeee1}\begin{split}\tilde \dr(t)&=e^{-t} \pi^{*}\dr_{0}+(1-e^{-t}) \chi +\ddb \tilde \varphi(t)\\
\tilde\varphi(t)&:=e^{-t}\varphi'(e^t-1).\\
\end{split}
\end{equation}
Also, $\tilde\varphi(t)$ solves the normalized \MA flow:

\begin{equation}\label{eq:nor ma flow 1}
\left\{
   \begin{array}{ll}
     \displaystyle\dt\tilde{\varphi} &=\displaystyle\log\frac{(\tilde{\dr}_{t}+\ddb\tilde{\varphi})^{n}\prod_{i}|S_{i}|_{i}^{2}\prod_{j}|S_{j}|_{j}^{2b_{j}}}
     {\Omega'\prod_{k}|S_{k}|_{k}^{2a_{k}}}-\tilde{\varphi}+nt\\
     \tilde{\varphi}(0) &= \pi^{*}\tilde{\varphi}_{0}.
   \end{array}
 \right.
\end{equation}
Moreover, we have the convergence $\tilde\varphi'_{l,v,\e}+\eta\sum_{j}|S_{j}|_{j}^{2(1-b_{j})}-(e^t-1) \log \log^{2}|S_{i}|_{i}^{2} \to \tilde \varphi$ on $(X'\setminus \widetilde{E}) \times(0, \infty)$ where  $\tilde\varphi'_{l,v,\e}(t):=e^{-t}\varphi'_{l,v,\e}(e^t -1)$ solves  the normalized approximate \MA flow

\begin{equation}\label{eq:nor ma flow 3}
\left\{
   \begin{array}{ll}
     \displaystyle\dt\tilde{\varphi}'_{l,v,\e} &=\displaystyle\log\frac{(\tilde{\dr}'_{t,l,v,\e}+\ddb\tilde{\varphi}'_{l,v,\e})^{n}
     \prod_{i}|S_{i}|_{i}^{2}\log^{2}|S_{i}|_{i}^{2}\prod_{j}(|S_{j}|_{j}^{2}+\e^{2})^{b_{j}}}{\Omega'\prod_{k}(|S_{k}|_{k}^{2}+\e^{2})^{a_{k}}}\\ \;\\ &\quad-\tilde{\varphi}'_{l,v,\e}+nt\\
     \tilde{\varphi}'_{l,v,\e}(0) &= \tilde{\varphi}_{l,0}-\eta\sum\limits_{j}\mathcal{F}(|S_{j}|_{j}^{2},1-b_{j},\e^{2}),
   \end{array}
 \right.
\end{equation}
where
\begin{align*}
\tilde{\dr}'_{t,l,v,\e}:=&e^{-t}\pi^{*}\dr_{0}+(1-e^{-t})\chi+l^{-1}\theta-(1+v-e^{-t})\sum\limits_{i}\ddb\log\log^{2}|S_{i}|_{i}^{2}\\
&+\eta\sum\limits_{j}\ddb\mathcal{F}(|S_{j}|_{j}^{2},1-b_{j},\e^{2}).
\end{align*}

 In particular, we have  the convergence $\tilde{\dr}'_{l,v,\e}(t) \to \tilde\dr(t)$ on $(X'\setminus \widetilde{E}) \times(0, \infty)$ where for each set of parameter values,  $\tilde{\dr}'_{l, v,\e}(t)= \tilde{\dr}'_{t,l,v,\e}+\ddb\tilde{\varphi}'_{l,v,\e}(t)$ is a family of complete bounded curvature \ka\, metrics on $X'\setminus \widetilde{E}$ equivalent to a Carlson-Grifiths metric.

\subsection{Uniform estimates and convergence}

 Note that by an addition of the same function of time only, $\tilde{\varphi}(t)$ and $\tilde{\varphi}'_{l,v,\e}(t)$ will solve the same equations as \eqref{eq:nor ma flow 1} and \eqref{eq:nor ma flow 3} (respectively), but without the term $nt$ on the RHS.  We will make this assumption in this subsection, and observe that this does not affect $\tilde{\dr}(t)$ in  \eqref{eeeee1}.
Under this assumption, we will prove that $\tilde\varphi(t)$ converges smoothly locally uniformly on $(X'\setminus \widetilde{E})$ to a limit as $t\to \infty$.  This combined with \eqref{eeeee1} will imply that $\tilde{\dr}(t)$ converges to a limit smoothly locally uniformly on $(X'\setminus \widetilde{E})$ as $t\to \infty$, which by \eqref{E2} will be \ka\, Einstein with negative scalar curvature.

 By \eqref{B1} the background form in \eqref{eq:nor ma flow 3} satisfies
\begin{equation}\label{BB1}c^{-1}_{\de}\hat{\dr}\leq  \tilde{\dr}'_{t,l,v,\e}+\de\ddb \log|\tilde{S}_{\tilde{E}}|^{2}\leq  c_{\de}\hat{\dr}\end{equation} on $X'\setminus \widetilde{E}\times[0, \infty)$ for all $\de$ sufficiently small and some $c_{\de}>0$ where $\hat{\dr}$ is the Carlson Grifiths metric on $X'\setminus \widetilde{E}$ defined in Theorem \ref{thm-C2}.  By this uniform equivalence, the arguments in \S 3.2 can be adapted to the case of  \eqref{eq:nor ma flow 3} in a straight forward manner to obtain uniform estimates for $\tilde{\varphi}'_{l,v,\e}$ and its derivatives which are local in space yet global in time (see Lemma \ref{lem:nor high order}).  We will illustrate this in detail below for the $C^{0}$-estimates for $\tilde{\varphi}'_{l,v,\e}$ and also prove an upper bound on $\dot{\tilde{\varphi}'}_{l,v,\e}$ showing that $ \tilde{\varphi}'_{l,v,\e}$ is essentially non-increasing in time.  All together this will imply the smooth local convergence of $\tilde\varphi(t)$ on $(X'\setminus \widetilde{E})$ to a limit as $t\to \infty$.

  We begin with the following upper bound estimate:
\begin{lemma}\label{lem:nor upper}
$\tilde{\varphi}'_{l,v,\e}\leq C_{0}$ for a uniform constant $C_{0}.$
\end{lemma}
\begin{proof}
As Lemma \ref{lem-upper 1}, define $\tilde{\dr}''_{t,l,v,\e}:=\tilde{\dr}'_{t,l,v,\e}-(1-e^{-t})\sum_{k}a_{k}\ddb\log(|S_{k}|_{k}^{2}+\e^{2})$ and
$\tilde{\phi}'_{l,v,\e}:=\tilde{\varphi}'_{l,v,\e}+(1-e^{-t})\sum_{k}a_{k}\log(|S_{k}|_{k}^{2}+\e^{2}),$ it follows that
$\tilde{\phi}'_{l,v,\e}$ satisfies the following equation
\begin{equation}\label{eq:nor ma flow 4}
\left\{
   \begin{array}{ll}
     \displaystyle\dt\tilde{\phi}'_{l,v,\e} &=\displaystyle\log\frac{(\tilde{\dr}''_{t,l,v,\e}+\ddb\tilde{\phi}'_{l,v,\e})^{n}
     \prod_{i}|S_{i}|_{i}^{2}\log^{2}|S_{i}|_{i}^{2}\prod_{j}(|S_{j}|_{j}^{2}+\e^{2})^{b_{j}}}{\Omega'}\\ \;\\ &\quad-\tilde{\phi}'_{l,v,\e}-\eta\sum\limits_{j}\mathcal{F}(|S_{j}|_{j}^{2},1-b_{j},\e^{2})\\
     \tilde{\phi}'_{l,v,\e}(0) &= \tilde{\varphi}_{l,0}-\eta\sum\limits_{j}\mathcal{F}(|S_{j}|_{j}^{2},1-b_{j},\e^{2}).
   \end{array}
 \right.
\end{equation}
By \eqref{BB1} we have $$\tilde{\dr}''_{t,l,v,\e}\leq\tilde{\dr}'_{t,l,v,\e}+c(1-e^{-t})\theta\leq C\hat{\dr}$$ where $\hat{\dr}$ was the Carlson-Grifiths metric defined in Theorem \ref{thm-C2}. Thus by a maximum principle argument as in the proof of Lemma \ref{lem-upper 1}, it follows that $$\tilde{\varphi}'_{l,v,\e}=\tilde{\phi}'_{l,v,\e}-(1-e^{-t})\sum_{k}a_{k}\log(|S_{k}|_{k}^{2}+\e^{2})\leq C-(1-e^{-t})\sum_{k}a_{k}\log(|S_{k}|_{k}^{2}+\e^{2}).$$
Next, using a similar argument as in the proof of Theorem \ref{thm:upper-2} we can show that $\tilde{\varphi}'_{l,v,\e}\leq C_{0}$ which is a uniform constant.
\end{proof}
Next we have the following lower bound estimate:
\begin{lemma}\label{lem:nor lower}
For any $\de>0$ there exists a constant $C_{\de}$ which only depends on $\de$ such that on $(X'\setminus\tilde{E})\times[0,+\infty),$ it holds that
$\tilde{\varphi}'_{l,v,\e}\geq\de\log|\tilde{S}_{\tilde{E}}|^{2}-C_{\de}.$
\end{lemma}
\begin{proof}
By \eqref{BB1} we have
$$\tilde{\dr}'_{t,l,v,\e,\de}:=\tilde{\dr}'_{t,l,v,\e}+\de\log|\tilde{S}_{\tilde{E}}|^{2}\geq c_{\de}\hat{\dr}$$ where $c_{\de}>0$ is
independent of time. We can rewrite \eqref{eq:nor ma flow 3} as following:
\begin{align*}
&\dt(\tilde{\varphi}'_{l,v,\e}-\de\log|\tilde{S}_{\tilde{E}}|^{2}) \\ =&\log\frac{(\tilde{\dr}'_{t,l,v,\e,\de}+\ddb(\tilde{\varphi}'_{l,v,\e}-\de\log|\tilde{S}_{\tilde{E}}|^{2}))^{n}
     \prod_{i}|S_{i}|_{i}^{2}\log^{2}|S_{i}|_{i}^{2}\prod_{j}(|S_{j}|_{j}^{2}+\e^{2})^{b_{j}}}{\Omega'\prod_{k}(|S_{k}|_{k}^{2}+\e^{2})^{a_{k}}}\\ &-(\tilde{\varphi}'_{l,v,\e}-\de\log|\tilde{S}_{\tilde{E}}|^{2})-(\de\log|\tilde{S}_{\tilde{E}}|^{2}+
     \eta\sum\limits_{j}\mathcal{F}(|S_{j}|_{j}^{2},1-b_{j},\e^{2})),
\end{align*}
where $\tilde{\varphi}'_{l,v,\e}(0)-\de\log|\tilde{S}_{\tilde{E}}|^{2}\geq-C'_{\de}$ by the assumption of zero Lelong number. By  a maximum principle argument as in the proof of Lemma \ref{lem-upper 1}, it follows that $$\tilde{\varphi}'_{l,v,\e}(t)-\de\log|\tilde{S}_{\tilde{E}}|^{2}\geq-e^{-t}C'_{\de}-C''_{\de}(1-e^{-t})\geq-C_{\de},$$ which concludes the proof of the Lemma.
\end{proof}
Next, we need the following upper bound for the time derivative $\dot{\tilde{\varphi}}'_{l,v,\e}$ 
\begin{lemma}\label{lem-nor-t-dev}
There exists a uniform constant $C>0$ such that $$\dot{\tilde{\varphi}}'_{l,v,\e}\leq Cte^{-t}.$$
\end{lemma}
\begin{proof}
We proceed as \cite{TZ}. First, set $\tilde{\dr}:=\tilde{\dr}'_{t,l,v,\e}+\ddb\tilde{\varphi}'_{l,v,\e},$ by direct computations, we have
\begin{equation}\label{eq:nor-t-dev 1}
\dt\dot{\tilde{\varphi}}'_{l,v,\e}=\Delta_{\tilde{\dr}}\dot{\tilde{\varphi}}'_{l,v,\e}+tr_{\tilde{\dr}}
\dot{\tilde{\dr}}'_{t,l,v,\e}-\dot{\tilde{\varphi}}'_{l,v,\e},
\end{equation}
From \eqref{eq:nor-t-dev 1}, it follows that
\begin{align*}
&(\dt-\Delta_{\tilde{\dr}})((e^{t}-1)\dot{\tilde{\varphi}}'_{l,v,\e}-\tilde{\varphi}'_{l,v,\e}-C't)\\ \leq&-tr_{\tilde{\dr}}
(\tilde{\dr}'_{0,l, v,\e})+n-C') \leq 0,
\end{align*}
where $C'>0$ is chosen to be large enough. Combined with Lemma \ref{lem:nor upper} the upper bound for $\dot{\tilde{\varphi}}'_{l,v,\e}$ follows by a maximum principle argument as in the proof of Lemma \ref{lem-upper 1}.

\end{proof}
Similar to the proofs of the above Lemmas, we may continue to adapt the arguments in the proofs of Theorem \ref{thm-C2} and Theorem \ref{thm-high order} to obtain a Laplacian estimate and higher order estimates for $\tilde{\varphi}'_{l,v,\e}$, using \eqref{B1}, to obtain
\begin{lemma}\label{lem:nor high order}
For any compact set $K\subset X'\setminus\tilde{E}$ there exist constants $C(m,K)$ such that
$$|\tilde{\varphi}'_{l,v,\e}|_{C^{m}(K\times(0,+\infty))}\leq C(m,K).$$
\end{lemma}

From Lemma \ref{lem:nor high order}, we may let the parameters $l,v,\e$ approach their limits along
appropriate sequences and obtain a limit $\tilde{\varphi}'(t)\in C^{\infty}
(X'\setminus\tilde{E})\times(0, \infty)$ having zero Lelong number for all
$t$. Moreover, by Lemmas \ref{lem:nor upper} and \ref{lem-nor-t-dev} it
follows that for a sufficiently large constant $C$ the function
$\tilde{\varphi}'(x, t)+Ct e^{-t}$ is non-decreasing in time for large $t$, while bounded
above by some function of $x$ only. We conclude the convergence
$\tilde{\varphi}'(x, t) \to \tilde{\varphi}'_{\infty}(x)\in
C^{\infty}(X'\setminus\tilde{E})$ as $t\to \infty$ where
$\tilde{\varphi}'_{\infty}(x)$ will also have zero Lelong number. Thus the function
$$\tilde\varphi(t)=\eta\sum_{j}|S_{j}|_{j}^{2(1-b_{j})}-(1-e^{-t}) \log \log^{2}|S_{i}|_{i}^{2} + \tilde\varphi'(t)$$ will solve \eqref{eq:nor ma flow 1} and converge to a limit $\tilde\varphi_{\infty}\in C^{\infty}(X'\setminus\tilde{E})$ as $t\to \infty$  where $\tilde\varphi_{\infty}$ will also have zero Lelong number and will satisfy

\begin{equation}\label{C1}\log\frac{(\tilde{\dr}'_{\infty}+\ddb\tilde{\varphi}_{\infty})^{n}\prod_{i}|S_{i}|_{i}^{2}\prod_{j}|S_{j}|_{j}^{2b_{j}}}
{\Omega'\prod_{k}|S_{k}|_{k}^{2a_{k}}}-\tilde{\varphi}_{\infty}=0\end{equation} in the current sense and in $C^{\infty}(X'\setminus\tilde{E})$ sense.

  We have thus established the convergence $$\tilde{\dr}(t)\to \tilde{\dr}'_{\infty}+\ddb\tilde{\varphi}_{\infty}=\chi+\ddb\tilde{\varphi}_{\infty}$$ as $t\to \infty$ in the sense of currents on $X'$ and in the $C^{\infty}(X'\setminus\tilde{E})$ sense.

Now as $\tilde{\varphi}_{\infty}\in C^{\infty}(X'\setminus \widetilde{E})$, we may define $Ric(\chi +\sqrt{-1} \partial \bar{\partial} \tilde{\varphi}_{\infty})$ as the smooth 1-1 form on $X'\setminus \widetilde{E}$ by:

\begin{equation}
	\begin{split}
		Ric&(\chi +\sqrt{-1} \partial \bar{\partial} \tilde{\varphi}_{\infty})\\
		&:=-\sqrt{-1} \partial \bar{\partial}  \log \frac{(\chi +\sqrt{-1} \partial \bar{\partial} \tilde{\varphi}_{\infty})^n}{\Omega'}+Ric(\Omega')\\
		&=\sqrt{-1} \partial \bar{\partial} (\sum_i \log |S_i|^2_i +b_j\sum_j \log |S_j|^2_j -a_k\sum_k \log |S_k|^2_k ) +\sqrt{-1} \partial \bar{\partial} \tilde{\varphi}_{\infty} +Ric(\Omega')
	\end{split}
\end{equation}
where we have used \eqref{C1} for the second equality.  On the other hand, the smooth form $\sqrt{-1} \partial \bar{\partial} \sum_i \log |S_i|^2_i $ on $X'\setminus \widetilde{E}$ extends naturally to the current $ \sum_i -\Theta_i +2\pi[D_i]$ on $X'$ and similarly for $\sqrt{-1} \partial \bar{\partial} \sum_i \log |S_j|^2_j$  and $\sqrt{-1} \partial \bar{\partial} \sum_i \log |S_k|^2_k$.  Moreover, as $\tilde{\varphi}_{\infty} \in L^1(X')$ (it is bounded above and has zero Lelong number), the smooth form $\sqrt{-1} \partial \bar{\partial}  \tilde{\varphi}_{\infty}$ on $X'\setminus \widetilde{E}$ also extends naturally to a well defined current on $X'$ which we will also denote as $\sqrt{-1} \partial \bar{\partial}  \tilde{\varphi}_{\infty}$.   Thus $Ric(\chi +\sqrt{-1} \partial \bar{\partial} \tilde{\varphi}_{\infty})$ extends naturally to the current on $X'$ given by 
\begin{equation}
	\begin{split}
		Ric&(\chi +\sqrt{-1} \partial \bar{\partial} \tilde{\varphi}_{\infty})\\
		=&\sqrt{-1} \partial \bar{\partial} (\sum_i -\Theta_i +2\pi[D_i] +b_j\sum_i -\Theta_j +2\pi[D_j] -a_k\sum_k -\Theta_k +2\pi[D_k] )\\
		& +\sqrt{-1} \partial \bar{\partial} \tilde{\varphi}_{\infty} +Ric(\Omega')\\
		:=&(\chi+ \sqrt{-1} \partial \bar{\partial}\tilde{\varphi}_{\infty})+(\sum_i 2\pi[D_i] +b_j\sum_i +2\pi[D_j] -a_k\sum_k  +2\pi[D_k] )
	\end{split}
\end{equation}

By the resolution $\pi:X'\to X,$ the normalized \ka-Ricci flow on $X$ will converge to the unique \ka-Einstein current
in the current and $C^{\infty}(X_{reg})$-sense.

\subsection{Improved lower bounds}

To show that $\dr_{KE}$ has bounded local potential away from $X_{lc},$ we can modify our proof of Theorem \ref{thm-lb2}.
As in the previous section, we will again assume without loss of generality that $\tilde{\varphi}$ solves \eqref{eq:nor ma flow 1}, but without the term $nt$ on the RHS.  We first perturb this equation to the following:
\begin{equation}\label{eq:nor ma flow 4 per}
\left\{
   \begin{array}{ll}
     \displaystyle\dt\tilde{\varphi}_{\e, l} &=\displaystyle\log\frac{(\tilde{\dr}_{t}+l^{-1}\theta +\ddb\tilde{\varphi}_{\e, l})^{n}\prod_{i}(|S_{i}|_{i}^{2}+\e^{2})\prod_{j}|S_{j}|_{j}^{2b_{j}}}
     {\Omega'\prod_{k}|S_{k}|_{k}^{2a_{k}}}-\tilde{\varphi}_{\e, l}\\
     \tilde{\varphi}_{\e}(0) &=\varphi_{l, 0}
   \end{array}
 \right.
\end{equation}
As in the beginning of the proof of Theorem \ref{thm-lb2}, for each $\e, l$ there exists a unique solution $\tilde{\varphi}_{\e, l}$ which is smooth away from $\tilde{E}$ and moreover $\tilde{\varphi}_{\e,l}$ is uniformly bounded on $X'\setminus\tilde{E}$ depending on $\e, l$.  As we can move $\tilde{\varphi}_{\e, l}$ to the left hand side of \eqref{eq:nor ma flow 4 per} and write $\tilde{\varphi}_{\e, l}+\dot{\tilde{\varphi}}_{\e, l}=e^{-t}\dt(e^{t}\tilde{\varphi}_{\e, l})$
we can argue as in the beginning of the proof of Theorem \ref{thm-lb2} that $\tilde{\varphi}_{\e, l}$ is non-increasing as each $\e\searrow 0$ and $l\to \infty$ separately, and that by comparing \eqref{eq:nor ma flow 1} and \eqref{eq:nor ma flow 4} we have $\tilde{\varphi}_{\e, l}\searrow\tilde{\varphi}$ as $\e\searrow 0$  and $l \to \infty$.

 Next, we will suppress the index $l$ in the following, and  similar to Theorem \ref{thm-lb2},
for any small enough constants $\de',\e>0$ we set
\begin{align}\label{eq:lower II-2-nor}
\tilde{\phi}_{\de',\e}(t):&=\tilde{\varphi}_{\e}(t)-\de'(1-e^{-t})\sum_{i}\log(|S_{i}|_{i}^{2}+\e^{2})\\
&-\de e^{-t}\sum_{i,j',k'}(\log(|S_{i}|_{i}^{2}+\e^{2})+\log(|S_{j'}|_{j'}^{2}+\e^{2})+\log(|S_{k'}|_{k'}^{2}+\e^{2}))\nonumber
\end{align}
thus $\tilde{\phi}_{\de',\e}(t)$ satisfies
\begin{equation}\label{eq:nor ma flow 5}
\left\{
   \begin{array}{ll}
     \displaystyle\dt\tilde{\phi}_{\de',\e} &=\displaystyle\log\frac{(\tilde{\dr}_{t,\de',\e} +\ddb\tilde{\phi}_{\de',\e})^{n}\prod_{i}(|S_{i}|_{i}^{2}+\e^{2})^{1-\de'}
     \prod_{j}|S_{j}|_{j}^{2b_{j}}}{\Omega'\prod_{k}|S_{k}|_{k}^{2a_{k}}}-\tilde{\phi}_{\de',\e}\\ \;\\
     \tilde{\phi}_{\de',\e}(0)&= \varphi_{l, 0}-\de\sum_{i,j',k'}(\log(|S_{i}|_{i}^{2}+\e^{2})+\log(|S_{j'}|_{j'}^{2}+\e^{2})+\log(|S_{k'}|_{k'}^{2}+\e^{2})),
   \end{array}
 \right.
\end{equation}
on $X'\setminus{\tilde{E}}$ where
\begin{align*}
\tilde{\dr}_{t,\de',\e}:&=\tilde{\dr}_{t}+l^{-1}\theta +\de'(1-e^{-t})\sum_{i}\ddb\log(|S_{i}|_{i}^{2}+\e^{2})\\
&+\de e^{-t}\sum_{i,j',k'}\ddb(\log(|S_{i}|_{i}^{2}+\e^{2})+\log(|S_{j'}|_{j'}^{2}+\e^{2})+\log(|S_{k'}|_{k'}^{2}+\e^{2}))\\
&\geq e^{-t}(\pi^{*}\dr_{0}-\de\sum_{i,j',k'}(\frac{|S_{i}|_{i}^{2}\Theta_{i}}{|S_{i}|_{i}^{2}+\e^{2}}+\frac{|S_{j'}|_{j'}^{2}\Theta_{j'}}
{|S_{j'}|_{j'}^{2}+\e^{2}}+\frac{|S_{k'}|_{k'}^{2}\Theta_{k'}}{|S_{k'}|_{k'}^{2}+\e^{2}}))+l^{-1}\theta \\
&+(1-e^{-t})(\chi-\de'\sum_{i}\frac{|S_{i}|_{i}^{2}\Theta_{i}}{|S_{i}|_{i}^{2}+\e^{2}})
\end{align*}
and by Assumption \ref{mainass} (4) and Remark \ref{changinghermitianmetrics} we may thus assume $\de, \de'$ are sufficiently small so that $\tilde{\dr}_{t,\de',\e}$ is uniformly positive on $X'$ for all $t\in[0,+\infty)$, as $D_{i},D_{j'},D_{k'}$ are exceptional divisors.      Similar to the proof of Theorem \ref{thm-lb2},
choose a \ka\, form $\kappa$ on $X'$ so that $\kappa\leq\frac{\tilde{\dr}_{t,\de',\e}}{2}$ for $t\in[0,+\infty)$ and consider the following complex \MA\ equation
\begin{equation}\label{eq:nor ma-compare}
(\kappa+\ddb\tilde{\psi}_{\de',\e})^{n}=\frac{C_{\de',\e}e^{\tilde{\psi}_{\de',\e}}\prod_{k}|S_{k}|_{k}^{2a_{k}}\Omega'}
{\prod_{i}(|S_{i}|_{i}^{2}+\e^{2})^{1-\de'}\prod_{j}|S_{j}|_{j}^{2b_{j}}}
\end{equation}
where $C_{\de',\e}$ is chosen such that
$$[\kappa]^{n}=C_{\de',\e}\int_{X'}\frac{\prod_{k}|S_{k}|_{k}^{2a_{k}}\Omega'}
{\prod_{i}(|S_{i}|_{i}^{2}+\e^{2})^{1-\de'}\prod_{j}|S_{j}|_{j}^{2b_{j}}}$$ and moreover $C_{\de',\e}$ are uniformly bounded positive constants whose bounds are independent of $\e.$ By the $L^{\infty}$-estimates in \cite{EGZ} there exists a unique solution $\tilde{\psi}_{\de',\e}$ which is uniformly bounded independent of $\e.$ Similarly, set $\tilde{\xi}_{\de',\e}(t):=\tilde{\phi}_{\de',\e}(t)-\tilde{\psi}_{\de',\e}$ it follows that $\tilde{\xi}(t)$ satisfies the following equation
\begin{equation}\label{eq:nor ma flow-compare}
\left\{
   \begin{array}{ll}
     \displaystyle\dt\tilde{\xi}_{\de',\e} &=\displaystyle\log\frac{((\kappa+\ddb\tilde{\psi}_{\de',\e})+(\tilde{\dr}_{t,\de',\e}-\kappa)+\ddb\tilde{\xi}_{\de',\e})^{n}}
     {(\kappa+\ddb\tilde{\psi}_{\de',\e})^{n}}-\tilde{\xi}_{\de',\e}+\log C_{\de',\e}\\ \;\\
     \tilde{\xi}_{\de',\e}(0)&= \tilde{\phi}_{\de',\e}(0)-\tilde{\psi}_{\de',\e}.
   \end{array}
 \right.
\end{equation}
 by a maximum principle argument as in the proof of Lemma \ref{lem-upper 1} we may conclude that
$$\tilde{\xi}_{\de',\e}(t)\geq e^{-t}(\tilde{\phi}_{\de',\e}(0)-\tilde{\psi}_{\de',\e})
+(1-e^{-t})\log C_{\de',\e}.$$
Combine this with same initial bound \eqref{eq:lower II-1} and it follows that
\begin{align}\label{eq:lower II-app-nor}
\tilde{\varphi}_{\e, l}(t)&\geq -C(\de,\de')+\de'(1-e^{-t})\sum_{i}\log(|S_{i}|_{i}^{2}+\e^{2})\nonumber\\
&+\de e^{-t}\sum_{i,j',k'}(\log(|S_{i}|_{i}^{2}+\e^{2})+\log(|S_{j'}|_{j'}^{2}+\e^{2})+\log(|S_{k'}|_{k'}^{2}+\e^{2}))\nonumber\\
&\geq-C(\de,\de')+\de'(1-e^{-t})\sum_{i}\log|S_{i}|_{i}^{2}+\de e^{-t}\sum_{i,j',k'}(\log|S_{i}|_{i}^{2}+\log|S_{j'}|_{j'}^{2}+\log|S_{k'}|_{k'}^{2}).
\end{align}
Recall that $\tilde{\varphi}_{\e, l}\searrow\tilde{\varphi}$ as $\e\searrow 0$ and $l\to \infty$.  In particular $\tilde{\varphi}$ satisfies the same lower bound estimate above on $X'\setminus \tilde{E}\times[0, \infty)$ and thus extends to be in $L^{\infty}(X'\setminus \tilde{E})$ for each $t$, and is bounded away from $\pi^{-1}X_{lc}$ uniformly for $t\in [0, \infty)$.  This completes the proof of Theorem \ref{thm-main2}.


\section{\ka-Ricci flow through birational surgeries with lc singularities (proof of Theorem \ref{thm-main3})}

In this section we will discuss the behaviour of the \ka-Ricci flow with log canonical singularities when birational surgeries happen. As \cite{ST3}, we will relate the \ka-Ricci flow with lc singularities to the minimal model program with scaling. First we briefly recall some related background materials in MMP with scaling and recommend \cite{BCHM,ST3} for more details. For a $\mathbb{Q}$-factorial projective variety $X$ with log terminal singularities, when $K_{X}$ is not nef, there exist extremal rays generated by algebraic curves which have negative intersection number with $K_{X}$ by cone theorem. Let $H$ be a $\mathbb{Q}$-semi-ample and big divisor, by rationality theorem $\lambda_{0}:=\inf\{\lambda>0|\lambda H+K_{X}\;is\;nef\}$ is a positive rational number. By Kawamata base point free theorem, the divisor $\lambda_{0}H+K_{X}$ is semi-ample and induces a morphism $\pi:=\Phi_{|m(\lambda_{0}H+K_{X})|}: X\to Y$ for sufficiently large $m\in\mathbb{N}$ which contracts all curves $C$ satisfying $(\lambda_{0}H+K_{X})\cdot C=0$ to points. Now considering the image of this contraction morphism, there are several different cases:
\begin{itemize}
 \item If $dim Y<dim X,$ $X$ is called a Mori fiber space where all fibers are Fano varieties.
 \item If $dim Y=dim X,$ i.e., $\pi$ is a birational morphism, depending on the dimension of the exceptional locus in $X,$ there are two different cases:
  \begin{itemize}\item If $dim Exc(\pi)=dim Y-1,$ $\pi$ is a divisorial contraction and we replace $X$ by $Y$ and let $H_{Y}$ be a strict transformation of
  $\lambda_{0}H+K_{X}$ by $\pi.$ Then continue the process to $(Y,H_{Y}).$
  \item If $dim Exc(\pi)<dim Y-1,$ $\pi$ is a small contraction and there exists a flip $X\to X^{+}$ as \eqref{eq:flip}. Let $H_{X^{+}}$ be a strict transformation of $\lambda_{0}H+K_{X}$ and continue the process to $(X^{+},H_{X^{+}}).$
  \end{itemize}
\end{itemize}

By \cite{Bir,Fu1,HX}, this program exists for $\mathbb{Q}$-factorial varieties with log canonical singularities and birational surgeries including divisorial contractions and flips also exist when there exist extremal rays and the induced contraction morphism is birational.

 As \cite{ST3}, we can relate the above to  \ka-Ricci
flow as follows.  Note that in Theorem \ref{thm-main1}, the solution $\omega(t)$ to the \ka-Ricci flow exists up to the rational time $T_{0}:=\sup\{t>0|H+tK_{X}\;is\;nef\}=\frac{1}{\lambda_{0}}$.    Also, $H+T_0 K_{X}$ is semi-ample and thus for a sufficiently large $m\in\mathbb{N}$ there exists a morphism $$\pi:=\Phi_{|m(H+T_{0}K_{X})|}: X\to Y.$$   Our goal here will be to show that there exists a limit current $\omega(T_0)$ which pushes down to $Y$ such that we can continue to flow the push down limit by \ka-Ricci flow on $Y$ using Theorem  \ref{thm-main1}.  For this, we will assume that $H+T_{0}K_{X}$ is in fact big and nef, $\pi$ is birational and induces divisorial contractions or flips as above for the projective variety $X$ with log canonical singularities.   Resolve the singularities of $Y$ and image of $Exc(\pi)$ we have $$\mu:\tilde{X}\to X\to Y$$ which satisfies that
\begin{itemize}
\item $\tilde{X}$ is smooth and $\tilde{\pi}:=\pi^{-1}\circ \mu : \tilde{X}\to X$ is a resolution of singularities of $X$.
\item there exists an effective divisor $E_{Y}$ on $\tilde{X}$ such that $(\pi^{-1}\circ \mu)^{*}[H+T_{0}K_{X}]-\de[E_{Y}]$ is ample for any positive $\de\ll 1$ and $supp(E_{Y})$ coincides with $Exc(\pi^{-1}\circ \mu).$
\end{itemize}

Recall that in proving Theorem \ref{thm-main1} we established the  a solution to the following complex \MA\ flow equation:
\begin{equation}\label{eq:ma flow 12}
\left\{
   \begin{array}{ll}
     \displaystyle\dt\varphi &=\displaystyle\log\frac{(\hat{\dr}_{t}+\ddb\varphi)^{n}}{\Omega}\\
     \varphi(0) &= \varphi_{0},
   \end{array}
 \right.
\end{equation}
on  $X\times[0,T_{0})$ where $\hat{\dr}(t):=\dr_{0}-tRic(\Omega)$ and $\Omega$ is the so-called adapted measure on $X$ (from \cite{ST3}) satisfying
$\pi^{*}\Omega=\frac{\prod_{k}|S_{k}|_{k}^{2a_{k}}}{\prod_{i}|S_{i}|_{i}^{2}\prod_{j}|S_{j}|_{j}^{2b_{j}}}\tilde\Omega$ for the resolution $\mu$.  Moreover, the solution $\varphi(t)$ is smooth on $X_{reg}\times(0,T_{0})$ and gives rise to a current with zero Lelong number.  In particular, $\varphi$ was constructed as the push forward of $\varphi'(t) +\eta\sum_{j}|S_{j}|_{j}^{2(1-b_{j})}-t \log \log^{2}|S_{i}|_{i}^{2}$, under the resolution  $\tilde{\pi}$, where $\varphi'$ was a solution to \eqref{eq:ma flow 2} on $(\tilde{X}\setminus E_Y )\times[0, T_0)$  (see Remark 3.4).

As \cite{ST3}, we have the following estimates:
\begin{lemma}\label{lem-sing time1}
Let $\varphi\in L^{\infty}_{loc}((X\setminus X_{lc}))\times[0,T_{0})\bigcap C^{\infty}(X_{reg}\times(0,T_{0}))$ solve \eqref{eq:ma flow 12} in the sense of Theorem  \ref{thm-main1} as above.  Then
\begin{enumerate}
\item $|\varphi|_{L^{\infty}}(K\times[0,T_{0}))\leq C_{K}$ for any $K\subset\subset(X\setminus\pi^{-1}(Y_{lc}));$
\item $|\varphi|_{C^{k}}(K\times[0,T_{0}))\leq C_{K,k}$ for any $K\subset\subset(X_{reg}\setminus Exc(\pi))$ and $k\in\mathbb{N}.$
\item for sufficiently small $\delta>0$ we have $\varphi' \geq \delta \log \| S_{E_Y} \|^2 +C_\delta$ on $(\tilde{X}\setminus E_Y) \times[0, T_0)$ for some constant $C_{\delta}$.
\end{enumerate}
\end{lemma}
\begin{proof}
Recall that $\tilde{\pi}$ in turn was a constructed as a smooth local limit $$\varphi'_{l,v,\e}\overset{C^{\infty}(X'\setminus\tilde{E})\times   \color{black}{ (0, T_0)}\color{black}{}  }\longrightarrow  \varphi' $$ as $v, \e\to 0$ and $l\to \infty$, where each $\varphi'_{l,v,\e}(t)$ solves the approximate equation \eqref{eq:ma flow 4} on $\tilde{X}\times[0, T_0)$.

 Now the estimates for $\varphi'_{l,v,\e}$ in established \S 3.2 (and the subsequent lower bound in \S 4.4) depended crucially on Assumption \ref{changinghermitianmetricsass} which could only be made on $(\tilde{X}\setminus \tilde{E})\times[0, T']$ for a time $T'<T_0$. Our current setting however, will allow us to replace the time interval $[0, T']$ with $[T_0/2, T_0]$ in Assumption \ref{changinghermitianmetricsass} provided we replace $\tilde{E}$ with the possibly larger set $E_Y$, and as a result the a priori estimates derived in \S 3.2 will in fact hold uniformly as $T'\to T_0$.  This will imply the estimates in the Lemma after letting the parameters $l,v,\e$ pass to their limits then pushing the solution $\varphi'$ forward to $X$.

We now explain why we may make the above mentioned replacements in  Assumption \ref{changinghermitianmetricsass}.
In the current setting we may choose the form $\chi$ in \eqref{eq:bg 1} so that  $\tilde{\pi}^{*}\omega_0+T_0 \chi$ is non-negative on $X'$, and our assumptions on $E_Y$ imply that for $\de$ sufficiently small we have

\begin{equation}\label{finalcd1}
	\tilde{\pi}^{*}\omega_0 +T_0 \chi + \de \ddb \log |S_{E_Y}|^2 \geq c_{\de} \theta 	
\end{equation}
for some $c_{\de}>0$ where $\theta$ is a fixed \ka\, form on $\tilde{X}$, and it follows that for $\de$ sufficiently small we may also have 

\begin{equation}\label{finalcd2}
	\tilde{\pi}^{*}\omega_0 +\frac{T_0}{2} \chi + \de \ddb \log |S_{E_Y}|^2 \geq c_{\de} \theta.	
\end{equation}
Indeed, we can write the LHS of \eqref{finalcd2} as $1/2(	\tilde{\pi}^{*}\omega_0+T_0 \chi) +(1/2)\tilde{\pi}^{*}\omega_0+\de \ddb \log |S_{E_Y}|^2 $
then note that the middle term is non-negative while the first plus last term is positive on $X'$ by \eqref{finalcd2} provided $\de$ is sufficiently small.

 Now using \eqref{finalcd1} and \eqref{finalcd2}, we may use the same reasoning following Assumption \ref{changinghermitianmetricsass} to conclude that the forms in (i), (ii), (iii) of the assumption are \ka\, on $X'\setminus E_Y$ for all $t\in [T_0/2, T_0]$.

\end{proof}
Now we have the following corollary which describes the limit behaviour of the \ka-Ricci flow at the singular time $t=T_{0}:$
\begin{corollary}\label{cor-sing time}
Let $\varphi_{T_{0}}:=\lim\limits_{t\nearrow T_{0}}\varphi(t)$ on $X,$ then the limit current $\dr(T_{0})=\hat{\dr}_{T_{0}}+\ddb\varphi_{T_{0}}$ descends to
a semi-positive current with zero Lelong number on $Y.$
\end{corollary}
\begin{proof}
From Lemma \ref{lem-sing time1}, $\varphi_{T_{0}}\in PSH(X,\hat{\dr_{T_{0}}})\bigcap L^{\infty}_{loc}(X\setminus\pi^{-1}(Y_{lc}))$ is also smooth in $X_{reg}\setminus Exc(\pi)$ and gives rise to a current with zero Lelong number.  Note by definition we have
$ [\hat\omega(T_0)]=[H+T_0 K_X]$.   On the other hand, from the birational morphism $\pi:=\Phi_{|m(H+T_{0}K_{X})|}: X\to Y$ we see that $[H+T_0 K_X]$ also contains the pull back $\pi^{*}\hat{\dr}_{Y}$ where $\hat{\dr}_{Y}$ is the restriction to $Y$ of a multiple of Fubini-Study metric on $\mathbb{CP}^{N_{m}}$.   Thus on any fiber of $\pi$ we must have $[\hat\omega(T_0)]=[\pi^{*}\hat{\dr}_{Y}]=0$ and just as in the proof of  Theorem 1.1, we may conclude  that $ \hat\omega(T_0)+\ddb\varphi(T_0)$ descends
 to Y as in the Corollary while $\varphi_{T_{0}}$ descends to a potential in $L^{\infty}_{loc}(Y\setminus Y_{lc}).$
\end{proof}
\begin{remark}\label{rk-sing time}
We note here that at the singular time $t=T_{0}$ the set where the local potential is $-\infty$ could be larger than $X_{lc}.$ The new generated locus may come
from the contraction $\pi.$ The local potential along the whole fiber of $\pi$ which has nonempty intersection with $X_{lc}$ will be $-\infty.$
\end{remark}
To show that the \ka-Ricci flow could be extended through the divisorial contractions or flips, similar to \cite{ST3}, we only need to check whether the new initial metric satisfies the conditions of Theorem \ref{thm-main1} on the new variety. In particular, we have
\begin{theorem}\label{thm-surgery}
Given a $\mathbb{Q}$-factorial projective variety $X$ with log canonical singularities and a $\mathbb{Q}$-semi-ample divisor $H,$ let $\dr(t)$ be the solution to the \ka-Ricci flow on $X\times[0,T_{0})$ with
$\dr(0)\in[H]$ and $T_{0}:=\sup\{t>0|H+tK_{X}\;is\;nef\}$  as in Theorem \ref{thm-main1}. Suppose $H+T_{0}K_{X}$ is big and semi-ample and induces a birational morphism
$\pi:X\to Y$.  Let $\dr(T_0)$ descend to the semi-positive current $\dr_{Y}$ on $Y$ as in Corollary \ref{cor-sing time}.  Then
\begin{enumerate}
\item if $\pi$ is a divisorial contraction, there exists a solution to the \ka-Ricci flow on $Y$ starting with $\dr_{Y};$
\item if $\pi$ is a small contraction and there exists a flip $\bar{\pi}=\pi^{-1}\circ\pi^+:X^{+}\to X$ defined as \eqref{eq:flip} in Theorem \ref{thm-main3} with the property that $X^{+}_{lc}\bigcap Exc(\pi^{+})=\varnothing,$ there exists a solution $\dr^{+}(t)$ to the \ka-Ricci flow on $X^{+}$ starting with $\pi^{+*}\dr_{Y}$ and $\dr^{+}(t)$ converges to $\pi^{+*}\dr_{Y}$ as $t\searrow T_{0}$ in both current and $C^{\infty}(X^{+}_{reg}\setminus Exc(\pi^{+}))$-senses.
\end{enumerate}
\end{theorem}
\begin{proof}
In \cite{ST3}, to continue the \ka-Ricci flow after birational surgeries, they need to verify the $L^{p}$-integrable condition of the new measure for $p>1.$
However, by Theorem \ref{thm-main1} we only need to verify that the new initial metric is a current with zero Lelong number and local potential in $L^{\infty}_{loc}$ away from the log canonical locus. For divisorial contraction case, by Lemma \ref{lem-sing time1} and Corollary \ref{cor-sing time} those
conditions are satisfied automatically on $Y$ so the \ka-Ricci flow can be continued on $Y$ directly. For flip case, the assumption $X^{+}_{lc}\bigcap Exc(\pi^{+})=\varnothing$ guarantees that $\pi^{+*}\dr_{Y}$ satisfies the initial condition of Theorem \ref{thm-main1} and all the conclusion follows from
Theorem \ref{thm-main1} and continuity properties at the initial time in section 4.
\end{proof}
\begin{remark}\label{rk-flip}
The assumption $X^{+}_{lc}\bigcap Exc(\pi^{+})=\varnothing$ could be dropped. Actually when the flip exists for a $\mathbb{Q}$-factorial variety with log canonical singularities, the log canonical loci of $X,Y,X^{+}$ are essentially isomorphic. We thank Professor Chenyang Xu for pointing out this property.
\end{remark}
As \cite{ST3}, we can also find a good initial semi-ample divisor $H$ such that at each singular time the induced contraction only contracts exact one
extremal ray and performs corresponding birational surgery. This process will finally terminate in finite steps when we arrive at a minimal model or a Mori fiber space with log canonical singularities.

\section{Further Discussions}

This paper is only a starting point of studying MMP with log canonical singularities by the \ka-Ricci flow. We will briefly discuss some further problems here.

First, instead of a variety, actually the pairs with the form $(X,D)$ are more frequently studied in MMP, where $X$ is the variety we studied in this paper and
$D$ is an effective simple normal crossing divisor on $X$ with the form that $D=\sum_{i}a_{i}D_{i}$ where $D_{i}$ are irreducible and $a_{i}\in(0,1]$ are rational numbers. In this case as the twisted canonical line bundle is $[K_{X}+D],$ we could design a twisted \ka-Ricci flow with conical and cusp singularities:
\begin{equation}\label{eq:twisted krf}
\dt\dr(t)=-Ric(\dr(t))+\sum_{i}2\pi a_{i}[D_{i}].
\end{equation}
Such types of twisted \ka-Ricci flow in manifold case have been studied in \cite{CLS,LZ,LiuZ,Shen}.
By the log resolution $\pi:X'\to X$ of the pair $(X,D),$ we have the general adjunction formula
\begin{equation}\label{eq:adjunction 3}
K_{X'}=\pi^{*}(K_{X}+\sum_{i}a_{i}D_{i})+\sum_{j}b_{j}E_{j},
\end{equation}
where some $E_{j}$ are strict transformations of $D_{i}$ and some are exceptional divisors. Combine the techniques in Theorem \ref{thm-main1} and \cite{CLS,Shen} we could show that the twisted \ka-Ricci flow exists whenever the corresponding cohomology class of the evolving metric is nef.
Moreover we could show that along $D_{i}\bigcap X_{reg}$ the evolving metric simultaneously has conical singularities with angle $2\pi(1-a_{i})$ when $a_{i}\in(0,1)$ and cusp singularities when $a_{i}=1.$

In section 5 we showed the convergence of the \ka-Ricci flow on semi-log canonical models. One problem is that what is the long time behaviour of the \ka-Ricci flow on general minimal varieties with log canonical singularities. As \cite{BBEGZ} showed that the Kodaira dimension of such varieties cannot be $-\infty,$ then what are the long time behaviours for nonnegative Kodaira dimensions. For smooth minimal manifold case this problem has been studied in \cite{ST1,ST2} and for the varieties with log terminal singularities this has been studied in \cite{EGZ1,ST3,SY2}. Moreover, the geometric convergence problem with singularities is still quite challenging, see the last section of \cite{ST3} for discussions in smooth case.


\begin{thebibliography}{99}
\bibitem{Au}\textit{T.~Aubin}, Equation de type \MA\ sur les vari\'{e}t\'{e}s k\"{a}hl\'{e}riennes compactes, Bull. Sci. Math. \textbf{102}
(1978), 63--95.

\bibitem{BBEGZ}\textit{R.~Berman}, \textit{S.~Boucksom}, \textit{P.~Eyssidieux}, \textit{V.~Guedj}, and \textit{A.~Zeriahi}, \ka-Einstein metrics
and the \ka-Ricci flow on log Fano varieties, J. Reine Angew. Math. (2016).
\bibitem{BBGZ}\textit{R.~Berman}, \textit{S.~Boucksom}, \textit{V.~Guedj}, and \textit{A.~Zeriahi}, A variational approach to complex \MA\ equations,
Publ. Math. de IH\'{E}S. \textbf{117}(2013), 179--245.
\bibitem{BCHM}\textit{C.~Birkar}, \textit{P.~Cascini}, \textit{C.~Hacon} and \textit{J.~McKernan}, Existence of minimal models for varieties of log general type, J. Amer. Math. Soc. \textbf{23}(2010), no. 2, 405--468.
\bibitem{Bir}\textit{C.~Birkar}, Existence of log canonical flips and a special LMMP, Publ. Math. de IH\'{E}S. \textbf{115}(2012), 325--368.
\bibitem{BG}\textit{R.~Berman} and \textit{H.~Guenancia}, \ka-Einstein metrics on stable varieties and log canonical pairs, Geom. Funct. Anal. \textbf{24}(2016), no. 6, 1683--1730.
\bibitem{Blo}\textit{Z.~Blocki} and \textit{S.~Kolodziej}, On regularization of plurisubharmonic functions on manifolds, Proc. Amer. Math. Soc. 135(7) (2007): 2089-2093.


\bibitem{Cao}\textit{H.~D.~Cao}, Deformation of \ka\ metrics to \ka-Einstein metrics on compact \ka\ manifolds, Invent. Math. \textbf{81}(1985), no. 2, 359--372.
\bibitem{CG}\textit{J.~Carlson} and \textit{P.~Griffiths}, A defect relation for equidimensional holomorphic mappings between algebraic varieties,
Ann. Math. \textbf{95}(1972), 557--584.
\bibitem{Chau}\textit{A.~Chau}, Convergence of the \ka-Ricci flow on noncompact \ka\ manifolds, J. Differential Geom. \textbf{66}(2004), 211--232.
\bibitem{CLS}\textit{A.~Chau}, \textit{K.~F.~Li} and \textit{L.~M.~Shen}, \ka-Ricci flow of cusp singularities on quasi projective varieties, Adv. Math. \textbf{339}(2018), 310--335.
\bibitem{De}\textit{J.~P.~Demailly}, Complex Analytic and Differential Geometry,
http://www-fourier. ujf-grenoble.fr/demailly/books.html, 1997.

\bibitem{D}\textit{Q-T Dang}, Pluripotential Monge-Amp\'{e}re  flows in big cohomology classes, preprint arXiv:2102.05189
\bibitem{EGZ}\textit{P.~Eyssidieux}, \textit{V.~Guedj}, and \textit{A.~Zeriahi}, Singular \ka-Einstein metrics, J. Amer. Math. Soc. \textbf{22}(2009), no. 3, 607--639.
\bibitem{EGZ1}\textit{P.~Eyssidieux}, \textit{V.~Guedj}, and \textit{A.~Zeriahi}, Convergence of weak \ka-Ricci flow on minimal models, Comm. Math. Phys. \textbf{357}(2018), 1179--1214.
\bibitem{Fu1}\textit{O.~Fujino}, Introduction to the log minimal model program for log canonical pairs, preprint, 2008.
\bibitem{Fu2}\textit{O.~Fujino}, Some remarks on the minimal model program for log canonical pairs, J. Math. Sci. Univ. Tokyo \textbf{22}(2015), no. 1, 149--192.
\bibitem{GZ}\textit{V.~Guedj} and \textit{A.~Zeriahi}, Regularizing properties of the twisted \ka-Ricci flow,
J. Reine Angew. Math.\textbf{729}(2017), 275--304.
\bibitem{Guenancia}\textit{H.~Guenancia}, K\"ahler-Einstein metrics with mixed Poincar\'e and cone singularities along a normal crossing divisor, Ann. Inst. Fourier 64 (3), 1291--1330 (2014)
\bibitem{GP}\textit{H.~Guenancia} and \textit{M.~P\u{a}un}, Conic singularities metrics with prescribed Ricci curvature: General cone angles along normal crossing divisors, J. Differential Geom. \textbf{103} (2016), no. 1, 15--57.
\bibitem{GLZ}\textit{V. Guedj, C.H. Lu} and \textit{A. Zeriahi},  Pluripotential \ka-Ricci flows, Geom. Topol. \textbf{24} (2020), no. 3,
1225-1296.
\bibitem{HX}\textit{C.~Hacon} and \textit{C.~Y.~Xu}, Existence of log canonical closures, Invent. math. \textbf{192}(2013), 161--195.
\bibitem{KMM}\textit{Y.~Kawamata}, \textit{K.~Matsuda} and \textit{K.~Matsuki}, Introduction to the minimal model problem,
in: Oda,T. (ed.) Algebraic Geometry (Sendai, 1985), Adv. Stud. Pure Math., Amsterdam (1987), no. 10, 283--360.
\bibitem{Koba}\textit{R.~Kobayashi}, \ka-Einstein metric on an open algebraic manifolds, Osaka J. Math. \textbf{21}(1984), 399--418.
\bibitem{KM}\textit{J.~Koll\'{a}r}, \textit{S.~Mori}, Birational geometry of algebraic varieties, Cambridge Tracts in Mathematics,
vol. 134, Cambridge University Press, 1998.
\bibitem{Ko}\textit{S.~Kolodziej}, The complex \MA\ equation, Acta Math. \textbf{180} (1998),69--117.
\bibitem{LiuZ}\textit{J.~Liu} and \textit{X.~Zhang}, Conical \ka-Ricci flows on Fano manifolds, Adv. Math. \textbf{307} (2017), 1324--1371.
\bibitem{LZ}\textit{J.~Lott} and \textit{Z.~Zhang}, Ricci flow on quasi-projective manifolds, Duke Math. J. \textbf{156}(2011), no. 1, 87--123.
\bibitem{Pe}\textit{G.~Perelman}, The entropy formula for the Ricci flow and its geometric applications, arXiv:math.DG/0211159v1
\bibitem{Shen}\textit{L.~M.~Shen},Maximal time existence of unnormalized conical \ka-Ricci flow, J. Reine Angew. Math.(2018).
\bibitem{Sh1}\textit{W.~X.~Shi}, Ricci deformation of the metric on complete Riemannian manifolds, J. Differential Geom. \textbf{30}(1989), 303--394.
\bibitem{So13}\textit{J.~Song}, Ricci flow and birational surgery, arXiv:1304.2607.
\bibitem{So17}\textit{J.~Song}, Degeneration of \ka-Einstein manifolds of negative scalar curvature, arXiv:1706.01518.
\bibitem{ST1}\textit{J.~Song} and \textit{G.~Tian}, The \ka-Ricci flow on surfaces of positive Kodaira dimension, Invent. math.\textbf{170}(2007), no. 3, 609--653.
\bibitem{ST2}\textit{J.~Song} and \textit{G.~Tian}, Canonical measures and \ka-Ricci flow, J. Amer. Math. Soc. \textbf{25}(2012), 303--353.
\bibitem{ST3}\textit{J.~Song} and \textit{G.~Tian}, The \ka-Ricci flow through singularities, Invent. math. \textbf{207}(2017), no. 2, 519--595.
\bibitem{SW1}\textit{J.~Song} and \textit{B.~Weinkove}, Contracting exceptional divisors by the \ka-Ricci flow, Duke Math. J. \textbf{162}(2013), no. 2, 367--415.
\bibitem{SW2}\textit{J.~Song} and \textit{B.~Weinkove}, Contracting exceptional divisors by the Kahler-Ricci flow II, Proc. Lond. Math. Soc. (3)\textbf{108} (2014), no. 6, 1529--1561.
\bibitem{SY1}\textit{J.~Song} and \textit{Y.~Yuan}, Metric flips with Calabi ansatz, Geom. Func. Anal. \textbf{22} (2012), no. 1, 240--265.
\bibitem{SY2}\textit{J.~Song} and \textit{Y.~Yuan}, Convergence of the \ka-Ricci flow on singular Calabi-Yau varietie.
In: Ji, L. et al. (eds.) Advances in Geometric Analysis, Adv. Lect. Math. (ALM), vol. 21, 119--137. Int. Press, Somerville, 2012.
\bibitem{TY}\textit{G.~Tian} and \textit{S.~T.~Yau}, Existence of \ka-Einstein metrics on complete \ka\ manifolds and their applications to algebraic geometry,
Mathematical Aspects of String Theory (San Diego, Calif., 1986), Adv. Ser. Math. Phys. 1, World Sci., Singapore (1987), 574--628.
\bibitem{TZL}\textit{G.~Tian} and \textit{Z.~L.~Zhang}, Relative volume comparison of Ricci Flow and its applications, arXiv:1802.09506.
\bibitem{TZ}\textit{G.~Tian} and \textit{Z.~Zhang}, On the \ka-Ricci flow on projective manifolds of general type, Chinese Annals of Mathematics - Series B, \textbf{27}(2006), no. 2, 179--192.
\bibitem{Ts}\textit{H.~Tsuji}, Existence and degeneration of Kaehler-Einstein metrics on minimal algebraic varieties of general type, Math. Ann. \textbf{281}(1988), no. 1, 123--133.
\bibitem{Yau1}\textit{S.~T.~Yau}, On the Ricci curvature of a compact \ka\ manifold and the complex Monge-Amp\'{e}re equation. I, Comm. Pure Appl. Math. \textbf{31}(1978), no. 3, 339--411.
\bibitem{Yau2}\textit{S.~T.~Yau}, A general Schwarz lemma for \ka\ manifolds, Amer. J. Math. \textbf{100}(1978), 197--203.
\end{thebibliography}
\end{document}